%% file: main.tex
\documentclass{article}
\usepackage[utf8]{inputenc}
\input{preamble}

\title{Unified inverse correspondence for DLE-logics}
\author{Willem Conradie\textsuperscript{1}, Andrea De Domenico\textsuperscript{2}, Giuseppe Greco\textsuperscript{2}, \\ Alessandra Palmigiano\textsuperscript{2}, Mattia Panettiere\textsuperscript{2} and Apostolos Tzimoulis\textsuperscript{2}}
\date{\textsuperscript{1}University of the Witwatersrand, Johannesburg, \\
\textsuperscript{2}Vrije Universiteit, Amsterdam}

\begin{document}
\maketitle

\begin{abstract}
By exploiting the algebraic and order theoretic mechanisms behind Sahlqvist correspondence, the theory of unified correspondence provides powerful tools for correspondence and canonicity across different semantics and signatures, covering all the logics whose algebraic semantics are given by normal (distributive) lattice expansions (referred to as (D)LEs). 
In particular, the algorithm ALBA, parametric in each (D)LE, effectively computes the first order correspondents of (D)LE-inductive formulas.
We present an algorithm that makes use of ALBA's rules and algebraic language to invert its steps in the DLE setting; therefore effectively computing an inductive formula starting from its first order correspondent.

\end{abstract}

\tableofcontents

\section{Introduction}
\input{01introduction}

\section{Preliminaries}
\label{sec:preliminaries}
\input{02preliminaries}

\section{Crypto inductive inequalities}
\label{sec:cryptoinductive}
\input{03cryptoinductive}

\section{Inverse ALBA}
\label{sec:inversealba}
\input{04inverse}

\section{Conclusions}
\label{sec:conclusions}
\input{05conclusions}

\bibliographystyle{abbrv}
\bibliography{ref}

\appendix
\section{ALBA for correspondence}
\label{appendix:albarules}
\input{appendix01albacorrespondence}

\end{document}

%% file: preamble.tex
\usepackage{graphicx}
\usepackage{hyperref}
\usepackage{amsmath}
\usepackage{amssymb}
\usepackage{marginnote}
\usepackage{txfonts}
\usepackage{cmll}
\usepackage{mathtools}
\usepackage{algpseudocode}
\usepackage{algorithm}
\usepackage{xcolor}
\usepackage{bussproofs}
\usepackage{proof}
\usepackage[margin=3.5cm]{geometry}

\usepackage{amsthm}
\renewcommand{\epsilon}{\varepsilon}

\newcommand{\metaand}{\ \with \ }
\newcommand{\bigmetaor}{\bigparr}

\newcommand{\bigmetaand}{\bigwith}

\newcommand{\metaor}{\ \parr\ }
\newcommand{\metanot}{{\sim}}
\newcommand\val[1]{{\lbrack\!\lbrack} {#1}{\rbrack\!\rbrack}}
\newcommand{\atprop}{\mathsf{AtProp}}
\newcommand{\nomset}{\mathrm{NL}}
\newcommand{\cnomset}{\mathrm{CNL}}

\newcommand{\nomh}{\mathbf{h}}
\newcommand{\nomi}{\mathbf{i}}
\newcommand{\nomj}{\mathbf{j}}
\newcommand{\nomk}{\mathbf{k}}
\newcommand{\cnoml}{\mathbf{l}}
\newcommand{\cnomm}{\mathbf{m}}
\newcommand{\cnomn}{\mathbf{n}}
\newcommand{\cnomo}{\mathbf{o}}
\newcommand{\pureu}{\mathbf{u}}
\newcommand{\purev}{\mathbf{v}}
\newcommand{\purew}{\mathbf{w}}
\newcommand{\jty}{J^{\infty}}
\newcommand{\mty}{M^{\infty}}
\newcommand{\pdla}{>\!\!\!-\:}
\newcommand{\pdra}{-\!\!\!<\:}
\newcommand{\blhd}{\blacktriangleleft}
\newcommand{\brhd}{\blacktriangleright}

\newcommand{\dneg}{{\rho}}
\newcommand{\langbase}{\mathcal{L}_{\mathrm{DLE}}}
\newcommand{\langresidual}{\langbase^*}
\newcommand{\langalba}{\langbase^+} 
\newcommand{\langineq}{\langbase^{\leq}}
\newcommand{\langmeta}{\langbase^{\mathrm{meta}}}

\newcommand{\bba}{\mathbb{A}}
\newcommand{\bbA}{\mathbb{A}}
\newcommand{\oz}{\overline{z}}

\theoremstyle{plain}
\newtheorem{thm}{Theorem}
\newtheorem{theorem}{Theorem}

\newtheorem{cor}[thm]{Corollary}
\newtheorem{prop}[thm]{Proposition}

\newtheorem{lemma}[thm]{Lemma}

\theoremstyle{definition}

\newtheorem{definition}[thm]{Definition}
\newtheorem{notation}[thm]{Notation}
\newtheorem{example}[thm]{Example}

\newtheorem{remark}[thm]{Remark}

%% file: 01introduction.tex
Driven by the ``insight that almost all completeness proofs can be reinterpreted
as definability results [...] and that also correspondence
theory is a kind of definability theory'', Kracht \cite{krachtphdthesis} developed the theory of internal description, sometimes referred to as  \emph{inverse correspondence} \cite{blackburn2002modal}. This theory can be regarded as  converse  to Sahlqvist correspondence \cite{sahlqvist1975completeness}; indeed, it syntactically identifies a class of first order formulas, each of which is the first order correspondent of some modal formula, and provides an effective procedure for computing such modal formula. 

Goranko and Vakarelov extended Sahlqvist theory to the class of \emph{polyadic Sahlqvist formulas} \cite{GoVa2000}, also referred to  as \emph{inductive formulae} \cite{GoVa2006}. In \cite{kikotinv}, Kikot  extends Kracht's result to inductive formulae, by syntactically characterizing a class of formulas in the first order language of Kripke frames for classical normal modal logic which correspond to 
inductive formulas in classical modal logic.

During the last decade, a  line of research was developed which  focuses on the order-theoretic underpinning of Sahlqvist  theory, thus allowing for the generalisations of this theory from classical modal logic to wide classes of nonclassical logics. This shift from a model-theoretic to an algebraic perspective made it  possible to uniformly define the class of Sahlqvist and inductive formulas/inequalities for a broad spectrum of logical languages, based on the order-theoretic properties of the algebraic interpretations of the logical connectives in each language, and to extend the algorithm SQEMA, for computing the first order correspondents of  inductive formulas of  classical modal logic \cite{CoGoVa2006}, to the algorithm ALBA \cite{conradie2012algorithmic,copa2016nondistr}, performing the same task as SQEMA for this spectrum of nonclassical languages which includes the {\em LE-logics}, i.e.~those logics the algebraic semantics of which is given by varieties of normal/regular lattice expansions (LEs), and their expansions with fixed points \cite{CoGhPa14,CCPZ} \cite{CoGhPa14}. 
Not only this very high level of generality allows to extend the benefits of correspondence and canonicity results to many well known logical systems such as bi-intuitionistic (modal) logic, the Lambek-Grishin calculus \cite{lambekgrishin}, and the multiplicative-additive fragment of linear logic \cite{bookreslattice}, but this new point of view paves the way to several  developments and connections among the meta-theories of several logical frameworks, examples of which are a general perspective on G\"odel-McKinsey-Tarski translations and correspondence/canonicity transfer results \cite{sahlqvistviatranslation,slanted}, systematic connections among   different relational semantics of a given logic \cite{dairapaper}, and systematic connections between correspondence-theoretic results and the proof-theoretic behaviour of  logical frameworks \cite{greco2018unified,greco2016linear,chen2021syntactic,greco2016lattice,greco2018semi}.

While many generalizations of Sahlqvist correspondence theory have been developed in recent times, no generalizations of Kracht's theory of inverse correspondence  have been investigated yet since Kikot's result. The results in the present paper start to fill this gap, by generalizing  Kikot's result from classical normal modal logic to all normal {\em DLE-logics}, i.e.~those logics the algebraic semantics of which is given by varieties of normal distributive lattice expansions (DLEs). In particular, we introduce an {\em inverse correspondence} algorithm targeting inductive inequalities in any DLE-signature. 

Key to this extension is the possibility to reformulate the main engine of Kracht's result in the algebraic environment of unified correspondence \cite{CoGhPa14} so as to exploit the language and algorithmic tools developed there, which work across signatures and relational semantics. Indeed, to achieve this objective, we approach the problem from an exclusively order theoretic perspective by making use of a slight extension of ALBA's language and rules, extending the work started in \cite{inverse_correspondence_tense} to the DLE setting.

The proof-strategy adopted to achieve this result is different from Kikot's. Indeed, rather than  relaxing the definition of Kracht's formula, which is given only in terms of forward-looking restricted quantifiers, we start by generalizing to the setting of DLE-logics the fact, well-known from  classical modal logic, that inductive formulas are semantically equivalent to  (a certain proper subclass) of scattered \emph{very simple Sahlqvist} formulas in the language of {\em tense logic}. Accordingly, for every DLE-language $\mathcal{L}$, we syntactically characterize the class $K$ of very simple Sahlqvist $\mathcal{L}^\ast$-inequalities (where  $\mathcal{L}^\ast$ is the language expansion of $\mathcal{L}$ obtained by closing the signature of  $\mathcal{L}$ under the residuals of each connective in $\mathcal{L}$) which are semantically equivalent to inductive $\mathcal{L}$-inequalities. Then, we syntactically characterize the class of formulas in the ALBA-language, referred to as {\em Kracht's formulas} (which can be readily translated into first-order formulas of a given frame correspondence language) which target the subclass $K$,
by allowing for the use of \emph{backward-looking} restricted quantifiers. 
Finally, we  show that each Kracht's formula in the ALBA-language can be effectively and equivalently transformed into  the ALBA-output of an $\mathcal{L}^\ast$-inequality in $K$.  

\paragraph{Structure of the paper. } In Section \ref{sec:preliminaries}, we present a brief introduction to DLE-logics, inductive formulae, ALBA, and we prove some useful technical lemmas essential to the rest of the paper in subsection \ref{ssec:fapproxrestrackermann}. In Section \ref{sec:cryptoinductive}, we show a class of very simple $\mathcal{L}^*$-Sahlqvist formulae equivalent to the class of $\mathcal{L}$-inductive formulae, given some DLE-language $\mathcal{L}$. Section \ref{sec:inversealba} contains the core result of the paper. Firstly, a class of first order correspondents is presented, and then it is proved that it contains all the correspondents of inductive formulas. Finally the algorithm for inverse correspondence is shown and its correctness is proved. Section \ref{sec:conclusions} suggests avenues for future research.

%% file: 02preliminaries.tex
In this section we present the languages that we consider and their minimal logics, recall the rules of ALBA, present some variations of the Ackermann lemma, and discuss the problem that arise in generalising Kracht's work to a more general setting.

\subsection{Language and axioms}
Our base language is an unspecified but fixed language $\mathcal{L}_\mathrm{DLE}$, to be interpreted over distributive lattice expansions of compatible similarity type.
We will make heavy use of the following auxiliary definition: an {\em order-type} over $n\in \mathbb{N}$\footnote{Throughout the paper, order-types will be typically associated with arrays of variables $\overline p: = (p_1,\ldots, p_n)$. When the order of the variables in $\overline p$ is not specified, we will sometimes abuse notation and write $\varepsilon(p) = 1$ or $\varepsilon(p) = \partial$.} is an $n$-tuple $\epsilon\in \{1, \partial\}^n$. For every order type $\epsilon$, we denote its {\em opposite} order type by $\epsilon^\partial$, that is, $\epsilon^\partial_i = 1$ iff $\epsilon_i=\partial$ for every $1 \leq i \leq n$. For any lattice $\bba$, we let $\bba^1: = \bba$ and $\bba^\partial$ be the dual lattice, that is, the lattice associated with the converse partial order of $\bba$. For any order type $\varepsilon$, we let $\bba^\varepsilon: = \Pi_{i = 1}^n \bba^{\varepsilon_i}$.
	
	The language $\mathcal{L}_\mathrm{DLE}(\mathcal{F}, \mathcal{G})$ (from now on abbreviated as $\mathcal{L}_\mathrm{DLE}$) takes as parameters: 1) a denumerable set $\mathsf{PROP}$ of proposition letters, elements of which are denoted $p,q,r$, possibly with indexes; 2) disjoint sets of connectives $\mathcal{F}$ and  $\mathcal{G}$. 
Each $f\in \mathcal{F}$ and $g\in \mathcal{G}$ has arity $n_f\in \mathbb{N}$ (resp.\ $n_g\in \mathbb{N}$) and is associated with some order-type $\varepsilon_f$ over $n_f$ (resp.\ $\varepsilon_g$ over $n_g$).\footnote{Unary $f$ (resp.\ $g$) will be sometimes denoted as $\Diamond$ (resp.\ $\Box$) if the order-type is 1, and $\lhd$ (resp.\ $\rhd$) if the order-type is $\partial$.} 
The terms (formulas) of $\mathcal{L}_\mathrm{DLE}$ are defined recursively as follows:
	\[
	\phi ::= p \mid \bot \mid \top \mid \phi \wedge \phi \mid \phi \vee \phi \mid f(\overline{\phi}) \mid g(\overline{\phi})
	\]
	where $p \in \mathsf{PROP}$, $f \in \mathcal{F}$, $g \in \mathcal{G}$. Terms in $\mathcal{L}_\mathrm{DLE}$ will be denoted either by $s,t$, or by lowercase Greek letters such as $\varphi, \psi, \gamma$ etc. 

\begin{definition}
		\label{def:DLE:logic:general}
		For any language $\mathcal{L}_\mathrm{DLE} = \mathcal{L}_\mathrm{DLE}(\mathcal{F}, \mathcal{G})$, the {\em basic}, or {\em minimal} $\mathcal{L}_\mathrm{DLE}$-{\em logic} is a set of sequents $\phi\vdash\psi$, with $\phi,\psi\in\mathcal{L}_\mathrm{LE}$, which contains the following axioms:
		\begin{itemize}
			\item Sequents for  lattice operations:
			\begin{align*}
				&p\vdash p, && \bot\vdash p, && p\vdash \top, & & p\wedge (q\vee r)\vdash (p\wedge q)\vee (p\wedge r), &\\
				&p\vdash p\vee q, && q\vdash p\vee q, && p\wedge q\vdash p, && p\wedge q\vdash q, &
			\end{align*}
			\item Sequents for $f\in \mathcal{F}$ and $g\in \mathcal{G}$:
			\begin{align*}
				& f(p_1,\ldots, \bot,\ldots,p_{n_f}) \vdash \bot,~\mathrm{for}~ \varepsilon_f(i) = 1,\\
				& f(p_1,\ldots, \top,\ldots,p_{n_f}) \vdash \bot,~\mathrm{for}~ \varepsilon_f(i) = \partial,\\
				&\top\vdash g(p_1,\ldots, \top,\ldots,p_{n_g}),~\mathrm{for}~ \varepsilon_g(i) = 1,\\
				&\top\vdash g(p_1,\ldots, \bot,\ldots,p_{n_g}),~\mathrm{for}~ \varepsilon_g(i) = \partial,\\
				&f(p_1,\ldots, p\vee q,\ldots,p_{n_f}) \vdash f(p_1,\ldots, p,\ldots,p_{n_f})\vee f(p_1,\ldots, q,\ldots,p_{n_f}),~\mathrm{for}~ \varepsilon_f(i) = 1,\\
				&f(p_1,\ldots, p\wedge q,\ldots,p_{n_f}) \vdash f(p_1,\ldots, p,\ldots,p_{n_f})\vee f(p_1,\ldots, q,\ldots,p_{n_f}),~\mathrm{for}~ \varepsilon_f(i) = \partial,\\
				& g(p_1,\ldots, p,\ldots,p_{n_g})\wedge g(p_1,\ldots, q,\ldots,p_{n_g})\vdash g(p_1,\ldots, p\wedge q,\ldots,p_{n_g}),~\mathrm{for}~ \varepsilon_g(i) = 1,\\
				& g(p_1,\ldots, p,\ldots,p_{n_g})\wedge g(p_1,\ldots, q,\ldots,p_{n_g})\vdash g(p_1,\ldots, p\vee q,\ldots,p_{n_g}),~\mathrm{for}~ \varepsilon_g(i) = \partial,
			\end{align*}
		\end{itemize}
		and is closed under the following inference rules:
		\begin{displaymath}
			\frac{\phi\vdash \chi\quad \chi\vdash \psi}{\phi\vdash \psi}
			\quad
			\frac{\phi\vdash \psi}{\phi(\chi/p)\vdash\psi(\chi/p)}
			\quad
			\frac{\chi\vdash\phi\quad \chi\vdash\psi}{\chi\vdash \phi\wedge\psi}
			\quad
			\frac{\phi\vdash\chi\quad \psi\vdash\chi}{\phi\vee\psi\vdash\chi}
		\end{displaymath}
		\begin{displaymath}
			 \frac{\phi\vdash\psi}{f(p_1,\ldots,\phi,\ldots,p_n)\vdash f(p_1,\ldots,\psi,\ldots,p_n)}{~(\varepsilon_f(i) = 1)}
		\end{displaymath}
		\begin{displaymath}
			 \frac{\phi\vdash\psi}{f(p_1,\ldots,\psi,\ldots,p_n)\vdash f(p_1,\ldots,\phi,\ldots,p_n)}{~(\varepsilon_f(i) = \partial)}
		\end{displaymath}
		\begin{displaymath}
			 \frac{\phi\vdash\psi}{g(p_1,\ldots,\phi,\ldots,p_n)\vdash g(p_1,\ldots,\psi,\ldots,p_n)}{~(\varepsilon_g(i) = 1)}
		\end{displaymath}
		\begin{displaymath}
			 \frac{\phi\vdash\psi}{g(p_1,\ldots,\psi,\ldots,p_n)\vdash g(p_1,\ldots,\phi,\ldots,p_n)}{~(\varepsilon_g(i) = \partial)}.
		\end{displaymath}
		The minimal DLE-logic is denoted by $\mathbf{L}_\mathrm{DLE}$. For any DLE-language $\mathcal{L}_{\mathrm{DLE}}$, by an {\em $\mathrm{DLE}$-logic} we understand any axiomatic extension of the basic $\mathcal{L}_{\mathrm{DLE}}$-logic in $\mathcal{L}_{\mathrm{DLE}}$.
	\end{definition}

	\subsection{The expanded language \texorpdfstring{ $\mathcal{L}_\mathrm{DLE}^*$}{LDLE*}}
	\label{ssec:expanded language}
	Any given language $\mathcal{L}_\mathrm{DLE} = \mathcal{L}_\mathrm{DLE}(\mathcal{F}, \mathcal{G})$ can be associated with the language $\mathcal{L}_\mathrm{DLE}^* = \mathcal{L}_\mathrm{DLE}(\mathcal{F}^*, \mathcal{G}^*)$, where $\mathcal{F}^*\supseteq \mathcal{F}$ and $\mathcal{G}^*\supseteq \mathcal{G}$ are obtained by expanding $\mathcal{L}_\mathrm{DLE}$ with the following connectives:
	\begin{enumerate}
		\item the binary connectives $\leftarrow$ and $\rightarrow$, the intended interpretations of which are the right residuals of $\wedge$ in the first and second coordinate respectively, and
		$\pdla$ and $ \pdra$, the intended interpretations of which are the left residuals of $\vee$ in the first and second coordinate, respectively;
		\item the $n_f$-ary connective $f^\sharp_i$ for $0\leq i\leq n_f$, the intended interpretation of which is the right residual of $f\in\mathcal{F}$ in its $i$th coordinate if $\varepsilon_f(i) = 1$ (resp.\ its Galois-adjoint if $\varepsilon_f(i) = \partial$);
		\item the $n_g$-ary connective $g^\flat_i$ for $0\leq i\leq n_g$, the intended interpretation of which is the left residual of $g\in\mathcal{G}$ in its $i$th coordinate if $\varepsilon_g(i) = 1$ (resp.\ its Galois-adjoint if $\varepsilon_g(i) = \partial$).
		\footnote{The adjoints of the unary connectives $\Box$, $\Diamond$, $\lhd$ and $\rhd$ are denoted $\Diamondblack$, $\blacksquare$, $\blhd$ and $\brhd$, respectively.}
	\end{enumerate}
	We stipulate that $\pdra, \pdla\in \mathcal{F}^*$, that $\rightarrow, \leftarrow\in \mathcal{G}^*$, and moreover, that
	$f^\sharp_i\in\mathcal{G}^*$ if $\varepsilon_f(i) = 1$, and $f^\sharp_i\in\mathcal{F}^*$ if $\varepsilon_f(i) = \partial$. Dually, $g^\flat_i\in\mathcal{F}^*$ if $\varepsilon_g(i) = 1$, and $g^\flat_i\in\mathcal{G}^*$ if $\varepsilon_g(i) = \partial$. The order-type assigned to the additional connectives is predicated on the order-type of their intended interpretations. That is, for any $f\in \mathcal{F}$ and $g\in\mathcal{G}$,
	\begin{enumerate}
		\item if $\epsilon_f(i) = 1$, then $\epsilon_{f_i^\sharp}(i) = 1$ and $\epsilon_{f_i^\sharp}(j) = (\epsilon_f(j))^\partial$ for any $j\neq i$.
		\item if $\epsilon_f(i) = \partial$, then $\epsilon_{f_i^\sharp}(i) = \partial$ and $\epsilon_{f_i^\sharp}(j) = \epsilon_f(j)$ for any $j\neq i$.
		\item if $\epsilon_g(i) = 1$, then $\epsilon_{g_i^\flat}(i) = 1$ and $\epsilon_{g_i^\flat}(j) = (\epsilon_g(j))^\partial$ for any $j\neq i$.
		\item if $\epsilon_g(i) = \partial$, then $\epsilon_{g_i^\flat}(i) = \partial$ and $\epsilon_{g_i^\flat}(j) = \epsilon_g(j)$ for any $j\neq i$.
	\end{enumerate}
	
	For instance, if $f$ and $g$ are binary connectives such that $\varepsilon_f = (1, \partial)$ and $\varepsilon_g = (\partial, 1)$, then $\varepsilon_{f^\sharp_1} = (1, 1)$, $\varepsilon_{f^\sharp_2} = (1, \partial)$, $\varepsilon_{g^\flat_1} = (\partial, 1)$ and $\varepsilon_{g^\flat_2} = (1, 1)$.\footnote{Warning: notice that this notation heavily depends from the connective which is taken as primitive, and needs to be carefully adapted to well known cases. For instance, consider the  `fusion' connective $\circ$ (which, when denoted  as $f$, is such that $\varepsilon_f = (1, 1)$). Its residuals
$f_1^\sharp$ and $f_2^\sharp$ are commonly denoted $/$ and
$\backslash$ respectively. However, if $\backslash$ is taken as the primitive connective $g$, then $g_2^\flat$ is $\circ = f$, and
$g_1^\flat(x_1, x_2): = x_2/x_1 = f_1^\sharp (x_2, x_1)$. This example shows
that, when identifying $g_1^\flat$ and $f_1^\sharp$, the conventional order of the coordinates is not preserved, and depends of which connective
is taken as primitive.}
	
	
\subsection{Algebraic semantics for basic DLE-logics}
	In the present section the standard algebraic semantics for distributive lattice expansion logics are presented.
	
	\begin{definition}
	\label{def:canonical:extension}
	The \emph{canonical extension} of a bounded  lattice $L$ is a complete  lattice $L^\delta$ containing $L$ as a sublattice, such that:
				\begin{enumerate}
					\item \emph{(denseness)} every element of $L^\delta$ is both the join of meets and the meet of joins of elements from $L$;
					\item \emph{(compactness)} for all $S,T \subseteq L$, if $\bigwedge S \leq \bigvee T$ in $L^\delta$, then $\bigwedge F \leq \bigvee G$ for some finite sets $F \subseteq S$ and $G\subseteq T$.
				\end{enumerate}
	An element $k \in L^\delta$ (resp.~$o\in L^\delta$) is \emph{closed} (resp.\ \emph{open}) if is the meet (resp.\ join) of some subset of $L$. We let $K(L^\delta)$ (resp.~$O(L^\delta)$) denote the set of the closed (resp.~open) elements of $L^\delta$. For every unary, order-preserving map $h : L \to M$ between bounded lattices, the $\sigma$-{\em extension} of $h$ is defined firstly by declaring, for every $k\in K(L^\delta)$,
$$h^\sigma(k):= \bigwedge\{ h(a)\mid a\in L\mbox{ and } k\leq a\},$$ and then, for every $u\in L^\delta$,
$$h^\sigma(u):= \bigvee\{ h^\sigma(k)\mid k\in K(L^\delta)\mbox{ and } k\leq u\}.$$
The $\pi$-{\em extension} of $f$ is defined firstly by declaring, for every $o\in O(L^\delta)$,
$$h^\pi(o):= \bigvee\{ h(a)\mid a\in L\mbox{ and } a\leq o\},$$ and then, for every $u\in L^\delta$,
$$h^\pi(u):= \bigwedge\{ h^\pi(o)\mid o\in O(L^\delta)\mbox{ and } u\leq o\}.$$
The definitions above apply also to operations of any finite arity and order-type. Indeed,
taking  order-duals interchanges closed and open elements:
$K({(L^\delta)}^\partial) = O(L^\delta)$ and $O({(L^\delta)}^\partial) = K(L^\delta)$;  similarly, $K({(L^n)}^\delta) =K(L^\delta)^n$, and $O({(L^n)}^\delta) =O(L^\delta)^n$. Hence,  $K({(L^\delta)}^\epsilon) =\prod_i K(L^\delta)^{\epsilon(i)}$ and $O({(L^\delta)}^\epsilon) =\prod_i O(L^\delta)^{\epsilon(i)}$ for every lattice $L$ and every order-type $\epsilon$ over any $n\in \mathbb{N}$, where
\begin{center}
\begin{tabular}{cc}
$K(L^\delta)^{\epsilon(i)}: =\begin{cases}
K(L^\delta) & \mbox{if } \epsilon(i) = 1\\
O(L^\delta) & \mbox{if } \epsilon(i) = \partial\\
\end{cases}
$ &
$O(L^\delta)^{\epsilon(i)}: =\begin{cases}
O(L^\delta) & \mbox{if } \epsilon(i) = 1\\
K(L^\delta) & \mbox{if } \epsilon(i) = \partial.\\
\end{cases}
$\\
\end{tabular}
\end{center}
From this it follows that
 ${(L^\partial)}^\delta$ can be  identified with ${(L^\delta)}^\partial$,  ${(L^n)}^\delta$ with ${(L^\delta)}^n$, and
${(L^\epsilon)}^\delta$ with ${(L^\delta)}^\epsilon$ for any order type $\epsilon$ over $n$, where $L^\epsilon: = \prod_{i = 1}^n L^{\epsilon(i)}$.
These identifications make it possible to obtain the definition of $\sigma$-and $\pi$-extensions of $\epsilon$-monotone operations of any arity $n$ and order-type $\epsilon$ over $n$ by instantiating the corresponding definitions given above for monotone and unary functions.

	\end{definition}
	
	\begin{definition}
		\label{def:DLE}
		For any tuple $(\mathcal{F}, \mathcal{G})$ of disjoint sets of function symbols as above, a {\em distributive lattice expansion} (abbreviated as DLE) is a tuple $\bba = (L, \mathcal{F}^\bbA, \mathcal{G}^\bbA)$ such that $L$ is a bounded distributive lattice, $\mathcal{F}^\bbA = \{f^\bbA\mid f\in \mathcal{F}\}$ and $\mathcal{G}^\bbA = \{g^\bbA\mid g\in \mathcal{G}\}$, such that every $f^\bbA\in\mathcal{F}^\bbA$ (resp.\ $g^\bbA\in\mathcal{G}^\bbA$) is an $n_f$-ary (resp.\ $n_g$-ary) operation on $\bbA$. A DLE is {\em normal} if every $f^\bbA\in\mathcal{F}^\bbA$ (resp.\ $g^\bbA\in\mathcal{G}^\bbA$) preserves finite (hence also empty) joins (resp.\ meets) in each coordinate with $\epsilon_f(i)=1$ (resp.\ $\epsilon_g(i)=1$) and reverses finite (hence also empty) meets (resp.\ joins) in each coordinate with $\epsilon_f(i)=\partial$ (resp.\ $\epsilon_g(i)=\partial$). Let $\mathbb{DLE}$ be the class of normal DLEs. Sometimes we will refer to certain DLEs as $\mathcal{L}_\mathrm{DLE}$-algebras when we wish to emphasize that these algebras have a compatible signature with the logical language we have fixed.
A distributive lattice is {\em perfect} if it is complete, completely distributive and completely join-generated by  its completely join-prime elements. Equivalently, a distributive lattice is perfect iff it is isomorphic to the lattice of up-sets of some poset. A normal DLE is {\em perfect} if its lattice-reduct is a perfect distributive lattice, and each $f$-operation (resp.\ $g$-operation) is completely join-preserving (resp.\ meet-preserving) in the coordinates $i$ such that $\epsilon_f(i) = 1$ (resp.~$\epsilon_g(i) = 1$) and completely meet-reversing (resp.\ join-reversing) in the coordinates $i$ such that $\epsilon_f(i) = \partial$ (resp.~$\epsilon_g(i) = \partial$). The {\em canonical extension}  of a  normal DLE $\mathbb{A} = (L, \mathcal{F}, \mathcal{G})$
 is  the perfect normal DLE $\mathbb{A}^\delta: = (L^\delta, \mathcal{F}^\sigma, \mathcal{G}^\pi)$, where $L^\delta$  is the canonical extension of  $L$ (cf. Definition \ref{def:canonical:extension}),
and $\mathcal{F}^\sigma: = \{f^\sigma\mid f\in \mathcal{F}\}$ and  $\mathcal{G}^\pi: = \{g^\pi\mid g\in \mathcal{G}\}$.
Canonical extensions of  Heyting algebras, Brouwerian algebras and bi-Heyting algebras are defined by instantiating the definition above in the corresponding  signatures. The canonical extension of any Heyting (resp.~Brouwerian, bi Heyting) algebra is a (perfect) Heyting (resp.~Brouwerian, bi-Heyting) algebra.
\end{definition}

In the remainder of the paper, we will abuse notation and write e.g.\ $f$ for $f^\bbA$ when this causes no confusion.
Normal DLEs constitute the main semantic environment of the present paper. Henceforth, since every DLE is assumed to be normal, the adjective will be typically dropped.
The class of all DLEs is equational, and can be axiomatized by the usual distributive lattice identities and the following equations for any $f\in \mathcal{F}$ (resp.\ $g\in \mathcal{G}$) and $1\leq i\leq n_f$ (resp.\ for each $1\leq j\leq n_g$):
	\begin{itemize}
		\item if $\varepsilon_f(i) = 1$, then 
		\[
		f(p_1,\ldots, p\vee q,\ldots,p_{n_f}) = f(p_1,\ldots, p,\ldots,p_{n_f})\vee f(p_1,\ldots, q,\ldots,p_{n_f});
		\] 
		moreover if $f\in \mathcal{F}_n$, then $f(p_1,\ldots, \bot,\ldots,p_{n_f}) = \bot$,
		\item if $\varepsilon_f(i) = \partial$, then 
		\[f(p_1,\ldots, p\wedge q,\ldots,p_{n_f}) = f(p_1,\ldots, p,\ldots,p_{n_f})\vee f(p_1,\ldots, q,\ldots,p_{n_f});
		\]
		moreover if $f\in \mathcal{F}_n$, then  $f(p_1,\ldots, \top,\ldots,p_{n_f}) = \bot$,
		\item if $\varepsilon_g(j) = 1$, then \[g(p_1,\ldots, p\wedge q,\ldots,p_{n_g}) = g(p_1,\ldots, p,\ldots,p_{n_g})\wedge g(p_1,\ldots, q,\ldots,p_{n_g});\]
		moreover if $g\in \mathcal{G}_n$, then  $g(p_1,\ldots, \top,\ldots,p_{n_g}) = \top$,
		\item if $\varepsilon_g(j) = \partial$, then 
		\[g(p_1,\ldots, p\vee q,\ldots,p_{n_g}) = g(p_1,\ldots, p,\ldots,p_{n_g})\wedge g(p_1,\ldots, q,\ldots,p_{n_g});
		\]
		moreover if $g\in \mathcal{G}_n$, then  $g(p_1,\ldots, \bot,\ldots,p_{n_g}) = \top$.
	\end{itemize}
	Each language $\mathcal{L}_\mathrm{DLE}$ is interpreted in the appropriate class of DLEs. In particular, for every DLE $\bbA$, each operation $f^\bbA\in \mathcal{F}^\bbA$ (resp.\ $g^\bba\in \mathcal{G}^\bbA$) is finitely join-preserving (resp.\ meet-preserving) in each coordinate when regarded as a map $f^\bbA: \bba^{\varepsilon_f}\to \bba$ (resp.\ $g^\bba: \bba^{\varepsilon_g}\to \bba$). 
	
	For every DLE $\bbA$, the symbol $\vdash$ is interpreted as the lattice order $\leq$. A sequent $\phi\vdash\psi$ is valid in $\bba$ if $h(\phi)\leq h(\psi)$ for every homomorphism $h$ from the $\mathcal{L}_\mathrm{DLE}$-algebra of formulas over $\mathsf{PROP}$ to $\bba$. The notation $\mathbb{DLE}\models\phi\vdash\psi$ indicates that $\phi\vdash\psi$ is valid in every DLE. Then, by means of a routine Lindenbaum-Tarski construction, it can be shown that the minimal DLE-logic $\mathbf{L}_\mathrm{DLE}$ is sound and complete with respect to its corresponding class of algebras $\mathbb{DLE}$, i.e.\ that any sequent $\phi\vdash\psi$ is provable in $\mathbf{L}_\mathrm{DLE}$ iff $\mathbb{DLE}\models\phi\vdash\psi$. 

\subsection{Inductive and very-simple Sahlqvist inequalities}

				
				In the present subsection, we will report on the definition of {\em inductive} $\mathcal{L}_\mathrm{DLE}$-inequalities
				on which the algorithm ALBA is guaranteed to succeed 
				(cf.\ \cite{CoGhPa14,conradie2012algorithmic}).
				
				\begin{definition}[\textbf{Signed Generation Tree}]
					\label{def: signed gen tree}
					The \emph{positive} (resp.\ \emph{negative}) {\em generation tree} of any $\mathcal{L}_\mathrm{DLE}$-term $s$ is defined by labelling the root node of the generation tree of $s$ with the sign $+$ (resp.\ $-$), and then propagating the labelling on each remaining node as follows:
					\begin{itemize}
						\item For any node labelled with $ \lor$ or $\land$, assign the same sign to its children nodes.
						\item For any node labelled with $h\in \mathcal{F}\cup \mathcal{G}$ of arity $n_h\geq 1$, and for any $1\leq i\leq n_h$, assign the same (resp.\ the opposite) sign to its $i$th child node if $\varepsilon_h(i) = 1$ (resp.\ if $\varepsilon_h(i) = \partial$).
					\end{itemize}
					Nodes in signed generation trees are \emph{positive} (resp.\ \emph{negative}) if are signed $+$ (resp.\ $-$).
				\end{definition}
				
				Signed generation trees will be mostly used in the context of term inequalities $s\leq t$. In this context we will typically consider the positive generation tree $+s$ for the left-hand side and the negative one $-t$ for the right-hand side. We will also say that a term-inequality $s\leq t$ is \emph{uniform} in a given variable $p$ if all occurrences of $p$ in both $+s$ and $-t$ have the same sign, and that $s\leq t$ is $\epsilon$-\emph{uniform} in a (sub)array $\overline{p}$ of its variables if $s\leq t$ is uniform in $p$, occurring with the sign indicated by $\epsilon$, for every $p$ in $\overline{p}$.
				
				For any term $s(p_1,\ldots p_n)$, any order type $\epsilon$ over $n$, and any $1 \leq i \leq n$, an \emph{$\epsilon$-critical node} in a signed generation tree of $s$ is a leaf node $+p_i$ with $\epsilon_i = 1$ or $-p_i$ with $\epsilon_i = \partial$. An $\epsilon$-{\em critical branch} in the tree is a branch from an $\epsilon$-critical node. The intuition, which will be built upon later, is that variable occurrences corresponding to $\epsilon$-critical nodes are \emph{to be solved for}, according to $\epsilon$.
				
				For every term $s(p_1,\ldots p_n)$ and every order type $\epsilon$, we say that $+s$ (resp.\ $-s$) {\em agrees with} $\epsilon$, and write $\epsilon(+s)$ (resp.\ $\epsilon(-s)$), if every leaf in the signed generation tree of $+s$ (resp.\ $-s$) is $\epsilon$-critical.
				In other words, $\epsilon(+s)$ (resp.\ $\epsilon(-s)$) means that all variable occurrences corresponding to leaves of $+s$ (resp.\ $-s$) are to be solved for according to $\epsilon$. We will also write $+s'\prec \ast s$ (resp.\ $-s'\prec \ast s$) to indicate that the subterm $s'$ inherits the positive (resp.\ negative) sign from the signed generation tree $\ast s$. Finally, we will write $\epsilon(\gamma) \prec \ast s$ (resp.\ $\epsilon^\partial(\gamma) \prec \ast s$) to indicate that the signed subtree $\gamma$, with the sign inherited from $\ast s$, agrees with $\epsilon$ (resp.\ with $\epsilon^\partial$).
				\begin{definition}
					\label{def:good:branch}
					Nodes in signed generation trees will be called \emph{$\Delta$-adjoints}, \emph{syntactically left residual (SLR)}, \emph{syntactically right residual (SRR)}, and \emph{syntactically right adjoint (SRA)}, according to the specification given in Table \ref{Join:and:Meet:Friendly:Table}.
					A branch in a signed generation tree $\ast s$, with $\ast \in \{+, - \}$, is called a \emph{good branch} if it is the concatenation of two paths $P_1$ and $P_2$, one of which may possibly be of length $0$, such that $P_1$ is a path from the leaf consisting (apart from variable nodes) only of PIA-nodes\footnote{For explanations of our choice of terminologies here, we refer to \cite[Remark 3.24]{PaSoZh15r}.}, and $P_2$ consists (apart from variable nodes) only of Skeleton-nodes. 
					\begin{table}[ht]
						\begin{center}
							\begin{tabular}{| c | c |}
								\hline
								Skeleton &PIA\\
								\hline
								$\Delta$-adjoints & SRA \\
								\begin{tabular}{ c c c c c c}
									$+$ &$\vee$ &$\wedge$ &$\phantom{\lhd}$ & &\\
									$-$ &$\wedge$ &$\vee$\\
									\hline
								\end{tabular}
								&
								\begin{tabular}{c c c c }
									$+$ &$\wedge$ &$g$ & with $n_g = 1$ \\
									$-$ &$\vee$ &$f$ & with $n_f = 1$ \\
									\hline
								\end{tabular}
								\\
								SLR &SRR\\
								\begin{tabular}{c c c c }
									$+$ & $\wedge$ &$f$ & with $n_f \geq 1$\\
									$-$ & $\vee$ &$g$ & with $n_g \geq 1$ \\
								\end{tabular}
								&\begin{tabular}{c c c c}
									$+$ &$\vee$ &$g$ & with $n_g \geq 2$\\
									$-$ & $\wedge$ &$f$ & with $n_f \geq 2$\\
								\end{tabular}
								\\
								\hline
							\end{tabular}
						\end{center}
						\caption{Skeleton and PIA nodes for $\mathrm{DLE}$.}\label{Join:and:Meet:Friendly:Table}
						\vspace{-1em}
					\end{table}
				\end{definition}
				
				\begin{definition}[Inductive inequalities]
					\label{Inducive:Ineq:Def}
					For any order type $\epsilon$ and any irreflexive and transitive relation $<_\Omega$ on $p_1,\ldots p_n$, the signed generation tree $*s$ $(* \in \{-, + \})$ of a term $s(p_1,\ldots p_n)$ is \emph{$(\Omega, \epsilon)$-inductive} if
					\begin{enumerate}
						\item for all $1 \leq i \leq n$, every $\epsilon$-critical branch with leaf $p_i$ is good (cf.\ Definition \ref{def:good:branch});
						\item every $m$-ary SRR-node occurring in the critical branch is of the form 
						\[	\circledast(\gamma_1,\dots,\gamma_{j-1},\beta,\gamma_{j+1}\ldots,\gamma_m),\] where for any $h\in\{1,\ldots,m\}\setminus j$: 
						\begin{enumerate}
							\item $\epsilon^\partial(\gamma_h) \prec \ast s$ (cf.\ discussion before Definition \ref{def:good:branch}), and
							%
							\item $p_k <_{\Omega} p_i$ for every $p_k$ occurring in $\gamma_h$ and for every $1\leq k\leq n$.
						\end{enumerate}

					\end{enumerate}
					
					We will refer to $<_{\Omega}$ as the \emph{dependency order} on the variables. An inequality $s \leq t$ is \emph{$(\Omega, \epsilon)$-inductive} if the signed generation trees $+s$ and $-t$ are $(\Omega, \epsilon)$-inductive. An inequality $s \leq t$ is \emph{inductive} if it is $(\Omega, \epsilon)$-inductive for some $\Omega$ and $\epsilon$.
				\end{definition}

\begin{notation}
\label{notation:inductive:compact}
Following \cite{chen2021syntactic},  we will often write $(\Omega, \epsilon)$-inductive inequalities as follows: 
\begin{equation}
\label{eqn:Inductive:Compact}
(\varphi\leq \psi)[\overline{\alpha}/!\overline{x}, \overline{\beta}/!\overline{y},\overline{\gamma}/!\overline{z}, \overline{\delta}/!\overline{w}],
\end{equation}
where $(\varphi\leq \psi)[!\overline{x}, !\overline{y},!\overline{z}, !\overline{w}]$ contains only skeleton nodes, is positive (resp.\ negative) in $!\overline x$ and $!\overline z$ (resp.\ $!\overline y$ and $!\overline w$), and it is \emph{scattered}, i.e.\ each variable occurs only once; each $\alpha$ in $\overline\alpha$ (resp.\ $\beta$ in $\overline\beta$) is a positive (resp.\ negative) PIA.
\end{notation}

\begin{definition}[(Very simple) Sahlqvist inequality]
An inductive inequality \[(\varphi\leq \psi)[\overline{\alpha}/!\overline{x}, \overline{\beta}/!\overline{y},\overline{\gamma}/!\overline{z}, \overline{\delta}/!\overline{w}],\] is \emph{Sahlqvist} if every $\alpha$ in $\overline\alpha$ and $\beta$ in $\overline\beta$ contains only unary connectives. It is \emph{very simple Sahlqvist} if every $\alpha$ and $\beta$ is a propositional variable.
\end{definition}

In what follows, we will refer to  a formula $\chi$ such that $+\chi$ (resp.\ $-\chi$)  consists only of skeleton nodes as a \emph{positive} (resp.\ \emph{negative}) \emph{skeleton}; and we dub formulas $\zeta$ as positive (resp.\ negative) PIA if there is a path from a leaf to the root of $+\zeta$ (resp.\ $-\zeta$) consisting only of PIA nodes.

An $(\varepsilon,\Omega)$-inductive formula is \emph{definite} $(\varepsilon,\Omega)$-inductive if it has no $+\wedge$, and no $-\vee$ in the PIA portion of its $\varepsilon$-critical branches. Every inductive formula is equivalent to some conjunction of definite inductive formulas which is obtained through ALBA's $\mathsf{Preprocessing}$ step (cf. Appendix \ref{appendix:albarules}).

\begin{definition}[$\mathsf{LA}$ and $\mathsf{RA}$]
\label{def:RA_and_LA}
For every positive (definite) PIA formula $\varphi=\varphi(!x,\overline z)$ and negative PIA formula $\psi=\psi(!x,\overline z)$ where $x$ is a leaf of a PIA-path to the root, we define the formulas $\mathsf{LA}(\varphi)(u,\overline z)$ and $\mathsf{RA}(\psi)(u, \overline z)$ (with $u$ a new fresh variable) by simultaneous recursion:
\begin{center}
				\begin{tabular}{r c l}
					$\mathsf{LA}(x)$ &= &$u$;\\
					$\mathsf{LA}(\Box \phi(x, \oz))$ &= &$\mathsf{LA}(\phi)(\Diamondblack u, \overline{z})$;\\
					$\mathsf{LA}(\psi(\oz) \rightarrow \phi(x, \oz))$ &= &$\mathsf{LA}(\phi)(u \wedge \psi(\oz), \oz)$;\\
					$\mathsf{LA}(\phi_1(\oz) \vee \phi_2(x, \oz))$ &= &$\mathsf{LA}(\phi_2)(u - \phi_1(\oz), \oz)$;\\
					$\mathsf{LA}(\psi(x, \oz)\rightarrow\phi(\oz))$ &= &$\mathsf{RA}(\psi)(u \rightarrow \phi(\oz), \oz)$;\\
					$\mathsf{LA}(g(\overline{\phi_{-j}(\oz)},\phi_j(x,\oz), \overline{\psi(\oz)}))$ &= &$\mathsf{LA}(\phi_j)(g^{\flat}_{j}(\overline{\phi_{-j}(\oz)},u, \overline{\psi(\oz)} ), \oz)$;\\
					$\mathsf{LA}(g(\overline{\phi(\oz)}, \overline{\psi_{-j}(\oz)},\psi_j(x,\oz)))$ &= &$\mathsf{RA}(\psi_j)(g^{\flat}_{j}(\overline{\phi(\oz)}, \overline{\psi_{-j}(\oz)},u), \oz)$;\\
					&&\\
					$\mathsf{RA}(x)$ &= &$u$;\\
					$\mathsf{RA}(\Diamond \psi(x, \oz))$ &= &$\mathsf{RA}(\psi)(\blacksquare u, \overline{z})$;\\
					$\mathsf{RA}(\psi(x, \oz) - \phi(\oz))$ &= &$\mathsf{RA}(\psi)(\phi(\oz) \vee u, \oz)$;\\
					$\mathsf{RA}(\psi_1(\oz) \wedge \psi_2(x, \oz))$ &= &$\mathsf{RA}(\psi_2)(\psi_1(\oz) \rightarrow u, \oz)$;\\
					$\mathsf{RA}(\psi(\oz) - \phi(x, \oz))$ &= &$\mathsf{LA}(\phi)(\psi(\oz) - u, \oz)$;\\
					$\mathsf{RA}(f(\overline{\psi_{-j}(\oz)},\psi_j(x,\oz), \overline{\phi(\oz)}))$ &= &$\mathsf{RA}(\psi_j)(f^{\sharp}_{j}(\overline{\psi_{-j}(\oz)},u, \overline{\phi(\oz)} ), \oz)$;\\
					$\mathsf{RA}(f(\overline{\psi(\oz)}, \overline{\phi_{-j}(\oz)},\phi_j(x,\oz)))$ &= &$\mathsf{LA}(\phi_j)(f^{\sharp}_{j}(\overline{\psi(\oz)}, \overline{\phi_{-j}(\oz)},u), \oz)$.\\
				\end{tabular}
			\end{center}
\end{definition}

\subsection{ALBA}
ALBA is a calculus for correspondence that is based on the Ackermann lemma that has been shown to work on inductive inequalities in an effective way \cite{CoGhPa14,CoPa2012distr,CoPa:non-dist}. A comprehensive description of the algorithm for correspondence and its rules can be found in Appendix \ref{appendix:albarules}.

\subsubsection{ALBA languages}

ALBA manipulates inequalities and quasi-inequalities\footnote{A {\em quasi-inequality} of $\mathcal{L}_{\mathrm{DLE}}$ is an expression of the form $\bigmetaand_{i = 1}^n s_i\leq t_i \Rightarrow s\leq t$, where $s_i\leq t_i$ and $s\leq t$ are $\mathcal{L}_{\mathrm{DLE}}$-inequalities for each $i$.} in the expanded language $\mathcal{L}_\mathrm{DLE}^{+}$, which is built up on the base of the lattice constants $\top, \bot$ and an enlarged set of propositional variables 
$\mathsf{NOM}\cup \mathsf{CONOM}\cup \mathsf{AtProp}$ (the variables $\nomh, \mathbf{i}, \mathbf{j}, \nomk$ in $\mathsf{NOM}$ are referred to as \emph{nominals}, and the variables $\cnoml, \mathbf{m}, \mathbf{n}, \cnomo$ in $\mathsf{CONOM}$ as \emph{conominals}), 
closing under the logical connectives of $\mathcal{L}_\mathrm{DLE}^*$. The natural semantic environment of $\mathcal{L}_\mathrm{DLE}^{+}$ is given by {\em perfect} $\mathcal{L}_\mathrm{DLE}$-{\em algebras}, which therefore are both completely join-generated by their completely join-irreducible elements $\jty$ and completely meet-generated by their completely meet-irreducible elements $\mty$.
Nominals and conominals are interpreted over the sets of the completely join-irreducible elements and the completely meet-irreducible elements of perfect DLEs.

We will extensively write $\pureu, \purev, \purew$ to indicate generic \emph{pure variables}, i.e., variables in $\mathsf{NOM}\cup\mathsf{CONOM}$. The language of $\langalba$-inequalities is referred to as $\langineq$. An inequality in $\langineq$ whose variables are all pure is a \emph{pure inequality}.

To define ALBA's language for inverse correspondence, the following considerations will play a key role.

\begin{remark}
In a distributive lattice $\bba$, the maps $\kappa:\jty(\bba)\to\mty(\bba)$ defined by $j \mapsto \bigvee\{ a \in \bba \mid j \nleq a \}$ and $\lambda:\mty(\bba)\to\jty(\bba)$ defined by $m \mapsto \bigwedge \{ a \in \bba \mid a \nleq m \}$ are order isomorphism when considering $\jty(\bba)$ and $\mty(\bba)$ as subposets of $\bba$ \cite[Section 2.3]{GEHRKE200565}. 
\end{remark}

The following lemma follows plainly from the definitions of $\kappa$ and $\lambda$.

\begin{lemma}
Given a distributive lattice $\bba$, for every $j \in \jty(\bba), m \in \mty(\bba), a \in \bba$,
\[
\begin{array}{rcl}
\label{lemma:flipnegation}
j \nleq a & \mbox{iff} & a \leq \kappa(j) \\
a \nleq m & \mbox{iff} & \lambda(m) \leq a
\end{array}
\]
\end{lemma}

\begin{definition}[ALBA language for inverse correspondence]
\label{def:albametalang}
ALBA's language for inverse correspondence $\langmeta$ is generated by the following rules:
\[
\xi ::= s \leq t \ |\ \xi \metaand \xi  \ |\ \xi \metaor \xi  \ |\ \metanot \xi  \ |\ \xi \Rightarrow \xi \ | \ \forall \nomj\ \xi  \ |\  \forall \cnomm\ \xi  \ |\ \exists \nomj \ \xi  \ |\ \exists \cnomm\ \xi,
\]
where $\metaand$ (resp.\ $\metaor$) denotes the meta-linguistic conjunction (resp.\ disjunction), $\Rightarrow$ (resp.\ $\metanot$) denotes the meta-linguistic implication (resp.\ negation), and $s\leq t$ is an inequality in $\langineq$ enriched with symbols for $\kappa$ and $\lambda$, which are interpreted as the isomorphisms between the join and meet irreducibles of the interpretation algebra.
\end{definition}

\subsubsection{Right-handed Ackermann lemmas}

\begin{lemma}[Right-handed universal Ackermann]
Let $\alpha, \beta(p), \gamma(p), \delta(p), \varepsilon(p)$ be formulas of a language $\mathcal{L}^+$ over the set of variables $\atprop$; let $p \in \atprop$ such that $p\notin FV(\alpha)$; if $\beta$ and $\varepsilon$ are positive in $p$ and $\gamma$ and $\delta$ are negative in $p$, then
 the following are equivalent for every $\mathcal{L}^+$-algebra  $\mathbb{A}$: 

\begin{enumerate}

\item[(a)] 
$\mathbb{A}\models [(\alpha\leq p\ \&\ \beta(p)\leq\gamma(p))\ \Rightarrow\ \delta(p)\leq \varepsilon(p)]$; 
\item[(b)] $ \mathbb{A}\models (\beta(\alpha/p)\leq \gamma(\alpha/p)\ \Rightarrow\ \delta(\alpha/p)\leq \varepsilon(\alpha/p))$.
\end{enumerate}
\end{lemma}
\begin{proof}
For the direction from (a) to (b), let us argue contrapositively and let $\mathbb{A}$ and $V$ be such that $ (\mathbb{A}, V)\models\beta(\alpha/p)\leq \gamma(\alpha/p)$ but $( \mathbb{A}, V)\not\models \delta(\alpha/p)\leq \varepsilon(\alpha/p)$. Let $V^\ast$ be the $p$-variant of $V$ such that $V^\ast(p): = \val{\alpha}_{V}$. Since the variable $p$ does not occur in $\alpha$, this implies that  $\val{\alpha}_{V^\ast} = \val{\alpha}_{V} = V^\ast(p)$, which proves that $(\mathbb{A}, V^\ast)\models \alpha\leq p$. The second part of the statement immediately follows from the observation that, for every formula $\xi$, the following chain of equalities holds:  $\val{\xi(p)}_{V^\ast} = \val{\xi}(V^\ast(p)) = \val{\xi}(\val{\alpha}_V) = \val{\xi(\alpha/p)}_{V}$.\\
Conversely, let $\mathbb{A}$ and $V$ be such that $(\mathbb{A}, V)\models\alpha\leq p$  and $(\mathbb{A}, V)\models \beta(p)\leq \gamma(p)$. Since $\val{\alpha}_{V}\leq V(p)$, and $\beta$ and $\gamma$ are respectively   positive and negative in $p$, this implies that  $\val{\beta(\alpha/p)}_{V} \leq \val{\beta(p)}_{V}\leq \val{\gamma(p)}_{V}\leq \val{\gamma(\alpha/p)}_{V}$, which proves that $(\mathbb{A}, V)\models \beta(\alpha/p)\leq\gamma(\alpha/p)$. Hence, by (b), we conclude that $(\mathbb{A}, V)\models \delta(\alpha/p)\leq \varepsilon(\alpha/p)$. Therefore,  since $\delta$ and $\varepsilon$ are respectively   negative and positive in $p$, the following chain of inequalities holds:
$\val{\delta(p)}_{V}\leq \val{\delta(\alpha/p)}_{V} \leq \val{\varepsilon(\alpha/p)}_{V} \leq \val{\varepsilon(p)}_{V}$, which finishes the proof of (a).
\end{proof}
Notice that in the following, existential version,  $\beta$, $\varepsilon$, $\gamma$ and $\delta$ are assumed to have the opposite polarities they had in the universal version above.
\begin{lemma}[Right-handed existential Ackermann]
Let $\alpha, \beta(p), \gamma(p), \delta(p), \varepsilon(p)$ be formulas of a language $\mathcal{L}^+$ over the set of variables $\mathsf{PROP}$; let $p \in \mathsf{PROP}$ such that $p\notin FV(\alpha)$; if $\beta$ and $\varepsilon$ are negative in $p$ and $\gamma$ and $\delta$ are positive in $p$, then
 the following are equivalent for every $\mathcal{L}^+$-algebra  $\mathbb{A}$ and any $\mathcal{L}^+$-valuation $V$ on $\mathbb{A}$:

\begin{enumerate}
\item[(a)] $ (\mathbb{A}, V)\models(\beta(\alpha/p)\leq \gamma(\alpha/p)\ \Rightarrow\ \delta(\alpha/p)\leq \varepsilon(\alpha/p))$;

\item[(b)] 
$(\mathbb{A}, V^\ast)\models \alpha\leq p$ and 
$ (\mathbb{A}, V^\ast)\models (\beta(p)\leq\gamma(p)\ \Rightarrow\ \delta(p)\leq \varepsilon(p))$ for some $p$-variant $V^\ast$ of $V$.  
\end{enumerate}
\end{lemma}
\begin{proof}
For the direction from (a) to (b), let $\mathbb{A}$ and $V$ be as above; let $V^\ast$ be the $p$-variant of $V$ such that $V^\ast(p): = \val{\alpha}_{V}$. Since the variable $p$ does not occur in $\alpha$, this implies that  $\val{\alpha}_{V^\ast} = \val{\alpha}_{V} = V^\ast(p)$, which proves that $(\mathbb{A}, V^\ast)\models \alpha\leq p$. The second part of the statement immediately follows from the observation that, for every formula $\xi$, the following chain of equalities holds:  $\val{\xi(p)}_{V^\ast} = \val{\xi}(V^\ast(p)) = \val{\xi}(\val{\alpha}_V) = \val{\xi(\alpha/p)}_{V}$.\\
Conversely, let $\mathbb{A}$ and $V$ be as above, and assume that $(\mathbb{A}, V)\models \beta(\alpha/p)\leq \gamma(\alpha/p)$. Since $\val{\alpha}_{V} = \val{\alpha}_{V^\ast}\leq V^\ast(p)$, and  $\beta$ and $\gamma$ are respectively  negative and positive in $p$, this implies that  $\val{\beta(p)}_{V^\ast} \leq \val{\beta(\alpha/p)}_{V}\leq \val{\gamma(\alpha/p)}_{V}\leq \val{\gamma(p)}_{V^\ast}$, which proves that $(\mathbb{A}, V^\ast)\models \beta(p)\leq\gamma(p)$. Hence, by (b), we conclude that $(\mathbb{A}, V^\ast)\models \delta(p)\leq \varepsilon(p)$. Therefore, since $\delta$ and $\varepsilon$ are respectively  positive and negative in $p$, the following chain of inequalities holds:
$\val{\delta(\alpha/p)}_{V}\leq \val{\delta(p)}_{V^\ast} \leq \val{\varepsilon(p)}_{V^\ast} \leq \val{\varepsilon(\alpha/p)}_{V}$, which finishes the proof of (a).
\end{proof}

\subsubsection{Left-handed Ackermann lemmas}

\begin{lemma}[Left-handed universal Ackermann]
Let $\alpha, \beta(p), \gamma(p), \delta(p), \varepsilon(p)$ be formulas of a language $\mathcal{L}^+$ over the set of variables $\atprop$; let $p \in \atprop$ such that $p\notin FV(\alpha)$; if $\beta$ and $\varepsilon$ are negative in $p$ and $\gamma$ and $\delta$ are positive in $p$, then
 the following are equivalent for every $\mathcal{L}^+$-algebra  $\mathbb{A}$: 

\begin{enumerate}

\item[(a)] 
$\mathbb{A}\models [(p\leq \alpha\ \&\ \beta(p)\leq\gamma(p))\ \Rightarrow\ \delta(p)\leq \varepsilon(p)]$; 
\item[(b)] $ \mathbb{A}\models (\beta(\alpha/p)\leq \gamma(\alpha/p)\ \Rightarrow\ \delta(\alpha/p)\leq \varepsilon(\alpha/p))$.
\end{enumerate}
\end{lemma}
\begin{proof}
For the direction from (a) to (b), let us argue contrapositively and let $\mathbb{A}$ and $V$ be such that $ (\mathbb{A}, V)\models\beta(\alpha/p)\leq \gamma(\alpha/p)$ but $( \mathbb{A}, V)\not\models \delta(\alpha/p)\leq \varepsilon(\alpha/p)$. Let $V^\ast$ be the $p$-variant of $V$ such that $V^\ast(p): = \val{\alpha}_{V}$. Since the variable $p$ does not occur in $\alpha$, this implies that  $\val{\alpha}_{V^\ast} = \val{\alpha}_{V} = V^\ast(p)$, which proves that $(\mathbb{A}, V^\ast)\models p\leq \alpha$. The second part of the statement immediately follows from the observation that, for every formula $\xi$, the following chain of equalities holds:  $\val{\xi(p)}_{V^\ast} = \val{\xi}(V^\ast(p)) = \val{\xi}(\val{\alpha}_V) = \val{\xi(\alpha/p)}_{V}$.\\
Conversely, let $\mathbb{A}$ and $V$ be such that $(\mathbb{A}, V)\models p\leq \alpha$  and $(\mathbb{A}, V)\models \beta(p)\leq \gamma(p)$. Since $V(p)\leq \val{\alpha}_{V}$, and $\beta$ and $\gamma$ are respectively   negative and positive in $p$, this implies that  $\val{\beta(\alpha/p)}_{V} \leq \val{\beta(p)}_{V}\leq \val{\gamma(p)}_{V}\leq \val{\gamma(\alpha/p)}_{V}$, which proves that $(\mathbb{A}, V)\models \beta(\alpha/p)\leq\gamma(\alpha/p)$. Hence, by (b), we conclude that $(\mathbb{A}, V)\models \delta(\alpha/p)\leq \varepsilon(\alpha/p)$. Therefore,  since $\delta$ and $\varepsilon$ are respectively   positive and negative in $p$, the following chain of inequalities holds: $\val{\delta(p)}_{V}\leq \val{\delta(\alpha/p)}_{V} \leq \val{\varepsilon(\alpha/p)}_{V} \leq \val{\varepsilon(p)}_{V}$, which finishes the proof of (a).
\end{proof}
Notice that in the following, existential version,  $\beta$, $\varepsilon$, $\gamma$ and $\delta$ are assumed to have the opposite polarities they had in the universal version above.
\begin{lemma}[Left-handed existential Ackermann]
Let $\alpha, \beta(p), \gamma(p), \delta(p), \varepsilon(p)$ be formulas of a language $\mathcal{L}^+$ over the set of variables $\atprop$; let $p \in \atprop$ such that $p\notin FV(\alpha)$; if $\beta$ and $\varepsilon$ are positive in $p$ and $\gamma$ and $\delta$ are negative in $p$, then
 the following are equivalent for every $\mathcal{L}^+$-algebra  $\mathbb{A}$ and any $\mathcal{L}^+$-valuation $V$ on $\mathbb{A}$:

\begin{enumerate}
\item[(a)] $ (\mathbb{A}, V)\models(\beta(\alpha/p)\leq \gamma(\alpha/p)\ \Rightarrow\ \delta(\alpha/p)\leq \varepsilon(\alpha/p))$;

\item[(b)] 
$(\mathbb{A}, V^\ast)\models p\leq \alpha$ and 
$ (\mathbb{A}, V^\ast)\models (\beta(p)\leq\gamma(p)\ \Rightarrow\ \delta(p)\leq \varepsilon(p))$ for some $p$-variant $V^\ast$ of $V$. 
\end{enumerate}
\end{lemma}
\begin{proof}
For the direction from (a) to (b), let $\mathbb{A}$ and $V$ be as above; let $V^\ast$ be the $p$-variant of $V$ such that $V^\ast(p): = \val{\alpha}_{V}$. Since the variable $p$ does not occur in $\alpha$, this implies that  $\val{\alpha}_{V^\ast} = \val{\alpha}_{V} = V^\ast(p)$, which proves that $(\mathbb{A}, V^\ast)\models p\leq \alpha$. The second part of the statement immediately follows from the observation that, for every formula $\xi$, the following chain of equalities holds:  $\val{\xi(p)}_{V^\ast} = \val{\xi}(V^\ast(p)) = \val{\xi}(\val{\alpha}_V) = \val{\xi(\alpha/p)}_{V}$.\\
Conversely, let $\mathbb{A}$ and $V$ be as above, and assume that $(\mathbb{A}, V)\models \beta(\alpha/p)\leq \gamma(\alpha/p)$. Since $V^\ast(p)\leq \val{\alpha}_{V^\ast} = \val{\alpha}_{V} $, and  $\beta$ and $\gamma$ are respectively  positive and negative in $p$, this implies that  $\val{\beta(p)}_{V^\ast} \leq \val{\beta(\alpha/p)}_{V}\leq \val{\gamma(\alpha/p)}_{V}\leq \val{\gamma(p)}_{V^\ast}$, which proves that $(\mathbb{A}, V^\ast)\models \beta(p)\leq\gamma(p)$. Hence, by (b), we conclude that $(\mathbb{A}, V^\ast)\models \delta(p)\leq \varepsilon(p)$. Therefore, since $\delta$ and $\varepsilon$ are respectively  negative and positive in $p$, the following chain of inequalities holds: $\val{\delta(\alpha/p)}_{V}\leq \val{\delta(p)}_{V^\ast} \leq \val{\varepsilon(p)}_{V^\ast} \leq \val{\varepsilon(\alpha/p)}_{V}$, which finishes the proof of (a).
\end{proof}

\subsubsection{First approximation as restricted Ackermann}
\label{ssec:fapproxrestrackermann}
A key role in the proof of the canonicity of inductive inequalities is played by a version of the Ackermann lemma where the existence of the minimal valuation is subject to the additional requirement that the minimal valuation in question be  {\em admissible}, rather than arbitrary. It is precisely to satisfy this additional requirement on the minimal valuation that the topological and order-theoretic properties of perfect algebras come into play. Similar considerations apply to the proof of the soundness of the first approximation rule, which can be framed as a case of `restricted Ackermann' argument, in which the minimal valuation is subject to the additional requirement that  nominal (resp.~conominal) variables are to be interpreted as complete join-generators (resp.~meet-generators) of the algebra, rather than as arbitrary elements. Also in this case, we need to appeal to additional order-theoretic properties of the term-functions, beyond monotonicity/antitonicity.
\begin{lemma}[Universal Ackermann with nominals and conominals]
\label{lemma:Universal_Ackermann_with_nominals_and_conominals}
Let $\mathcal{L}$ be an LE-language, and  $(\varphi\leq \psi) [\overline{\rho}/!\overline{z}, \overline{\lambda}/!\overline{w}] $  be an $\mathcal{L}^+$-inequality. If $+\varphi(!\overline{z}, !\overline{w})$ and $-\psi(!\overline{w}, !\overline{z})$ are Skeleton formulas, and $+z\prec +\varphi$ and $+z\prec -\psi$ and $-w\prec +\varphi$ and $-w\prec -\psi$, then
 the following are equivalent for every $\mathcal{L}^+$-algebra  $\mathbb{A}$:

\begin{enumerate}
\item[(a)] $ \mathbb{A}\models (\varphi\leq \psi) [\overline{\rho}/!\overline{z}, \overline{\lambda}/!\overline{w}] $;

\item[(b)] 
$ \mathbb{A}\models \forall\overline{\nomj}\forall\overline{\cnomm}[ (\overline{\nomj\leq \rho}\ \&\ \overline{\lambda\leq\cnomm}) \Rightarrow\ (\varphi\leq \psi) [!\overline{\nomj}/!\overline{z}, !\overline{\cnomm}/!\overline{w}]]$, where all $\nomj$ in $\overline{\nomj}$ and $\cnomm$ in $\overline{\cnomm}$ are fresh.
\end{enumerate}
\end{lemma}
\begin{proof}
For the direction from (a) to (b), let $V$ be an $\mathcal{L}^+$-valuation on $\mathbb{A}$ s.t.~$( \mathbb{A}, V)\models \overline{\nomj\leq \rho}$ and $( \mathbb{A}, V)\models  \overline{\lambda\leq\cnomm}$. Since $\varphi$ (resp.~$\psi$) is positive (resp.~negative) in each variable $z$ in $\overline{z}$ and is negative (resp.~positive) in each variable $w$ in $\overline{w}$, the following chain of inequalities holds: 
\[\val{\varphi [!\overline{\nomj}/!\overline{z}, !\overline{\cnomm}/!\overline{w}]}_V\leq \val{\varphi [\overline{\rho}/!\overline{z}, \overline{\lambda}/!\overline{w}] }_V\leq \val{\psi [\overline{\rho}/!\overline{z}, \overline{\lambda}/!\overline{w}] }_V\leq \val{\psi [!\overline{\nomj}/!\overline{z}, !\overline{\cnomm}/!\overline{w}]}_V,
\]
as required. \\
Conversely, 
 let us argue contrapositively, and let  $V$ be an $\mathcal{L}^+$-valuation on $\mathbb{A}$ s.t. \[( \mathbb{A}, V)\not \models (\varphi\leq \psi) [\overline{\rho}/!\overline{z}, \overline{\lambda}/!\overline{w}], \] that is, $\val{\varphi[\overline{\rho}/!\overline{z}, \overline{\lambda}/!\overline{w}] }_V\nleq\val{\psi[\overline{\rho}/!\overline{z}, \overline{\lambda}/!\overline{w}] }_V$;
the assumptions on $\varphi$ and $\psi$ imply that
\[\val{\varphi[\overline{\rho}/!\overline{z}, \overline{\lambda}/!\overline{w}]}_V = \bigvee\{\val{\varphi} (\overline{j}/!\overline{z}, \overline{m}/!\overline{w})\mid  \overline{j\leq \val{\rho}_V}\ \&\ \overline{\val{\lambda}_V\leq m} \}\]
\[\val{\psi[\overline{\rho}/!\overline{z}, \overline{\lambda}/!\overline{w}]}_V = \bigwedge\{\val{\psi} (\overline{j}/!\overline{z}, \overline{m}/!\overline{w})\mid \overline{j\leq \val{\rho}_V}\ \&\ \overline{\val{\lambda}_V\leq m}\},\]
where each $j$ in $\overline{j}$ (resp.~$m$ in $\overline{m}$) belongs to the subset which completely join-generates (resp.~meet-generates) $\bba$, and e.g.~$\val{\varphi} (\overline{j}/!\overline{z}, \overline{m}/!\overline{w})$ denotes the element of $\bba$ obtained by applying the  term function $\val{\varphi (!\overline{z}, !\overline{w})}$ to $(\overline{j}, \overline{m})$.
Hence, $\val{\varphi[\overline{\rho}/!\overline{z}, \overline{\lambda}/!\overline{w}]}_V \nleq\val{\psi[\overline{\rho}/!\overline{z}, \overline{\lambda}/!\overline{w}] }_V$
 can  equivalently be rewritten as follows: 
\[\bigvee\{\val{\varphi} (\overline{j}/!\overline{z}, \overline{m}/!\overline{w})\mid \overline{j\leq \val{\rho}_V}\ \&\ \overline{\val{\lambda}_V\leq m} \}
\nleq  \bigwedge\{\val{\psi} (\overline{j}/!\overline{z}, \overline{m}/!\overline{w})\mid \overline{j\leq \val{\rho}_V}\ \&\ \overline{\val{\lambda}_V\leq m}\}.\]
By the definition of suprema and infima, and noticing that $\varphi(!\overline{z}, !\overline{w})$ and $\psi(!\overline{z}, !\overline{w})$ have disjoint sets of variables, given that each $z$ in $\overline{z}$ and $w$ in $\overline{w}$ can occur only once in $(\varphi\leq \psi) [!\overline{z}, !\overline{w}]$, this is equivalent to
\[\val{\varphi} (\overline{j}/!\overline{z}, \overline{m}/!\overline{w})\nleq \val{\psi} (\overline{j}/!\overline{z}, \overline{m}/!\overline{w})\]
 for some $\overline{j}$ and $\overline{m}$ such that $\overline{j\leq \val{\rho}_V}$ and $\overline{\val{\lambda}_V\leq m}$.
 Let $V^\ast$ be the $(\overline{\nomj}, \overline{\cnomm})$-variant of $V$ such that $\overline{V^\ast(\nomj) = j}$ and $\overline{V^\ast(\cnomm) = m}$. Then, since no (co)nominal variable in $\overline{\nomj}$ and $\overline{\cnomm}$ occurs in  any $\rho$ and $\lambda$, it follows that $\val{\varphi}_{V^\ast} = \val{\varphi}_{V} $ and $\val{\psi}_{V^\ast} = \val{\psi}_{V} $, and hence $(\bba, V^\ast)\models \overline{\nomj\leq \rho}$ and $(\bba, V^\ast)\models\overline{\lambda\leq\cnomm}$. However, $(\bba, V^\ast)\not \models (\varphi\leq \psi) [!\overline{\nomj}/!\overline{z}, !\overline{\cnomm}/!\overline{w}]$, as required.
\end{proof}
Instantiating $\overline{\rho}: = \overline{\top}$ and $\overline{\lambda}: = \overline{\bot}$ in the lemma above we get the following 
\begin{cor}[Eliminating  single occurrences of nominals and conominals]
Let $(\varphi\leq \psi) [!\overline{z}, !\overline{w}] $  be an $\mathcal{L}^+$-inequality. If $+\varphi(!\overline{z}, !\overline{w})$ and $-\psi(!\overline{w}, !\overline{z})$ are Skeleton formulas, and $+z\prec +\varphi$ and $+z\prec -\psi$ and $-w\prec +\varphi$ and $-w\prec -\psi$, then
 the following are equivalent for every $\mathcal{L}^+$-algebra  $\mathbb{A}$:

\begin{enumerate}
\item[(a)] $ \mathbb{A}\models (\varphi\leq \psi) [\overline{\top}/!\overline{z}, \overline{\bot}/!\overline{w}] $;

\item[(b)] 
$ \mathbb{A}\models \forall\overline{\nomj}\forall\overline{\cnomm}[ (\overline{\nomj\leq \top}\ \&\ \overline{\bot\leq\cnomm}) \Rightarrow (\varphi\leq \psi) [!\overline{\nomj}/!\overline{z}, !\overline{\cnomm}/!\overline{w}]]$, where all $\nomj$ in $\overline{\nomj}$ and $\cnomm$ in $\overline{\cnomm}$ are fresh.
\item[(c)] 
$ \mathbb{A}\models \forall\overline{\nomj}\forall\overline{\cnomm}( (\varphi\leq \psi) [!\overline{\nomj}/!\overline{z}, !\overline{\cnomm}/!\overline{w}])$, where all $\nomj$ in $\overline{\nomj}$ and $\cnomm$ in $\overline{\cnomm}$ are fresh.
\end{enumerate}
\end{cor}
Instantiating $\varphi: = z$ and $\psi: = w$ in the lemma above we get the following 
\begin{cor}[Approximating from  below (resp.~above) with nominals (resp.~conominals)]
\label{cor:firstapproxnomcnom}
 The following are equivalent for every $\mathcal{L}^+$-algebra  $\mathbb{A}$ and all $\mathcal{L}^+$-formulas $\rho$ and $\lambda$:

\begin{enumerate}
\item[(a)] $ \mathbb{A}\models \rho\leq \lambda$;

\item[(b)] 
$ \mathbb{A}\models \forall\nomj\forall \cnomm[( \nomj\leq \rho\ \&\ \lambda\leq\cnomm) \Rightarrow \nomj\leq \cnomm]$, where  $\nomj$ and $\cnomm$ are fresh.

\item[(c)] 
$ \mathbb{A}\models \forall\nomj( \nomj\leq \rho \Rightarrow \nomj\leq \lambda)$, where  $\nomj$ is fresh.
\item[(d)] 
$ \mathbb{A}\models \forall \cnomm( \lambda\leq\cnomm \Rightarrow \rho\leq \cnomm)$, where   $\cnomm$ is fresh.
\end{enumerate}
\end{cor}
In the next lemma, we are crucially making use of the environment of perfect DLEs, and hence that nominals and conominals are interpreted as completely join- and meet-{\em prime} elements.
\begin{lemma}[Existential Ackermann with nominals and conominals in the distributive setting]
\label{lemma:Existential_Ackermann_with_nominals_and_conominals}
Let $\mathcal{L}$ be a DLE-language. 
\begin{enumerate}
\item If $+\varphi(!\overline{z}, !\overline{w})$ is a definite Skeleton $\mathcal{L}^+$-formula, and $+z\prec +\varphi$ and $-w\prec +\varphi$, then
 the following are equivalent for every perfect $\mathcal{L}^+$-algebra  $\mathbb{A}$ and all $\mathcal{L}^+$-formulas $\rho$ in $\overline{\rho}$  and $\lambda$  in $\overline{\lambda}$:

\begin{enumerate}
\item[(a)] $ \mathbb{A}\models \nomi\leq \varphi [\overline{\rho}/!\overline{z}, \overline{\lambda}/!\overline{w}] $;

\item[(b)] 
$ \mathbb{A}\models \exists\overline{\nomj}\exists\overline{\cnomm}( \overline{\nomj\leq \rho}\ \&\ \overline{\lambda\leq\cnomm}\ \&\ \nomi\leq \varphi [!\overline{\nomj}/!\overline{z}, !\overline{\cnomm}/!\overline{w}])$, where all $\nomj$ in $\overline{\nomj}$ and $\cnomm$ in $\overline{\cnomm}$ are fresh.
\end{enumerate}
\item If  $-\psi(!\overline{w}, !\overline{z})$ is a definite Skeleton $\mathcal{L}^+$-formula, and $+z\prec -\psi$ and $-w\prec -\psi$, then
 the following are equivalent for every perfect $\mathcal{L}^+$-algebra  $\mathbb{A}$ and all $\mathcal{L}^+$-formulas $\rho$ in $\overline{\rho}$  and $\lambda$  in $\overline{\lambda}$:

\begin{enumerate}
\item[(a)] $ \mathbb{A}\models \psi [\overline{\rho}/!\overline{z}, \overline{\lambda}/!\overline{w}] \leq \cnomn$;

\item[(b)] 
$ \mathbb{A}\models \exists\overline{\nomj}\exists\overline{\cnomm}( \overline{\nomj\leq \rho}\ \&\ \overline{\lambda\leq\cnomm}\ \&\ \psi [!\overline{\nomj}/!\overline{z}, !\overline{\cnomm}/!\overline{w}]\leq \cnomn)$, where all $\nomj$ in $\overline{\nomj}$ and $\cnomm$ in $\overline{\cnomm}$ are fresh.
\end{enumerate}
\end{enumerate}
\end{lemma}
\begin{proof}
Let us show item 2, the proof of item 1 being order-dual. For the direction from (b) to (a), assume contrapositively that $\val{\psi [\overline{\rho}/!\overline{z}, \overline{\lambda}/!\overline{w}] }_V \nleq \val{\cnomn}_V$ for some $\mathcal{L}^+$-valuation $V$ on $\mathbb{A}$. Since $\val{\psi(!\overline{z}, !\overline{w})}$ is monotone in every $w$ in $\overline{w}$ and antitone in every $z$ in $\overline{z}$, the assumptions imply that, for any  $\{\overline{\nomj}, \overline{\cnomm}\}$-variant  $V^*$  of $V$ such that $(\mathbb{A}, V^*)\models  \overline{\nomj\leq \rho}$ and $(\mathbb{A}, V^*)\models\overline{\lambda\leq\cnomm}$, it must be that $(\mathbb{A}, V^*)\not \models \psi [!\overline{\nomj}/!\overline{z}, !\overline{\cnomm}/!\overline{w}]\leq \cnomn$, since otherwise $\val{\psi [\overline{\rho}/!\overline{z}, \overline{\lambda}/!\overline{w}] }_V  = \val{\psi [\overline{\rho}/!\overline{z}, \overline{\lambda}/!\overline{w}] }_{V^*} \leq \val{\psi [\overline{\nomj}/!\overline{z}, \overline{\cnomm}/!\overline{w}] }_{V^*}\leq \val{\cnomn}_{V^*} =  \val{\cnomn}_V$, against the assumption. This shows that (b) fails under $V$, as required. Conversely, fix an $\mathcal{L}^+$-valuation $V$ on $\mathbb{A}$; the assumptions on  $\psi$ imply that
\[\val{\psi[\overline{\rho}/!\overline{z}, \overline{\lambda}/!\overline{w}]}_V = \bigwedge\{\val{\psi} (\overline{j}/!\overline{z}, \overline{m}/!\overline{w})\mid \overline{j\leq \val{\rho}_V}\ \&\ \overline{\val{\lambda}_V\leq m}\},\]
where each $j$ in $\overline{j}$ (resp.~$m$ in $\overline{m}$) is  completely join-prime (resp.~meet-prime) and hence belongs to the subset which completely join-generates (resp.~meet-generates) $\bba$, and e.g.~$\val{\varphi} (\overline{j}/!\overline{z}, \overline{m}/!\overline{w})$ denotes the element of $\bba$ obtained by applying the  term function $\val{\varphi (!\overline{z}, !\overline{w})}$ to $(\overline{j}, \overline{m})$.
Hence, $\val{\varphi[\overline{\rho}/!\overline{z}, \overline{\lambda}/!\overline{w}]}_V \leq\val{\cnomn }_V$
 can  equivalently be rewritten as follows: 
\[ \bigwedge\{\val{\psi} (\overline{j}/!\overline{z}, \overline{m}/!\overline{w})\mid \overline{j\leq \val{\rho}_V}\ \&\ \overline{\val{\lambda}_V\leq m}\}\leq \val{\cnomn }_V.\]
Since $\val{\cnomn }_V$ is completely meet-prime, the inequality above implies that 
$\val{\psi} (\overline{j^*}/!\overline{z}, \overline{m^*}/!\overline{w}) \leq \val{\cnomn }_V $ for some vectors $\overline{j^*}$ and $\overline{m^*}$ such that  $\overline{j^*\leq \val{\rho}_V} $ and $ \overline{\val{\lambda}_V\leq m^*}$.  Let $V^\ast$ be the $(\overline{\nomj}, \overline{\cnomm})$-variant of $V$ such that $\overline{V^\ast(\nomj) = j^*}$ and $\overline{V^\ast(\cnomm) = m^*}$. By construction, $(\mathbb{A}, V^*)\models  \overline{\nomj\leq \rho}$ and $(\mathbb{A}, V^*)\models\overline{\lambda\leq\cnomm}$, and $(\mathbb{A}, V^*)\models \psi [!\overline{\nomj}/!\overline{z}, !\overline{\cnomm}/!\overline{w}]\leq \cnomn$, which proves that (b) holds under $V$, as required.
\end{proof}

Notice that, if $f(!\overline{z}, !\overline{w})\in \mathcal{F}^\ast$ (resp.~$g(!\overline{z}, !\overline{w})\in \mathcal{G}^\ast$) is such that $\epsilon_f(z) = 1$ (resp.~$\epsilon_g(z) = \partial$) for each $z$-coordinate and  $\epsilon_f(w) = \partial$ (resp.~$\epsilon_g(z) = 1$)  for each $w$-coordinate, then if $+ f(\overline{\varphi}/!\overline{z}, \overline{\psi}/!\overline{w})$ or $+( \varphi_1\wedge \varphi_2)$ (resp.~$-g(\overline{\varphi}/!\overline{z}, \overline{\psi}/!\overline{w})$ or $-(\psi_1\vee\psi_2)$) is definite Skeleton, then every $\varphi$ in $\overline{\varphi}$ is positive definite Skeleton, and every $\psi$ in $\overline{\psi}$ is negative definite Skeleton. Hence, by repeated applications of Lemma \ref{lemma:Existential_Ackermann_with_nominals_and_conominals} we get the following
\begin{cor}
\label{lemma:flattening_skeleton_formulas}
Let $\mathcal{L}$ be a DLE-language. 
\begin{enumerate}
\item If $+\varphi = \odot(\overline{\varphi}/!\overline{z}, \overline{\psi}/!\overline{w})$, where $\odot\in \mathcal{F}^\ast\cup\{\wedge\}$ is a pure definite Skeleton $\mathcal{L}^+$-formula, and $\epsilon_\odot(z) = 1$ (resp.~$\epsilon_g(z) = \partial$) for each $z$-coordinate and  $\epsilon_\odot(w) = \partial$ (resp.~$\epsilon_g(z) = 1$)  for each $w$-coordinate, then  
 the following are equivalent for every perfect $\mathcal{L}^+$-algebra  $\mathbb{A}$: 

\begin{enumerate}
\item[(a)] $ \mathbb{A}\models \nomi\leq \varphi$; 

\item[(b)] 
$ \mathbb{A}\models \exists\overline{\nomj}\exists\overline{\cnomm}( \overline{\nomj\leq \rho}\ \&\ \overline{\lambda\leq\cnomm}\ \&\ \nomi\leq \odot(\overline{\nomj}/!\overline{z}, \overline{\cnomm}/!\overline{w}))$, where all $\nomj$ in $\overline{\nomj}$ and $\cnomm$ in $\overline{\cnomm}$ are fresh and all inequalities are flat (modulo splitting if $\odot = \wedge$), and for every nominal $\nomj$ (resp.~conominal $\cnomm$) occurring in display in the left-hand (resp.~right-hand) side of an inequality, there exists a unique flat inequality relative to  which $\nomj$ (resp.~$\cnomm$) occurs in negative (resp.~positive) position in the scope of a $-f$-connective (resp.~$+g$-connective).
\end{enumerate}
\item If $-\psi = \odot(\overline{\varphi}/!\overline{z}, \overline{\psi}/!\overline{w})$, where $\odot\in \mathcal{G}^\ast\cup\{\vee\}$ is a pure definite Skeleton $\mathcal{L}^+$-formula, and $\epsilon_\odot(z) = \partial$  for each $z$-coordinate and  $\epsilon_\odot(w)  = 1$  for each $w$-coordinate, then  
 the following are equivalent for every perfect $\mathcal{L}^+$-algebra  $\mathbb{A}$: 

\begin{enumerate}
\item[(a)] $ \mathbb{A}\models \psi\leq \cnomn$; 

\item[(b)] 
$ \mathbb{A}\models \exists\overline{\nomj}\exists\overline{\cnomm}( \overline{\nomj\leq \rho}\ \&\ \overline{\lambda\leq\cnomm}\ \&\  \odot(\overline{\nomj}/!\overline{z}, \overline{\cnomm}/!\overline{w}))\leq \cnomn$, where all $\nomj$ in $\overline{\nomj}$ and $\cnomm$ in $\overline{\cnomm}$ are fresh and all inequalities are flat (modulo splitting if $\odot = \vee$), and for every nominal $\nomj$ (resp.~conominal $\cnomm$) occurring in display in the left-hand (resp.~right-hand) side of an inequality, there exists a unique flat inequality relative to which $\nomj$ (resp.~$\cnomm$) occurs in negative (resp.~positive) position in the scope of a $-f$-connective (resp.~$+g$-connective).
\end{enumerate}
\end{enumerate}
\end{cor}

\subsubsection{ALBA output shape}
\begin{prop}
\label{prop:ALBA_output_in_ackermann}
Let $\mathcal{L}$ be an arbitrary DLE-language, 
and \[(\varphi\leq \psi)[\overline{\alpha}/!\overline{x}, \overline{\beta}/!\overline{y},\overline{\gamma}/!\overline{z}, \overline{\delta}/!\overline{w}]\] be a definite  inductive $\mathcal{L}$-inequality with no uniform variables (and such that at least one vector among $\overline{\gamma}$ and $\overline{\delta}$ is nonempty, and all occurrences of $-\wedge$ and $+\vee$ in PIA formulas have $\epsilon$-critical branches on both coordinates). The output of ALBA is in Ackermann shape w.r.t.~the application of the reversed first approximation rule for the elimination of the nominal and conominal variables corresponding to the subformulas $\gamma$ in $\overline{\gamma}$ and $\delta$ in $\overline{\delta}$, and can hence be further transformed into a pure $\mathcal{L}^+$-inequality of the following form:
{{\begin{equation}
\label{eq: pure inequality:statement}
 \forall\overline{\nomj} \forall\overline{\cnomm}\left((\varphi\leq \psi)[!\overline{\nomj}/!\overline{x}, !\overline{\cnomm}/!\overline{y},\overline{\gamma}\left[\overline{\bigvee\mathsf{Mv}(p)}/\overline{p}, \overline{\bigwedge\mathsf{Mv}(q)}/\overline{q}\right]/!\overline{z}, \overline{\delta}\left [\overline{\bigvee\mathsf{Mv}(p)}/\overline{p}, \overline{\bigwedge\mathsf{Mv}(q)}/\overline{q}\right]/!\overline{w}] \right),\end{equation}
 }}
 which in its turn is the ALBA output of the following very simple Sahlqvist $\mathcal{L}^*$-inequality: 

 {{\begin{equation}
\label{eq: very simple Sahlqvist in reversive language}
\begin{array}{l}
 \forall\overline{p'} \forall\overline{q'}\Bigg((\varphi\leq \psi)\Big[!\overline{p'}/!\overline{x}, !\overline{q'}/!\overline{y},\overline{\gamma}\left[\overline{\bigvee\mathsf{Mv}(p)[\overline{p'}/\overline{\nomj}, \overline{q'}/\overline{\cnomm}]}/\overline{p}, \overline{\bigwedge\mathsf{Mv}(q)[\overline{p'}/\overline{\nomj}, \overline{q'}/\overline{\cnomm}]}/\overline{q}\right]/!\overline{z}, \\ \hfill\overline{\delta}\left [\overline{\bigvee\mathsf{Mv}(p)[\overline{p'}/\overline{\nomj}, \overline{q'}/\overline{\cnomm}]}/\overline{p}, \overline{\bigwedge\mathsf{Mv}(q)[\overline{p'}/\overline{\nomj}, \overline{q'}/\overline{\cnomm}]}/\overline{q}\right]/!\overline{w}\Big] \Bigg),
 \end{array}
 \end{equation}
 }}
 which is the inequality obtained by substituting $\overline{p'}$ for $\overline{\nomj}$ and $\overline{q'}$ for $\overline{\cnomm}$ in \eqref{eq: pure inequality:statement}.
\end{prop}
\begin{proof}
In what follows, we adapt  the ALBA run of $(\varphi\leq \psi)[\overline{\alpha}/!\overline{x}, \overline{\beta}/!\overline{y},\overline{\gamma}/!\overline{z}, \overline{\delta}/!\overline{w}]$ described in \cite[Section 1.6]{slanted}.  In the DLE-setting, we can drop the assumption that at least one vector among $\overline{\gamma}$ and $\overline{\delta}$ be nonempty, since we can assume w.l.o.g.~that this is the case. Indeed, notice preliminarily that, in the DLE-setting, we can assume w.l.o.g.~that no $-\vee$ or $+\wedge$ nodes occur in  each $\alpha$ in $\overline{\alpha}$ and $\beta$ in $\overline{\beta}$. Indeed, if this was not the case, then 
during pre-processing,  in each $\alpha$ in $\overline{\alpha}$ and $\beta$ in $\overline{\beta}$, ALBA exhaustively distributes $-f\in \mathcal{F}$ over $-\vee$ in its positive coordinates and over $+\wedge$ in its negative coordinates, and  $+g\in \mathcal{G}$ over $+\wedge$ in its positive coordinates and over $-\vee$ in its negative coordinates, so as to bring occurrences of $-\vee$ and $+\wedge$ as close as possible to the root of each PIA subformula. In the DLE-setting, these  nodes are SLR, and can be then incorporated into the Skeleton of the input inequality. 
 Under this additional assumption, the inductive shape implies that each $\alpha$ in $\overline{\alpha}$ and $\beta$ in $\overline{\beta}$ contains exactly one critically occurring proposition variable (in the LE-setting, the same situation is reached modulo application of splitting rules; however, this will entail the duplication of the corresponding nominal/conominal variables in the antecedent of the quasi-inequality).
 Now for the proof of the main claim, assume for contradiction that  both $\overline{\gamma}$ and $\overline{\delta}$ were empty. Then,  
 all maximal $\epsilon^\partial$-subformulas occur as subformulas of some PIA-subformulas. Hence,  because there are no occurrences of $-\vee$ and $+\wedge$ in any PIA-subformula, the separation, within every PIA-formula, between $\epsilon$-critical branches and maximal $\epsilon^\partial$-uniform subformulas must be effected by non-unary SRR-connectives, which  creates $<_\Omega$-dependence relations among variables. In particular, if all maximal $\epsilon^\partial$-subformulas occur as subformulas of some PIA-subformulas, then every atomic proposition has some $<_\Omega$-successor variable, which contradicts the well-foundedness of the strict partial order $<_\Omega$. 

By Lemma \ref{lemma:Universal_Ackermann_with_nominals_and_conominals} (cf.~\cite[Section 6]{CoPa:non-dist} and \cite[Section 5]{CoPa-constructive}),
 ALBA transforms 
\[(\varphi\leq \psi)[\overline{\alpha}/!\overline{x}, \overline{\beta}/!\overline{y},\overline{\gamma}/!\overline{z}, \overline{\delta}/!\overline{w}]
\]
into the following \emph{initial quasi-inequality}:
 
\begin{equation}
\label{eq:initial_quasi-inequality_p_and_q}
\begin{array}{l}
 \forall\overline{\nomj} \forall\overline{\cnomm}\forall\overline{\nomi}\forall\overline{\cnomn}\Big(( \overline{\nomj}\leq \overline{\alpha_p}\ \&\  \overline{\nomj}\leq \overline{\alpha_q}\ \&\ \overline{\beta_p}\leq \overline{\cnomm} \ \&\ \overline{\beta_q}\leq \overline{\cnomm} \ \&\ \overline{\nomi}\leq \overline{\gamma} \ \&\ \overline{\delta}\leq \overline{\cnomn})\Rightarrow \hspace{1cm} \\ 
 \hfill(\varphi\leq \psi)[!\overline{\nomj}/!\overline{x}, !\overline{\cnomm}/!\overline{y},!\overline{\nomi}/!\overline{z}, !\overline{\cnomn}/!\overline{w}] \Big),
 \end{array}
\end{equation}
where $\overline{p}$ (resp.~$\overline{q}$) is the vector of the atomic propositions in $\varphi\leq \psi$ such that $\varepsilon(p) = 1$ (resp.~$\varepsilon(q) = \partial$), and 
the subscript in each PIA-formula in $\overline{\alpha}$ and $\overline{\beta}$ indicates the unique $\varepsilon$-critical propositional variable occurrence contained in that formula.

Next, we enter the reduction/elimination stage,  aimed at eliminating all occurring propositional variables. 
By applying adjunction and residuation rules on all PIA-formulas $\alpha$ and $\beta$, the antecedent of \eqref{eq:initial_quasi-inequality_p_and_q} can be equivalently written as follows (cf.~Definition \ref{def:RA_and_LA}):  
\begin{equation}
\label{eq:adj_anind}
\begin{array}{rclcl}
\overline{\mathsf{LA}(\alpha_p)[\nomj/u, \overline{p},\overline{q}]}\leq\overline{p} &  \& &   \overline{\mathsf{RA}(\beta_p)[\cnomm/u, \overline{p},\overline{q}]}\leq \overline{p} & \& & \overline{\nomi}\leq \overline{\gamma}\ \  \& \\
\overline{q}\leq \overline{\mathsf{LA}(\alpha_q)[\nomj/u, \overline{p},\overline{q}]} & \& & \overline{q}\leq\overline{\mathsf{RA}(\beta_q)[\cnomm/u, \overline{p},\overline{q}]} & \& &  \overline{\delta}\leq \overline{\cnomn}.
\end{array}
\end{equation}
Notice that the `parametric' (i.e.~non-critical) variables in $\overline{p}$ and $\overline{q}$ actually occurring in each formula $\mathsf{LA}(\alpha_p)[\nomj/u, \overline{p},\overline{q}]$, $\mathsf{RA}(\beta_p)[\cnomm/u, \overline{p},\overline{q}]$, $\mathsf{LA}(\alpha_q)[\nomj/u, \overline{p},\overline{q}]$, and $\mathsf{RA}(\beta_q)[\cnomm/u, \overline{p},\overline{q}]$ are those that are strictly $<_\Omega$-smaller than the (critical) variable indicated in the subscript of the given PIA-formula. After applying adjunction and residuation as indicated above, the resulting quasi-inequality is in Ackermann shape relative to the $<_\Omega$-minimal variables.

For every $p\in\overline{p}$ and  $q\in\overline{q}$ let us define the sets $\mathsf{Mv}(p)$ and $\mathsf{Mv}(q)$ by recursion on $<_\Omega$ as follows:
\begin{itemize}
	\item
	$\mathsf{Mv}(p):=\{\mathsf{LA}(\alpha_p)[\nomj_k/u,\overline{\mathsf{mv}(p)}/\overline{p},\overline{\mathsf{mv}(q)}/\overline{q}], \mathsf{RA}(\beta_p)[\cnomm_h/u,\overline{\mathsf{mv}(p)}/\overline{p},\overline{\mathsf{mv}(q)}/\overline{q}]\mid 1\leq k\leq n_{i_1}, 1\leq h\leq n_{i_2}, \overline{\mathsf{mv}(p)}\in\overline{\mathsf{Mv}(p)},\overline{\mathsf{mv}(q)}\in\overline{\mathsf{Mv}(q)}  \}$
	\item $\mathsf{Mv}(q):=\{\mathsf{LA}(\alpha_q)[\nomj_h/u,\overline{\mathsf{mv}(p)}/\overline{p},\overline{\mathsf{mv}(q)}/\overline{q}], \mathsf{RA}(\beta_q)[\cnomm_k/u,\overline{\mathsf{mv}(p)}/\overline{p},\overline{\mathsf{mv}(q)}/\overline{q})\mid 1\leq h\leq m_{j_1}, 1\leq k\leq m_{j_2}, \overline{\mathsf{mv}(p)}\in\overline{\mathsf{Mv}(p)},\overline{\mathsf{mv}(q)}\in\overline{\mathsf{Mv}(q)}  \}$
\end{itemize}
where,  $n_{i_1}$ (resp.~$n_{i_2}$) is the number of occurrences of $p$ in $\alpha$s (resp.~in $\beta$s) for every $p\in\overline{p}$, and $m_{j_1}$ (resp.~$m_{j_2}$) is the number of occurrences of $q$ in $\alpha$s (resp.~in $\beta$s) for every $q\in\overline{q}$. 
By induction on $<_\Omega$, we can apply the Ackermann rule exhaustively so as to eliminate all variables $\overline{p}$  and  $\overline{q}$.  Then, the resulting {\em purified} quasi-inequality has the following form: 
 
\begin{equation}
\label{eq:after_Ackermann_anind}
\begin{array}{l}
 \forall\overline{\nomj} \forall\overline{\cnomm}\forall\overline{\nomi}\forall\overline{\cnomn}\Bigg((\overline{\nomi}\leq \overline{\gamma}\left[\overline{\bigvee\mathsf{Mv}(p)}/\overline{p}, \overline{\bigwedge\mathsf{Mv}(q)}/\overline{q}\right]\ \&\ \overline{\delta}\left [\overline{\bigvee\mathsf{Mv}(p)}/\overline{p}, \overline{\bigwedge\mathsf{Mv}(q)}/\overline{q}\right]\leq \overline{\cnomn})\Rightarrow \\ \hfill(\varphi\leq \psi)[!\overline{\nomj}/!\overline{x}, !\overline{\cnomm}/!\overline{y},!\overline{\nomi}/!\overline{z}, !\overline{\cnomn}/!\overline{w}] \Bigg),
 \end{array}
 \end{equation}
 The inequality above is in Ackermann shape relative to $\overline{\nomi}$ and $\overline{\cnomn}$. 
 Hence, by applying Lemma \ref{lemma:Universal_Ackermann_with_nominals_and_conominals} in the bottom-up direction we can eliminate these variables and obtain the following required inequality:
 \begin{equation}
\label{eq:pure_inequality}
 \forall\overline{\nomj} \forall\overline{\cnomm}\left((\varphi\leq \psi)[!\overline{\nomj}/!\overline{x}, !\overline{\cnomm}/!\overline{y},\overline{\gamma}\left[\overline{\bigvee\mathsf{Mv}(p)}/\overline{p}, \overline{\bigwedge\mathsf{Mv}(q)}/\overline{q}\right]/!\overline{z}, \overline{\delta}\left [\overline{\bigvee\mathsf{Mv}(p)}/\overline{p}, \overline{\bigwedge\mathsf{Mv}(q)}/\overline{q}\right]/!\overline{w}] \right).\end{equation}
 \end{proof}

\subsection{Problems in generalizing Kracht formulas}

The original Kracht's theorem is a model theoretic argument which has many ties with the classical setting, thus rendering it unsuitable for direct generalisation. In this subsection we discuss some of the problems in porting Kracht formulae to the general setting, besides the trivial ones (such as the fact that the original definition is tightly bound to a specific semantic).

\begin{definition}[Kracht formulas]
\label{def:originalkracht}
A \emph{Kracht formula}\footnote{The definition we present is commonly referred to as \emph{type 1 Kracht formula}. As is well known (cf. \cite{blackburn2002modal}), type 1 Kracht formulas are Kracht formulas in prenex normal form where the matrix is rewritten in DNF.} is a clean $\mathrm{FO}$-formula in prenex normal form with a single free variable $x_0$ and shape:
\[
\forall^R x_1 \cdots \forall^R x_n Q_1^R y_1 \cdots Q_m^R y_m \ \beta(x_0, x_1, \ldots, x_n, y_1, \ldots, y_m)\,,
\]
where $Q_i \in \{\forall,\exists\}$ (for $1 \leq i \leq m$), variables in $X = \{x_1, \ldots, x_n\}$ and $Y=\{y_1,\ldots,y_m\}$ are called \emph{inherently universal} and \emph{non-inherently universal} respectively; $\beta$ is an unquantified formula in DNF whose atoms are of the form: $\top$, $\bot$, $uRx$, $xRu$, $x=u$ where $x\in X \cup \{x_0\}$ and $u\in X\cup Y$.
\end{definition}

\begin{theorem}[{\cite{blackburn2002modal}}] 
\label{thm:originalkracht}
Any Kracht formula can be effectively shown to be the first order correspondent of some Sahlqvist formula.
\end{theorem}

\begin{example}[\cite{inverse_correspondence_tense}]
\label{ex:goranko_backward_looking} 
The inequality $p\wedge \Box(\Diamond p\rightarrow \Box q)\leq \Diamond\Box\Box q$  (cf. \cite{GoVa2006}) is not Sahlqvist for any order type, but it is inductive w.r.t.~the order-type $\varepsilon(p, q) = (1, 1)$ and $p <_\Omega q$. Running ALBA on it yields
\begin{center}
\begin{tabular}{c l}
& $\forall p\forall q(p\wedge \Box(\Diamond p\rightarrow \Box q)\leq \Diamond\Box\Box q)$\\
iff & $\forall p\forall q\forall \nomj\forall \cnomm[(\nomj\leq p\wedge \Box(\Diamond p\rightarrow \Box q) \ \&\  \Diamond\Box\Box q\leq \cnomm) \Rightarrow \nomj\leq \cnomm]$\\
iff & $\forall p\forall q\forall \nomj\forall \cnomm[(\nomj\leq p\ \&\  \nomj\leq\Box(\Diamond p\rightarrow \Box q) \ \&\  \Diamond\Box\Box q\leq \cnomm) \Rightarrow \nomj\leq \cnomm]$\\
iff & $\forall q\forall \nomj\forall \cnomm[(\nomj\leq\Box(\Diamond \nomj\rightarrow \Box q) \ \&\  \Diamond\Box\Box q\leq \cnomm) \Rightarrow \nomj\leq \cnomm]$\\
iff & $\forall q\forall \nomj\forall \cnomm[(\Diamondblack(\Diamond \nomj\wedge \Diamondblack\nomj)\leq q \ \&\  \Diamond\Box\Box q\leq \cnomm) \Rightarrow \nomj\leq \cnomm]$\\
iff & $\forall \nomj\forall \cnomm[\Diamond\Box\Box \Diamondblack(\Diamond \nomj\wedge \Diamondblack\nomj)\leq \cnomm\Rightarrow \nomj\leq \cnomm]$\\
iff & $\forall \nomj[\nomj\leq \Diamond\Box\Box \Diamondblack(\Diamond \nomj\wedge \Diamondblack\nomj)]$
\end{tabular}
\end{center}
In the classical setting, as the nominal $\nomj$ represents a world $x$ of the Kripke frame, this can be rewritten as follows:
\begin{center}
\begin{tabular}{cl}
& $\forall x(x\in\val{ \Diamond\Box\Box \Diamondblack(\Diamond \nomj\wedge \Diamondblack\nomj)}[\nomj: = x])$\\
iff & $\forall x\exists y(xR y\ \&\ y\in\val{\Box\Box \Diamondblack(\Diamond \nomj\wedge \Diamondblack\nomj)}[\nomj: = x])$\\
iff & $\forall x\exists y(xR y\ \&\ \forall z(yR^2 z\Rightarrow z\in\val{\Diamondblack(\Diamond \nomj\wedge \Diamondblack\nomj)}[\nomj: = x]))$\\
iff & $\forall x\exists y(xR y\ \&\ \forall z(yR^2 z\Rightarrow \exists w(wRz\ \&\ wRx\ \&\ xRw)))$.\\
\end{tabular}
\end{center}
This last condition can equivalently be rewritten in three ways:
\begin{center}
\begin{tabular}{c l}
 & $\forall x(\exists y\rhd x)( \forall z_1\rhd y)(\forall z \rhd z_1)( \exists w\blacktriangleright z)(wRx\ \&\ xRw)$\\
iff & $\forall x(\exists y\rhd x)( \forall z_1\rhd y)(\forall z \rhd z_1)( \exists w\rhd x)(wRz\ \&\ wRx)$\\
iff & $\forall x(\exists y\rhd x)( \forall z_1\rhd y)(\forall z \rhd z_1)( \exists w\blacktriangleright x)(wRz\ \&\ xRw)$,\\
\end{tabular}
\end{center}
where $(\exists w \blacktriangleright z)$ quantifies $w$ over the \emph{predecessors} of $z$. The second and third ones are not Kracht formulas, as the atom $wRz$ has no inherently universal variables in it (the only inherently universal is $x$). The first one is a \emph{tense} Kracht formula. This consideration suggests that in order to express first-order conditions in Kracht shape for an inductive formula, we need to admit the presence of the adjoint operators, thus allowing for backwards looking restricted quantifiers. 
\end{example}

%% file: 03cryptoinductive.tex
\begin{definition}
\label{def:adj:res:conservative}
	An SRA node $+g$ (resp.\ $-f$) is \emph{adjunction-conservative} if the adjoint $g^{\flat}$ of $g$ (resp.\ $f^{\sharp}$ of $f$) belongs to $\mathcal{L}$. Similarly, an SRR node $+g$ (resp.\ $-f$) is \emph{residuation-conservative in the $i$-th coordinate} for $1 \leq i \leq n_g$ ($1 \leq i \leq n_f$) if the residual $g_i^{\flat}$ of $g$ (resp.\ $f_i^{\sharp}$ of $f$) belongs to $\mathcal{L}$.
\end{definition}

\begin{definition}	
\label{def:inner:splitable:branch}
	A branch in a signed generation tree $\pm s$ is called \emph{splittable} if it is the concatenation of two paths $P_1$ and $P_2$, one of which may possibly be of length $0$, such that $P_1$ is a path from the leaf consisting (apart from variable nodes) only of adjunction-conservative SRA-nodes and SRR-nodes which are residuation-conservative in the coordinate through which $P_1$ passes, and $P_2$ consists of (any) $\mathcal{L}$-nodes. 
\end{definition}

\begin{definition}
\label{def:epsomega:unpackable}
Given an order type $\varepsilon$, a strict partial order $\Omega$ on propositional variables, a signed generation tree $\pm \phi$ is called \emph{$(\Omega, \varepsilon)$-unpackable} if $\varepsilon^{\partial}(\pm \phi)$ and
	\begin{enumerate}
	    \item $\phi$ is a propositional variable or constant, or
		\item among $var(\phi)$\footnote{With $var(\phi)$ we denote the collection of variables in $\atprop$ occurring in $\phi$.} there is a variable $p_0$ maximal with respect to $\Omega$, such that
			\begin{enumerate}
			\item the path $P$ in $\pm \phi$ ending in $p_0$ is splittable, 
			\item if $P$ passes through the $j$th coordinate of an $m$-ary SRR-node \[\circledast(\gamma_1,\dots,\gamma_{j-1},\beta,\gamma_{j+1}\ldots,\gamma_m), \] then the signed subtree $\pm \gamma_i$ is $(\Omega, \epsilon)$-unpackable for all $i \neq j$.
			\end{enumerate}	
	\end{enumerate}
\end{definition}

\begin{definition}\label{def:Crypto:inductive}
	An $\mathcal{L}^{\ast}$ inequality $\phi \leq \psi$ is called a \emph{crypto $\mathcal{L}$-inductive} if it is a  very simple $\epsilon$-Sahlqvist formula in $\mathcal{L}^{\ast}$ and in the signed generation trees $+\phi$ and $- \psi$:
	\begin{enumerate}
		\item All $\epsilon$-critical branches contain only signed connectives from $\mathcal{L}$,
		\item There exists a strict partial order $\Omega$ on the propositional variables occurring in $\phi \leq \psi$, such that for every  $\epsilon$-non-critical branch the signed subtree rooted at the topmost (closest to the root) node on the branch properly belonging to $\mathcal{L}^{\ast}$ is $(\Omega, \epsilon)$-unpackable. 
	\end{enumerate}
\end{definition}

\begin{example}
	Consider the language $\mathcal{L}_{0}$ of basic (classical) modal logic. Then the $\mathcal{L}_{0}^{*}$ inequality $\Diamondblack p \leq \blacksquare p$ is very simple Sahlqvist, and is valid at a point in a Kripke frame iff that point has at most one predecessor. However,  as that condition is not invariant under generated subframes (for the language $\mathcal{L}_{0}$), it is not definable by a  $\mathcal{L}_{0}$ formula. Indeed, $\Diamondblack p \leq \blacksquare p$ is not crypto $\mathcal{L}_0$-inductive, as any choice of $\epsilon$ will result in an $\epsilon$-critical branch with a node not in $\mathcal{L}_{0}$. 
\end{example}

\begin{prop}
\label{prop:cryptotoinductive}
	Every crypto $\mathcal{L}^{\ast}$-inductive inequality is equivalent to an inductive formula in $\mathcal{L}$. Moreover, the latter can be obtained from the former by the application of ALBA-rules.
\end{prop}
\begin{proof}
	Suppose that $\phi \leq \psi$ is  a $\mathcal{L}^{\ast}$-inductive inequality and let $\epsilon$ and $\Omega$ be an order type and a strict partial order satisfying Definition \ref{def:Crypto:inductive}. Suppose that $var(\phi \leq \psi) = \{p_1, \ldots p_n\}$. We may assume w.l.o.g.\ that $p_i <_{\Omega} p_j$ implies $i < j$. Starting from propositional variables $p_i$ minimal with respect to $\Omega$, apply the inverse Ackermann rules to extract the subformulas corresponding to the subtree rooted at the topmost (closest to the root) node on the branch properly belonging to $\langresidual$. This transforms $\phi \leq \psi$ into a quasi-inequality of the form 
    \[
    q_1 \leq \alpha_1, \ldots, q_m \leq \alpha_m, \beta_{m+1} \leq r_{m+1}, \ldots, \beta_{\ell} \leq r_{\ell} \Rightarrow (\phi'\leq \psi')[\overline{q}/!\overline{x}, \overline{r}/!\overline{y}]
    \]
    where $\phi'\leq \psi'$ contains only connectives from $\langbase$, the $q_i$ and $r_i$ are new variables, for each $\alpha_i$ contains exactly one variable among $p_1, \ldots p_n$ (namely that which was $\Omega$-maximal in the extracted subtree from which $\alpha$ originates. Applying adjunction and residuation rules this can be transformed into 
    \[
    \begin{array}{lcl}
    \mathsf{LA}(\alpha_1)(q_1) \leq_{\varepsilon(p_{i_1})} p_{i_1}, \ldots, 
    \mathsf{LA}(\alpha_m)(q_m) \leq_{\varepsilon(p_{i_m})} p_{i_m}, & & \\ 
    p_{i_{m+1}} \leq_{\varepsilon(p_{i_{m+1}})} \mathsf{RA}(\beta_1)(r_{m+1}), \ldots, 
    p_{i_{\ell}} \leq_{\varepsilon(p_{i_{\ell}})} \mathsf{RA}(\beta_{\ell})(r_{\ell}) 
    & \Rightarrow & 
    (\phi'\leq \psi')[\overline{q}/!\overline{x}, \overline{r}/!\overline{y}]
    \end{array}
    \]
    Note that each $\mathsf{LA}(\alpha_i)(q_i)$ and $\mathsf{RA}(\beta_{i})(r_{i})$ is a $\langbase$-formula. 
    
    This is now in Ackermann-shape w.r.t.\ the variables $p_1, \ldots p_n$. Applying the Ackermann rules produces   
    
    \[
    (\phi'\leq \psi')[\overline{q}/!\overline{x}, \overline{r}/!\overline{y}, \overline{\xi}/\overline{p}]
    \]
    
    which is an $(\Omega', \epsilon)$inductive inequality in $\langbase$ where $<\Omega'$
\end{proof}

\begin{lemma}
\label{lemma:inductivetocrypto}
Every (definite) $\langbase$-inductive inequality is equivalent to some crypto $\langresidual$-inductive inequality.
\end{lemma}
\begin{proof}
Given a definite inductive formula $(\varphi \leq \psi)[\overline\alpha/\overline x, \overline\beta/\overline y, \overline\gamma / \overline z, \overline\delta/\overline w]$, by Proposition \ref{prop:ALBA_output_in_ackermann} running ALBA on it yields
\[
\begin{array}{l}
 \forall\overline{\nomj} \forall\overline{\cnomm}\forall\overline{\nomi}\forall\overline{\cnomn}\Big(
 \overline{\nomi}\leq \overbrace{\overline{\gamma}\left[\overline{\bigvee\mathsf{Mv}(p)}/\overline{p}, \overline{\bigwedge\mathsf{Mv}(q)}/\overline{q}\right]}^{\overline{\gamma}^{mv}} \ \&\ 
\overbrace{\overline{\delta}\left [\overline{\bigvee\mathsf{Mv}(p)}/\overline{p}, \overline{\bigwedge\mathsf{Mv}(q)}/\overline{q}\right]}^{{\overline\delta}^{mv}}\leq \overline{\cnomn}
\Rightarrow \\ \hfill (\varphi\leq \psi)[!\overline{\nomj}/!\overline{x}, !\overline{\cnomm}/!\overline{y},!\overline{\nomi}/!\overline{z}, !\overline{\cnomn}/!\overline{w}] \Big),
 \end{array}
\]
Consider now the inequality
\begin{equation} 
\label{eqn:verysimpleshape}
\left(
(\varphi\leq \psi)
[!\overline{\nomj}/!\overline{x}, 
!\overline{\cnomm}/!\overline{y},
!\overline{\gamma}^{mv}/!\overline{z}, 
!\overline{\delta}^{mv}/!\overline{w}] 
\right)
[\overline{p_j}/\overline{\nomj}, 
\overline{p_m}/\overline{\cnomm}],
\end{equation}
with $\overline{p_j}$ (resp.\ $\overline{q_m}$) fresh variables, one for each nominal in $\overline\nomj$ (resp.\ conominal in $\overline\cnomm$). Clearly, the inequality is very simple Sahlqvist in $\langresidual$ for $\varepsilon$ such that each for each $p_j$ (resp.\ $q_m$), $\varepsilon(p_j) = 1$ (resp.\ $\varepsilon(q_m)=\partial$) and some inductive order type $<_\Omega$. More precisely, in the ALBA run each PIA in $\overline\alpha$ (resp.\ $\overline\beta$) is approximated by some nominal in $\overline\nomj$ (resp.\ $\overline\cnomm$), let $\tau$ be the map that given a variable in $\overline\nomj$ and $\overline\cnomm$, yields the critical variable in the corresponding PIA formula. Let $<_{\Omega'}$ the inductive order used in the ALBA run. The inequality \eqref{eqn:verysimpleshape} is very simple Sahlqvist for the inductive order $<_\Omega$ such that for every $r,t \in \{p,q\}$ and $\pureu,\purev$ in $\overline\nomj$ or $\overline\cnomm$,
\[
r_u \leq t_v \quad \quad \text{iff} \quad \quad \tau(u) \leq \tau(v). 
\]
It is clear how the $\varepsilon$-critical branches contain only connectives in $\langbase$, as the only connectives found there are the ones found in the skeleton of the original inductive inequality. Hence, it remains to show condition (2) of Definition \ref{def:Crypto:inductive}. We show that every signed subtree rooted at the topmost connective properly in $\langresidual$ is $(\varepsilon,\Omega)$-unpackable for the $\varepsilon$ and $\Omega$ defined above. Since the operators properly belonging to $\langresidual$ can only occur in the minimal valuations, any of such nodes has to occur in some formula in $\overline\gamma^{mv}$ or $\overline\delta^{mv}$ inside some $\mathsf{Mv}(p)$ (resp.\ $\mathsf{Mv}(q)$) for some variable $p$ (resp.\ $q$); hence the paths passing through these nodes ending in $<_\Omega$-maximal variables are splittable. Suppose that one of such paths passes through the $j$th coordinate of some $m$-ary SRR node in a non $\varepsilon$-critical branch
\[
\circledast(\gamma_1,\dots,\gamma_{j-1},\beta,\gamma_{j+1}\ldots,\gamma_m), 
\]
and let $i$ be any index in $\{1,\ldots,m\}\setminus\{j\}$. The subformula $\gamma_i$ is $(\varepsilon,\Omega)$-unpackable as it is of course splittable, and, inductively, every topmost node not properly in $\langresidual$ is part of some $\mathsf{Mv}(p')$ (resp.\ $\mathsf{Mv}(q')$) for some $p'$ (resp.\ $q'$) preceding $p$ (resp.\ $q$).
\end{proof}

%% file: 04inverse.tex
In this section, we define Kracht formulas for arbitrary DLE-languages and semantics, and then show an algorithm
for inverse correspondence targeting the class of inductive inequalities in such languages. In the remainder of the section we will refer to the set of nominals $\nomi,\nomj,\nomh,\nomk,\ldots$ and $\lambda(\cnoml),\lambda(\cnomm),\lambda(\cnomn),\lambda(\cnomo),\ldots$ as $\nomset$. Respectively, the set of conominals $\cnoml,\cnomm,\cnomn,\cnomo,\ldots$ and $\kappa(\nomi),\kappa(\nomj),\kappa(\nomh),\kappa(\nomk),\ldots$ as $\cnomset$. We also make use of the function $\dneg:\nomset \cup \cnomset \to \nomset \cup \cnomset$ defined in the following way:
\[
\dneg(\purev) =
\begin{cases}
\kappa(\purev) & \text{if} \;\, \purev \in \nomset \\
\lambda(\purev) & \text{if} \;\, \purev \in \cnomset,
\end{cases}
\]
and we will often write $\dneg\purev$ instead of $\dneg(\purev)$. Across the whole section we fix a DLE-language $\langbase(\mathcal{F},\mathcal{G})$.

\subsection{Kracht's formulas in ALBA language}

\begin{definition}[Flat and restricting inequalities]
\label{def:restricting_inequality}
A pure $\langalba$-inequality is \emph{flat} if it is of the form:
\begin{center}
\begin{tabular}{ccccc}
 $\nomi \leq f(\overline{\purew_1})$, & 
 $\nomi \leq g(\overline{\purew_1})$, & 
 $f(\overline{\purew_2}) \leq \cnomn$, &
 $g(\overline{\purew_2}) \leq \cnomn$, & 
 $\pureu \leq \purev$; \\
\end{tabular}
\end{center}
with $\pureu,\purev \in \nomset \cup \cnomset$, $\overline{\purew_1}\in(\nomset \cup \cnomset \setminus \{\nomi\})^\omega$,  $\overline{\purew_2}\in(\nomset \cup \cnomset \setminus \{\cnomn\})^\omega$, $f \in \mathcal{F}^*$, and $g \in \mathcal{G}^*$. Pure variables $\nomi,\cnomn,\pureu,\purev$ are \emph{on display} in the flat inequality. \emph{Restricting inequalities} are flat $\mathcal{L}^+$ inequalities of shape:
\[
\nomi \leq f(\overline{\nomj},\overline{\cnomm}), \ \ \ \ \ g(\overline{\cnomm},\overline{\nomj}) \leq \cnomn, \ \ \ \ \ \nomi  \leq \nomh, \ \ \ \ \  \cnomo \leq \cnomn;
\]
where $f \in \mathcal{F}^*$ (resp. $g \in \mathcal{G}^*$), and, without loss of generality, the first $|\overline{\nomj}|$ (resp.\ $|\overline{\cnomm}|$) coordinates of $f$ (resp. $g$) are the positive ones.
For any restricted inequality of shape $\nomi \leq f(\overline{\nomj},\overline{\cnomm})$ or $\nomi  \leq \nomh$ (resp.~$g(\overline{\cnomm},\overline{\nomj}) \leq \cnomn$ or $\cnomo \leq \cnomn$) , $\nomi$ (resp.~$\cnomn$) is the \emph{restrictor} of $\overline\nomj$ and $\overline\cnomm$.
\end{definition}

In the classical setting, restricting inequalities encode, among the other things, relational and equality atoms. Indeed, in the complex algebra of a Kripke frame $\mathbb{K} = (W,R)$, nominals (resp.\ conominals) are interpreted singletons (resp.\ cosingletons) $\{x\}, \{y\}, \ldots \subseteq W$ (resp.\ $\{x\}^c, \{y\}^c, \ldots \subseteq W$); hence the following equivalences hold\footnote{In the right column of the table, we write $x,y,z$ instead of $\{x\}, \{y\}, \{z\}$.}:

\[
\begin{array}{r||clclclcl}
x R y  & x \leq  \Diamond y & \mbox{iff} &\Box y^c  \leq  x^c & \mbox{iff} & y \leq \Diamondblack x & \mbox{iff} & \blacksquare x^c \leq y^c \, \\
x=y & x \leq y & \mbox{iff} & y^c \leq x^c & \mbox{iff} & x \nleq y^c & & \,\\
x=y=z & x \rightarrow y^c \leq  z^c & \mbox{iff} & x \leq y \pdla z^c.
\end{array}
\]

\begin{example}
\label{ex:classicalexamplerewritenomcnom}
The first order condition expressing confluence is a well-known Kracht formula which can be written in the following way using Kracht's notation:
\[
\forall z (\forall x\rhd z)(\forall y\rhd z)(\exists w\rhd x)(yRw).
\] 
Using the equivalences above we can translate it in the language of ALBA as follows:
\begin{center}
\begin{tabular}{c l}
& $\forall z (\forall x\rhd z)(\forall y\rhd z)(\exists w\rhd x)(yRw)$\\
iff & $\forall x\forall z (zRx\Rightarrow \forall y(zRy\Rightarrow \exists w(xRw\ \&\ yRw)))$\\
iff & $\forall x\forall z (z \leq \Diamond x \Rightarrow \forall y(y \leq \Diamondblack z \Rightarrow \exists w(w\leq \Diamondblack x\ \&\ y\leq \Diamond w)))$\\
i.e. & $\forall \nomj\forall \nomi (\nomi \leq \Diamond \nomj \Rightarrow \forall \nomh(\nomh \leq \Diamondblack \nomi \Rightarrow \exists \nomk(\nomk\leq \Diamondblack \nomj\ \&\ \nomi\leq \Diamond \nomk)))$
\end{tabular}
\end{center}
\end{example} 

\begin{notation}
\label{notation:restrictedquant}
To better reflect the original Kracht's notation, it is useful to use the following notation for \emph{restricted quantifiers} in $\langmeta$:
\[
\begin{array}{rcl}
(\forall \overline\nomi,\overline\cnomn \rhd_f \nomj)\beta & \equiv  & (\forall \overline\nomi,\overline\cnomn)(\nomj \leq f(\overline\nomi,\overline\cnomn) \Rightarrow \beta)\,, \\
(\forall \overline\nomi,\overline\cnomn \rhd_g \cnomm)\beta & \equiv & (\forall \overline\nomi,\overline\cnomn)(g(\overline\nomi,\overline\cnomn) \leq \cnomm \Rightarrow \beta)\,, \\
(\exists \overline\nomi,\overline\cnomn \rhd_f \nomj)\beta & \equiv  & (\exists \overline\nomi,\overline\cnomn)(\nomj \leq f(\overline\nomi,\overline\cnomn) \ \metaand \ \beta)\,, \\
(\exists \overline\nomi,\overline\cnomn \rhd_g \cnomm)\beta & \equiv & (\exists \overline\nomi,\overline\cnomn)(g(\overline\nomi,\overline\cnomn) \leq \cnomm \ \metaand \ \beta)\,, \\
\end{array}
\]
where $f \in \mathcal{F}^*$ and $g \in \mathcal{G}^*$. We allow for residual operators in restricted quantifiers under the insight given by Example \ref{ex:goranko_backward_looking}. For instance, the $\langmeta$ formula
\[
\forall \nomi\forall \nomj (\nomi \leq \Diamond \nomj \Rightarrow \forall \nomh(\nomi \leq \Diamondblack \nomh \Rightarrow \exists \nomk(\nomj\leq \Diamondblack \nomk\ \&\ \nomi\leq \Diamond \nomk))),
\]
can be rewritten using this notation as,
\[
\forall \nomi(\forall \nomj \rhd_\Diamond \nomi)(\forall \nomh \rhd_\Diamondblack \nomi)(\exists \nomk \rhd_\Diamondblack \nomj) \, \nomi\leq \Diamond \nomk.
\]
\end{notation}

While it is straightforward to see that the atoms in Kracht's formulae matrices are restricting we will argue and prove that the \emph{same shape} can be kept in: expanding to inductive formulae by allowing for flat inequalities to appear in the matrix (cf.\ \cite{inverse_correspondence_tense} for a detailed survey of this phenomenon in the classical case); expanding to general DLE-logics in an uniform way across semantics by taking (co)nominals.

\subsection{Kracht DLE-formulae}

\begin{definition}[Kracht disjunct]
\label{def:succedentbranches}
A \emph{Kracht-DLE disjunct} is a formula $\theta(\purew)$ in $\langmeta$ defined inductively together with its \emph{main pure variable} $\purew\in \mathrm{NL} \cup \mathrm{CNL}$. It either is:
\begin{itemize}
	\item a flat inequality (cf. Definition \ref{def:restricting_inequality}) $s \leq \purew$ or $\purew \leq t$;
	\item $\theta_1(\purew) \metaand \theta_2(\purew) \& \cdots \& \theta_n(\purew)$ where all the $\theta_i$ (with $1 \leq i \leq n$) are Kracht disjuncts,
	\item  $\theta_1(\purew) \parr \theta_2(\purew) \parr \cdots \parr \theta_n(\purew)$ where all the $\theta_i$ (with $1 \leq i \leq n$) are Kracht disjuncts,
	\item $(\exists \ \overline\pureu \rhd_h \purew) \bigmetaand_i \theta_i(\pureu_i)$, or $(\forall \ \overline\pureu \rhd_h  \dneg\purew)\bigmetaor_i \theta_i(\dneg\pureu_i)$, where each $\theta_i$ is a Kracht disjunct where $\purew$ does not occur, and $h \in \mathcal{F}^* \cup \mathcal{G}^*$\footnote{Note that our restricted quantifier notation implies that the types of the pure variables in $\overline\pureu$ can be deduced by the order type and type of $h$.}.
\end{itemize}
Furthermore, in the generation trees of all the flat inequalities of $\theta(\purew)$, each nominal (resp.\ conominal) different from $\purew$ occurs in negative (resp.\ positive) polarity if it is under the scope of an even number of universal quantifiers, the opposite otherwise.
\end{definition}

\begin{definition}
\label{def:dleantecedent}
A \emph{Kracht-DLE antecedent} is an $\langmeta$-formula $\eta(\nomj,\cnomm)$ which is a conjunction of inequalities of the form $\nomi \leq \nomh$ and $\cnomo \leq \cnomn$, plus a single negated inequality $\nomj \nleq \cnomm$ called a \emph{pivotal inequality}; the variables $\nomj$ and $\cnomm$ are the \emph{pivotal pure variables} of the antecedent.
\end{definition}

\begin{definition}[Kracht DLE-formula]
\label{def:dlekracht}
A closed $\langmeta$-formula is \emph{Kracht DLE} if it is of the following shape:
\begin{equation}
\label{eqn:kracht_shape}
\forall\nomj\forall\cnomm\forall\overline\nomh\forall\overline\cnomo\forall^R\overline\nomi,\overline\cnomn(\eta(\nomj,\cnomm) \Rightarrow \theta_1(\purew_1) \metaor \cdots \metaor \theta_n(\purew_n) ),
\end{equation}
where $\eta(\nomj,\cnomm)$ is a Kracht antecedent, each $\theta_i$ is a Kracht disjunct, and $\forall^R\overline\nomi,\overline\cnomn$ denotes a sequence of restricted universal quantifiers introducing the (co)nominals in $\overline\nomi$ and $\overline\cnomn$. The variables quantified in the prefix are \emph{inherently universal} variables. The formula has also to satisfy the following conditions:
\begin{enumerate} 
	\item each nominal in $\overline\nomh$ (or, resp., conominal in $\overline\cnomo$) must appear on the right (resp.\ left) hand side of exactly one non-pivotal inequality in $\eta(\nomj,\cnomm)$,
	\item the non-main variables (cf.\ Definition \ref{def:succedentbranches}) in each atom in the consequent are all inherently universal,
	\item each $\purew$ has shape $\dneg\pureu$ with $\pureu \in \mathrm{NOM}\cup\mathrm{CONOM}$, 
	\item quantifiers in $\forall^R\overline\nomi,\overline\cnomn$ must be of either of the following types: \emph{type 1}  quantifiers bind variables occurring in the consequent, but not in the antecedent or as restrictors in the prefix;  \emph{type 2} quantifiers bind variables that occur either in the antecedent or as restrictors (exactly once) in the prefix, but not in the consequent.
\end{enumerate}
\end{definition}

\begin{remark}
It follows from Definitions  \ref{def:dleantecedent} and \ref{def:dlekracht}(i) that the variables in $\overline\nomh$ and $\overline\cnomo$ provide alternative names for either pivotal variables or restricted variables in the prefix. This is why we will sometimes refer to them as \emph{aliases}.
\end{remark}

\begin{lemma}
\label{lemma:kracht_pivotal_not_in_succedent}
Every Kracht-DLE formula is equivalent to some Kracht-DLE formula where the pivotal variables do not occur in the consequent.
\end{lemma}
\begin{proof}
Any Kracht formula has the following form
\[
\forall\nomj\forall\cnomm\forall\overline\nomh\forall\overline\cnomo\forall^R\overline\nomi,\overline\cnomn(\eta' \metaand \nomj \nleq \cnomm \Rightarrow \theta_1(\purew_1) \metaor \cdots \metaor \theta_n(\purew_n) ),
\]
and hence it can be equivalently rewritten as the following Kracht-DLE formula
\[
\forall\nomj'\forall\cnomm'\forall\nomj\forall\cnomm\forall\overline\nomh\forall\overline\cnomo\forall^R\overline\nomi,\overline\cnomn(\eta' \metaand \nomj' \leq \nomj \metaand \cnomm \leq \cnomm' \metaand \nomj' \nleq \cnomm' \Rightarrow \theta_1(\purew_1) \metaor \cdots \metaor \theta_n(\purew_n) ),
\] 
where $\nomj'$ and $\cnomm'$ are fresh variables, and therefore they do not occur in the consequent. The variables $\nomj$ and $\cnomm$ become part of the $\overline\nomh$ and $\overline\cnomo$ respectively of the new formula.
\end{proof}

\begin{lemma}
\label{lemma:kracht_aliases_in_succedent}
Any Kracht formula is equivalent to some Kracht formula such that each alias variable occurs in the succedent.
\end{lemma}
\begin{proof}
Suppose that an alias nominal $\nomh$ (resp.\ conominal $\cnomo$) does not occur in the succedent. By definition of Kracht formulas, it occurs exactly once in the antecedent in an inequality of shape $\nomk_\nomh \leq \nomh$ (resp.\ $\cnomo \leq \cnoml_\cnomo$) respectively. As it does not occur in the succedent, the universal quantifier that introduces it can be rewritten as an existential quantifier in the antecedent. Now the formula
$\exists \nomh(\nomk_\nomh \leq \nomh)$ (resp.\ $\exists \cnomo (\cnomo \leq \cnoml_\cnomo)$) is equivalent to $\top$, and, therefore it can be eliminated from the antecedent.
\end{proof}

Thanks to Lemmas \ref{lemma:kracht_pivotal_not_in_succedent} and \ref{lemma:kracht_aliases_in_succedent}, we will henceforth consider only  Kracht formulas where the pivotal variables do not occur in the succedent and whose unrestricted non-pivotal variables occur in the succedent. We will also assume that the variables introduced by \emph{type 1} restricted quantifier occur in the succedent, since, otherwise, the formula would be equivalent to the same formula without those quantifiers. We refer to such formulas as \emph{refined Kracht formulas}.

\input{examples04krachtdle}

\subsection{From inductive to Kracht-DLE}
\label{ssec:inductivetokracht}
\input{04z03inductivetokracht}

\subsection{From Kracht-DLE to very simple Sahlqvist with residuals}
\label{ssec:krachtovssahlq}
\input{04z04krachtoinductive}

%% file: examples04krachtdle.tex
\begin{example} The following formula
\label{eg:goranko}
\[
\begin{array}{l}
\forall  \nomj \forall  \cnomm \forall  \nomh_{1}  \forall  \nomh_{2}    [ \nomj  \leq  \nomh_{1}  \ \&\  \nomj  \leq  \nomh_{2}  \ \&\  \nomj \nleq \cnomm \Rightarrow \hspace{6cm} \\ \hfill
(\exists   \nomi_{1}   \rhd_\Diamond  \lambda(\cnomm)  )  (\forall   \cnomn_{1}   \rhd_\Box  \kappa(\nomi_{1})  )  (\forall   \cnomn_{2}   \rhd_\Box  \cnomn_{1} )  (\exists   \nomi_{2}   \rhd_\Diamondblack \lambda(\cnomn_{2})  ) ( \nomi_{2}  \leq \Diamond  \nomh_{1}  \ \&\  \nomi_{2}  \leq \Diamondblack  \nomh_{2}) ]
\end{array}
\]
is Kracht with pivotal variables $\nomj$ and $\cnomm$, aliases $\nomh_1$ and $\nomh_2$, and a single Kracht disjunct. Indeed, in $\nomi_2 \leq \Diamond \nomh_1$ and $\nomi_2 \leq \Diamondblack\nomh_2$, $\nomh_1$ and $\nomh_2$ are inherently universal and they occur in negative polarity while being under the scope of an even number of universal quantifiers. By Lemma \ref{lemma:kracht_pivotal_not_in_succedent} and by renaming $\cnomm$ to $\cnomo_1$, it is equivalent to the following refined Kracht formula:
\begin{equation}
\begin{array}{l}
\forall  \nomj \forall  \cnomm \forall  \nomh_{1}  \forall  \nomh_{2}   \forall \cnomo_1 [ \nomj  \leq  \nomh_{1}  \ \&\  \nomj  \leq  \nomh_{2}  \metaand \cnomo_1 \leq \cnomm \ \&\  \nomj \nleq \cnomm \Rightarrow \hspace{3.8cm} \\ \hfill
(\exists   \nomi_{1}   \rhd  \lambda(\cnomo_1)  )  (\forall   \cnomn_{1}   \rhd_\Diamond  \kappa(\nomi_{1})  )  (\forall   \cnomn_{2}   \rhd_\Diamond  \cnomn_{1} )  (\exists   \nomi_{2}   \rhd_\Diamondblack \lambda(\cnomn_{2})  ) ( \nomi_{2}  \leq \Diamond  \nomh_{1}  \ \&\  \nomi_{2}  \leq \Diamondblack  \nomh_{2}) ]
\end{array}
\end{equation}
\end{example}

\begin{example}
\label{eg:secondgoranko}
The Kracht formula
\[
\begin{array}{l}
\forall \nomj \forall \cnomm \forall \nomh_1\forall \nomh_2  (\forall \nomi_1 \rhd_\Diamond \nomj)(\forall \cnomn_1 \rhd_\Box \cnomm) [\nomi_1 \leq \nomh_1 \metaand \nomi_1 \leq \nomh_2 \metaand \nomj \nleq \cnomm \Rightarrow \hspace{3cm}\\
\hfill \lambda(\cnomm) \leq \nomh_2 \metaor (\exists \nomi_2 \rhd \lambda(\cnomn_1))(\forall \cnomn_2 \rhd_\Box \kappa(\nomi_2))(\lambda(\cnomn_2) \leq \Diamond\nomh_1)]
\end{array}
\]
is equivalent to the following refined Kracht formula (introducing an alias for $\cnomm$)
\begin{equation}
\begin{array}{l}
\forall \nomj \forall \cnomm \forall \nomh_1\forall \nomh_2  \forall \cnomo_1(\forall \nomi_1 \rhd_\Diamond \nomj)(\forall \cnomn_1 \rhd_\Box \cnomm) [ \nomi_1 \leq \nomh_1 \metaand \nomi_1 \leq \nomh_2 \metaand \cnomo_1 \leq \cnomm \metaand \nomj \nleq \cnomm \Rightarrow \hspace{0.8cm}\\
\hfill \lambda(\cnomo_1) \leq \nomh_2 \metaor (\exists \nomi_2 \rhd_\Diamond \lambda(\cnomn_1))(\forall \cnomn_2 \rhd_\Box \kappa(\nomi_2))(\lambda(\cnomn_2) \leq \Diamond\nomh_1)]
\end{array}
\end{equation}
\end{example}

\begin{example}
\label{eg:full_lambek}
Consider the language of distributive full Lambek calculus with $\mathcal{F} = \{e, \circ \}$ with $n_e = 0$, $n_\circ = 2$, $\varepsilon_\circ = (1,1)$, and $\mathcal{G} = \{\slash, \backslash\}$ with $n_\slash = n_\backslash = 2$, $\varepsilon_\slash = (1,\partial)$, and $\varepsilon_\backslash = (\partial,1)$. The following formula
\[
\begin{array}{l}
\forall \nomj \forall \cnomm
\forall \nomh_1
(\forall \cnomn_1,\nomi_1 \rhd_\slash \cnomm)
(\forall \nomi_2,\cnomn_2 \rhd_\backslash \cnomn_1)[
\nomi_1 \leq \nomh_1 \metaand
\nomj \nleq \cnomm
\Rightarrow \hspace{3.4cm} \\
\hfill
(\exists \nomi_3,\nomi_4 \rhd_\circ \lambda(\cnomn_2))
(
\nomi_3 \leq \nomi_2
\metaand
(\forall \cnomn_5,\nomi_5 \rhd_\slash \kappa(\nomi_4))
(
\lambda(\cnomn_5) \leq \nomi_2 \circ \nomj
\metaand
\nomi_5 \leq \nomj
)
)
],
\end{array}
\]
is Kracht and by Lemma \ref{lemma:kracht_pivotal_not_in_succedent} and by renaming $\nomj$ to $\nomh_2$, it is equivalent to
\[
\begin{array}{l}
\forall \nomj \forall \cnomm
\forall \nomh_1 \forall \nomh_2
(\forall \cnomn_1,\nomi_1 \rhd_\slash \cnomm)
(\forall \nomi_2,\cnomn_2 \rhd_\backslash \cnomn_1)[
\nomi_1 \leq \nomh_1 \metaand
\nomj \leq \nomh_2 \metaand
\nomj \nleq \cnomm
\Rightarrow \hspace{1.7cm} \\
\hfill
(\exists \nomi_3,\nomi_4 \rhd_\circ \lambda(\cnomn_2))
(
\nomi_3 \leq \nomi_2
\metaand
(\forall \cnomn_5,\nomi_5 \rhd_\slash \kappa(\nomi_4))
(
\lambda(\cnomn_5) \leq \nomi_2 \circ \nomh_2
\metaand
\nomi_5 \leq \nomh_2
)
)
],
\end{array}
\]
which is a refined Kracht DLE-formula.
\end{example}

%% file: 04z03inductivetokracht.tex
Before presenting an algorithm for inverse correspondence that accepts Kracht-DLE formulae in input, we show that the first order correspondent of any inductive formula is equivalent to some refined Kracht-DLE formula.

Given a definite inductive formula $(\varphi \leq \psi)[\overline\alpha,\overline x, \overline\beta/\overline y, \overline\gamma / \overline z, \overline\delta/\overline w]$, by Proposition \ref{prop:ALBA_output_in_ackermann} running ALBA on it yields
\begin{equation}
\label{eq:afterAckermannanind}
\begin{array}{l}
 \forall\overline{\nomj} \forall\overline{\cnomm}\forall\overline{\nomi}\forall\overline{\cnomn}\Big(
 \overline{\nomi}\leq \overbrace{\overline{\gamma}\left[\overline{\bigvee\mathsf{Mv}(p)}/\overline{p}, \overline{\bigwedge\mathsf{Mv}(q)}/\overline{q}\right]}^{\overline{\gamma}^{mv}} \ \&\ 
\overbrace{\overline{\delta}\left [\overline{\bigvee\mathsf{Mv}(p)}/\overline{p}, \overline{\bigwedge\mathsf{Mv}(q)}/\overline{q}\right]}^{{\overline\delta}^{mv}}\leq \overline{\cnomn}
\Rightarrow \\ \hfill (\varphi\leq \psi)[!\overline{\nomj}/!\overline{x}, !\overline{\cnomm}/!\overline{y},!\overline{\nomi}/!\overline{z}, !\overline{\cnomn}/!\overline{w}] \Big),
 \end{array}
\end{equation}
which is equivalent to its contrapositive
\begin{equation}
\label{eq:eqnbeforestrip}
 \forall\overline{\nomj} \forall\overline{\cnomm}\forall\overline{\nomi}\forall\overline{\cnomn}\left(
 (\varphi\nleq \psi)[!\overline{\nomj}/!\overline{x}, !\overline{\cnomm}/!\overline{y},!\overline{\nomi}/!\overline{z}, !\overline{\cnomn}/!\overline{w}] \Rightarrow 
 \bigmetaor_{i=1}^n \gamma_{i}^{mv} \leq \dneg\nomi_i \ \metaor \  \bigmetaor_{i=1}^m \dneg\cnomn_i \leq \delta_{i}^{mv} 
 \right).
\end{equation}

We put $(\varphi' \leq \psi') \equiv (\varphi\leq \psi)[!\overline{\nomj}/!\overline{x}, !\overline{\cnomm}/!\overline{y},!\overline{\nomi}/!\overline{z}, !\overline{\cnomn}/!\overline{w}] $, by approximating the antecedent (Corollary \ref{cor:firstapproxnomcnom}) we have:
\begin{equation}
\label{eq:eqnbeforestrippivotalinserted}
\forall\nomj' \forall\cnomm' \forall\overline{\nomj} \forall\overline{\cnomm}\forall\overline{\nomi}\forall\overline{\cnomn}\left(
\nomj' \leq \varphi' \metaand \psi' \leq \cnomm' \metaand \nomj' \nleq \cnomm' \Rightarrow 
 \bigmetaor_{i=1}^n \gamma_{i}^{mv} \leq \dneg\nomi_i \ \metaor \  \bigmetaor_{i=1}^m \dneg\cnomn_i \leq \delta_{i}^{mv} 
 \right).
\end{equation}
The variables $\nomj'$ and $\cnomm'$ will be the pivotal pure variables of the inductive Kracht formula that we will compute. In the remainder of the subsection we will show that \eqref{eq:eqnbeforestrippivotalinserted} is equivalent to some refined Kracht-DLE formula.

\begin{lemma}
\label{lemma:inductivetokracht_succedent}
Each $\gamma^{mv}_i \leq \dneg\nomi_i$ and $\dneg\cnomn_i \leq \delta^{mv}_i$ in \eqref{eq:eqnbeforestrippivotalinserted} is equivalent to some Kracht-DLE disjunct where all the non-main variables in each atom are variables in $\gamma_i^{mv}$ and $\delta_i^{mv}$ respectively.
\end{lemma}
\begin{proof}
We prove that given any formula $\theta$, nominal $\nomj$ (resp.\ conominal $\cnomm$) $\nomj \leq \theta$ (resp.\ $\theta \leq \cnomm$) is equivalent to a Kracht-DLE disjunct where the non-main variables in each atom are variables in $\theta$. It will follow, as $\dneg\nomi_i$ and $\dneg\cnomn_i$ do not occur in $\gamma_i^{mv}$ and $\delta_i^{mv}$ respectively, that also each $\gamma^{mv}_i \leq \dneg\nomi_i$ and $\dneg\cnomn_i \leq \delta^{mv}_i$ is a Kracht-DLE disjunct. We proceed by induction on the structure of $\theta$. The base cases are the ones in which the inequality is already flat, so let us consider the case such that $\theta = f(\alpha_1,\ldots,\alpha_n,\beta_1,\ldots,\beta_m)$ where the first $n$ coordinates of $f$ are the positive ones and where at least one of the arguments of $f$ is not pure. Inequalities of the form  $\nomj \leq f(\alpha_1,\ldots,\alpha_n,\beta_1,\ldots,\beta_m)$ can be treated either by splitting if $f\equiv\wedge$, yielding two Kracht-DLE disjuncts; or otherwise by applying Lemma \ref{lemma:flattening_skeleton_formulas} introducing new fresh variables $\nomh_1,\ldots,\nomh_n,\cnomn_1,\ldots,\cnomn_m$, yielding
\begin{equation}
\label{eqn:epspartackermannstrip}
\begin{array}{ll}
\exists\nomh_1,\ldots,\nomh_n\exists\cnomo_1,\ldots,\cnomo_m( & \nomh_1 \leq \alpha_1\ \metaand\ \ldots\ \metaand\ \nomh_n \leq \alpha_n\ \metaand\  \\
				& \beta_1 \leq \cnomo_1\ \metaand\ \ldots\ \metaand\ \beta_m \leq \cnomo_m\ \metaand\  \\
				& \nomj \leq f(\nomh_1,\ \ldots,\nomh_n,\cnomo_1,\ \ldots,\cnomo_m)\ )\,.
\end{array}
\end{equation}
Clearly $\nomj \leq f(\nomh_1,\ \ldots,\nomh_n,\cnomo_1,\ \ldots,\cnomo_m)$ is a restricting inequality by definition. By inductive hypothesis, all the other inequalities are 
equivalent to Kracht-$\mathrm{DLE}$ disjuncts;
hence, as \eqref{eqn:epspartackermannstrip} can be compactly rewritten as follows
\[ 
(\exists \overline\nomh, \overline\cnomo \rhd_f \nomj)(\bigmetaand_{j=1}^{n}{\nomh_j \leq \alpha_j} \ \& \ \bigmetaand_{j=1}^{m}{\beta_j \leq \cnomo_j} ), 
\]
and each nominal $\nomh_j$ (resp.\ conominal in $\cnomo_j$) does not occur in any of the formulas in $\alpha$ and $\beta$, also $\nomj \leq \theta$ is a Kracht-DLE disjunct\footnote{Note that we have not introduced universal quantifiers, so the polarity of the variables in $\theta$ respects the constraints in Kracht-DLE disjunct definition by inductive hypothesis.}. Since the non-pure variables of the formula are in the $\alpha_i$ and $\beta_i$ by inductive hypothesis, the non-pure variables have to be in $\theta$.
As for inequalities $\theta \leq \cnomm$, by applying the same argument,
\begin{center}
\begin{tabular}{rlr}
& $\theta \leq \cnomm$ & \\
iff & $\lambda(\cnomm) \nleq f(\alpha_1,\ \ldots,\alpha_n,\beta_1,\ \ldots,\beta_n)$ & Lemma \ref{lemma:flipnegation}\\
iff & $\metanot[ \lambda(\cnomm) \leq f(\alpha_1,\ \ldots,\alpha_n,\beta_1,\ \ldots,\beta_n) ]$ & property of $\metanot$ \\
iff & $\metanot(\exists \overline\nomh, \overline\cnomo \rhd_f \lambda(\cnomm))(\bigmetaand_{j=1}^{n}{\nomh_j \leq \alpha_j} \ \& \ \bigmetaand_{j=1}^{m}{\beta_j \leq \cnomo_j} )$ & Lemma \ref{lemma:flattening_skeleton_formulas} \\
iff & $(\forall \overline\nomh, \overline\cnomo \rhd_f \lambda(\cnomm))\metanot(\bigmetaand_{j=1}^{n}{\nomh_j \leq \alpha_j} \ \& \ \bigmetaand_{j=1}^{m}{\beta_j \leq \cnomo_j} )$ & Notation \ref{notation:restrictedquant}\\
iff & $(\forall \overline\nomh, \overline\cnomo \rhd_f \lambda(\cnomm))(\bigmetaor_{j=1}^{n}{\nomh_j \nleq \alpha_j} \metaor \bigmetaor_{j=1}^{m}{\beta_j \nleq \cnomo_j} )$ & property of $\metanot$ \\
iff & $(\forall \overline\nomh, \overline\cnomo \rhd_f \lambda(\cnomm))(\bigmetaor_{j=1}^{n}{\alpha_j \leq \kappa(\nomh_j)} \metaor  \bigmetaor_{j=1}^{m}{\lambda(\cnomo_j) \leq \beta_j} ).$ & Lemma \ref{lemma:flipnegation} 
\end{tabular}
\end{center}
As in the previous case, it is sufficient to note that the inductive hypothesis applies to all the inequalities in the matrix of the formula. Here the condition on the polarity of the non-main variables of the atoms is respected as a universal quantifier is introduced, but at the same time all the polarities are flipped in the inequalities in the matrix by the meta-negation.
The case in which $\theta =  g(\alpha_1,\ \ldots,\alpha_n,\beta_1,\ \ldots,\beta_n) $ can be treated similarly.
\end{proof}

Taking stock, the applications of Corollary \ref{lemma:flattening_skeleton_formulas} in the previous lemma are condensed in the following table using the restricted quantifier notation.

\begin{table}[h]
\begin{center}
\label{table:ackermann_restr_quant}
\begin{tabular}{|c|c|}
\hline
$\nomj \leq f(\overline\alpha,\overline\beta)$ &  $(\exists \overline\nomi, \overline\cnomn \rhd_f \nomj)(\bigmetaand_{i=1}^{n}{\nomi_i \leq \alpha_i} \ \& \ \bigmetaand_{i=1}^{n}{\beta_i \leq \cnomn_i} )$ \\
\hline
$g(\overline\alpha,\overline\beta) \leq \cnomm$ & $(\exists \overline\cnomn,\overline\nomi \rhd_g \cnomm)[\bigmetaand_{i=1}^{n}{\alpha_i\leq\cnomn_i} \ \& \ \bigmetaand_{i=1}^{n}{\nomi_i \leq \beta_i}]$ \\
\hline
$f(\overline\alpha,\overline\beta) \leq \cnomm$ & $(\forall \overline\nomi, \overline\cnomn \rhd_{f} \lambda(\cnomm))[\bigmetaor_{i=1}^n{\alpha\leq\kappa(\nomi)} \ \metaor \ \bigmetaor_{i=1}^m{\lambda(\cnomn)\leq\beta}]$ \\
\hline
$\nomj \leq g(\overline\alpha,\overline\beta)$ & $(\forall \overline\cnomn, \overline\nomi \rhd_{g} \kappa(\nomj))[\bigmetaor_{i=1}^n{\lambda(\cnomn_i) \leq \alpha} \ \metaor \ \bigmetaor_{i=1}^m{\beta\leq\kappa(\nomi_i)}]$ \\
\hline
\end{tabular}
\end{center}
\caption{Corollary \ref{lemma:flattening_skeleton_formulas} with restricted quantifiers}
\end{table}

\begin{lemma}
\label{lemma:inductivetokracht_antecedent}
The inequalities $\nomj'\leq \varphi'$ and $\psi' \leq \cnomm'$ in \eqref{eq:eqnbeforestrippivotalinserted} are equivalent to some prenex formula whose prefix is a sequence of restricted existential quantifiers using operators in $\mathcal{F}\cup\mathcal{G}$, and whose matrix is a conjunction of restricting inequalities where $\nomj'$ and $\cnomm'$ do not occur in restricted position.
\end{lemma}
\begin{proof}
We prove the statement for every inequality $\nomi \leq \varphi$ (resp.\ $\psi\leq\cnomn$), where $\varphi$ (resp.\ $\psi$) is a positive (resp.\ negative) pure scattered skeleton formula in $\langbase$, in particular it will hold for $\nomj'\leq \varphi'$ and $\psi' \leq \cnomm'$. We proceed by induction on the structure $\varphi$ (resp.\ $\psi$). The base is case is when $\nomj \leq \varphi$ is a restricted inequality, which satisfies the statement trivially. Suppose now $\varphi = f(\overline\alpha, \overline\beta)$ where $f \in \mathcal{F}$ and the $\alpha$s (resp.\ $\beta$s) are in positive (resp.\ negative) position\footnote{As $\varphi$ is a positive skeleton formula, the outermost connective cannot be in $\mathcal{G}$.}. If $f\equiv\wedge$ it is sufficient to apply splitting and by inductive hypothesis on the two conjuncts, the statement holds. If $f\nequiv\wedge$, by Lemma \ref{lemma:flattening_skeleton_formulas}, the inequality $\nomj\leq f(\overline\alpha,\overline\beta)$ is equivalent to,
\[
(\exists \overline\nomh,\overline\cnomo \rhd_f \nomj)(\overline{\nomh\leq\alpha} \metaand \overline{\beta\leq\cnomo}).
\]
Since $\varphi$ is a positive skeleton formula, each one of the $\alpha$s (resp.\ $\beta$s) is a positive (resp.\ negative) skeleton.
Therefore, by applying the inductive hypothesis to the inequalities in the matrix, and moving the resulting existential quantifiers in the prefix, the statement is proved. The lemma is proved similarly for inequalities $\psi\leq\cnomn$.
\end{proof}

\begin{lemma}
\label{lemma:inductivetokracht}
The $\langmeta$-formula \eqref{eq:eqnbeforestrippivotalinserted} is equivalent to some refined Kracht-DLE formula.
\end{lemma}
\begin{proof}
It is sufficient to apply Lemma \ref{lemma:inductivetokracht_succedent} to each inequality in the consequent, and Lemma \ref{lemma:inductivetokracht_antecedent}.
By Lemma \ref{lemma:inductivetokracht_succedent}, the non-main variables of the atoms in the succedent are variables that were originally in some $\gamma_i^{mv}$ or $\delta_i^{mv}$, and therefore some variable in $\overline\nomj$ or $\overline\cnomm$, which are inherently universal. The sequence on existential quantifiers produced by Lemma \ref{lemma:inductivetokracht_antecedent} can be rewritten as a sequence of universal restricted quantifiers in the prefix; as they introduce variables that cannot occur in the consequent, these are the \emph{type 2} quantifiers of the formula. The remaining restricting inequalities with at least one operator can be rewritten using Notation \ref{notation:restrictedquant} and incorporated into restricted quantifiers in the prefix restricting some of the variables in $\overline\nomj$ and $\overline\cnomm$. As these can occur in the consequent, but not in the antecedent (a skeleton formula is scattered), these are the variables introduced by the \emph{type 1} quantifiers of the formula. The remaining variables in $\overline\nomj$ and $\overline\cnomm$ are the \emph{alias} variables, whilst $\nomj'$ and $\cnomm'$ are the \emph{pivotal variables} of the formula.
\end{proof}

The lemmas shown in this section implicitly encode an algorithm that `strips' the operators in the ALBA output of an inductive formula to produce a Kracht-DLE formula. The following examples show how this algorithm works in practice.

\input{examples04z03inductivetokracht}

%% file: examples04z03inductivetokracht.tex
\begin{example}
\label{ex:inductivetokracht:transitivity}
Let us show how to express the first order correspondence of $\Box p \leq \Box\Box p$ in Kracht shape. By applying the steps to get in the shape of \eqref{eq:eqnbeforestrippivotalinserted},
\begin{center}
\begin{tabular}{rlr}
    & $\forall p (\Box p \leq \Box\Box p)$ & \\
    iff &  $\forall \nomj\forall\cnomn\left(\nomj \leq \Box\cnomn \Rightarrow \nomj \leq \Box\Box\cnomn \right)$ & ALBA run \\
    iff &  $\forall \nomj\forall\cnomn\left( \nomj \nleq \Box\Box\cnomn \Rightarrow \nomj \nleq \Box\cnomn \right)$ & contrapositive \\
    iff &  $\forall \nomj\forall\cnomm\forall\cnomn\left( \Box\Box\cnomn \leq \cnomm \metaand \nomj \nleq \cnomm \Rightarrow \nomj \nleq \Box\cnomn \right)$ & first approximation \\
    iff &  $\forall \nomj\forall\cnomm\forall\cnomn\left( \Box\Box\cnomn \leq \cnomm \metaand \nomj \nleq \cnomm \Rightarrow \Box\cnomn \leq \kappa(\nomj) \right)$. & Lemma \ref{lemma:flipnegation}
\end{tabular}
\end{center}
The consequent contains only one disjunct which is already a flat inequality. Let us apply Lemma \ref{lemma:inductivetokracht_antecedent} to the antecedent.
\begin{center}
\begin{tabular}{rlr}
    & $\forall \nomj\forall\cnomm\forall\cnomn\left( \Box\Box\cnomn \leq \cnomm \metaand \nomj \nleq \cnomm \Rightarrow \Box\cnomn \leq \kappa(\nomj) \right)$ & \\
    iff & $\forall \nomj\forall\cnomm\forall\cnomn \left( (\exists\cnomo \rhd_\Box \cnomm)(\Box\cnomn \leq \cnomo) \metaand \nomj \nleq \cnomm \Rightarrow \Box\cnomn \leq \kappa(\nomj) \right)$ & Corollary \ref{lemma:flattening_skeleton_formulas} \\
    iff & $\forall \nomj\forall\cnomm\forall\cnomn (\forall\cnomo \rhd_\Box \cnomm) \left( \Box\cnomn \leq \cnomo \metaand \nomj \nleq \cnomm \Rightarrow \Box\cnomn \leq \kappa(\nomj) \right)$ & FO meta-language \\
    iff & $\forall \nomj\forall\cnomm(\forall\cnomo \rhd_\Box \cnomm)(\forall\cnomn \rhd_\Box \cnomo) \left( \nomj \nleq \cnomm \Rightarrow \Box\cnomn \leq \kappa(\nomj) \right)$. & Notation \ref{notation:restrictedquant}
\end{tabular}
\end{center}
The last formula is in Kracht shape, as required. Notice how the two restricted quantifiers in the prefix $(\forall\cnomo \rhd_\Box \cnomm)$ and $(\forall\cnomn \rhd_\Box \cnomo)$, are result of different parts of the procedure. Indeed, the first is a \emph{type 2} quantifier, whereas the second is a \emph{type 1} one.
\end{example}

\begin{example}
\label{ex:inductivetokracht:morecomplex}
We show the procedure above on a slightly more complex formula, namely
\[
    \forall p \forall q \forall r[ \Diamond( (p \wedge q) \rightarrow r) \wedge \Box q \leq \Box(\Diamond p \rightarrow \Diamond(q \wedge r) ].
\]
The formula is inductive with order type $\varepsilon(p,q,r) = (1,1,1)$ and inductive order $p <_\Omega q <_\Omega r$. An ALBA run yields
\begin{center}
\begin{tabular}{rl}
     & $\forall p \forall q \forall r[ \Diamond( (p \wedge q) \rightarrow r) \wedge \Box q \leq \Box(\Diamond p \rightarrow \Diamond(q \wedge r) ]$ \\
    iff & $\forall p \forall q \forall r \forall \nomi \forall \nomh \forall \nomk \forall \cnomn[ \nomi \leq (p \wedge q) \rightarrow r \metaand \nomh \leq \Box q \metaand \nomk \leq p \metaand \Diamond(q \wedge r) \leq \cnomn$\\
    & \hfill $\Rightarrow \Diamond\nomi \wedge \nomh \leq \Box(\Diamond\nomk \rightarrow \cnomn) ]$ \\
    iff & $\forall p \forall q \forall r \forall \nomi \forall \nomh \forall \nomk \forall \cnomn[ \nomi \wedge (p \wedge q) \leq r \metaand \Diamondblack\nomh \leq q \metaand \nomk \leq p \metaand \Diamond(q \wedge r) \leq \cnomn$ \\
    & \hfill $\Rightarrow \Diamond\nomi \wedge \nomh \leq \Box(\Diamond\nomk \rightarrow \cnomn) ]$\\
    iff & $\forall \nomi \forall \nomh \forall \nomk \forall \cnomn[ \Diamond(\Diamondblack\nomh \wedge \nomi \wedge \nomk \wedge \Diamondblack\nomh) \leq \cnomn \Rightarrow \Diamond\nomi \wedge \nomh \leq \Box(\Diamond\nomk \rightarrow \cnomn) ]$,
\end{tabular}
\end{center}
whose contrapositive is
\[
\forall \nomi \forall \nomh \forall \nomk \forall \cnomn[ 
\Diamond\nomi \wedge \nomh \nleq \Box(\Diamond\nomk \rightarrow \cnomn)
\Rightarrow 
\Diamond(\Diamondblack\nomh \wedge \nomi \wedge \nomk \wedge \Diamondblack\nomh) \nleq \cnomn.
 ]
\]
By Corollary \ref{cor:firstapproxnomcnom} and Lemma \ref{lemma:flipnegation}, it is equivalent to
\[
\forall\nomj\forall\cnomm\forall \nomi \forall \nomh \forall \nomk \forall \cnomn[ 
\nomj \leq \Diamond\nomi \wedge \nomh \metaand \Box(\Diamond\nomk \rightarrow \cnomn) \leq \cnomm \metaand \nomj \nleq \cnomm
\Rightarrow 
\lambda(\cnomn) \leq \Diamond(\Diamondblack\nomh \wedge \nomi \wedge \nomk \wedge \Diamondblack\nomh) 
].
\]
Let us now apply Lemma \ref{lemma:inductivetokracht_succedent} on the consequent.
\begin{center}
\begin{tabular}{rlr}
     & $\lambda(\cnomn) \leq \Diamond(\Diamondblack\nomh \wedge \nomi \wedge \nomk \wedge \Diamondblack\nomh)$ & \\
     iff & $(\exists \nomi' \rhd_\Diamond \lambda(\cnomn))( \nomi' \leq \Diamondblack\nomh \wedge \nomi \wedge \nomk \wedge \Diamondblack\nomh)$ & Corollary \ref{lemma:flattening_skeleton_formulas} \\
     iff & $(\exists \nomi' \rhd_\Diamond \lambda(\cnomn))( \nomi' \leq \Diamondblack\nomh \metaand \nomi' \leq \nomi \metaand \nomi' \leq \nomk \metaand \nomi' \leq \Diamondblack\nomh)$ & splitting. \\
\end{tabular}
\end{center}
By the procedure described in Lemma \ref{lemma:inductivetokracht_antecedent}, the first conjunct on the antecedent becomes
\[
\nomj \leq \Diamond\nomi \metaand \nomj \leq \nomh,
\]
by splitting, whereas the second becomes
\begin{center}
\begin{tabular}{rlr}
     & $\Box(\Diamond\nomk \rightarrow \cnomn) \leq \cnomm$ & \\
     iff & $(\exists \cnomo \rhd_\Box \cnomm)(\Diamond \nomk \rightarrow \cnomn \leq \cnomo)$ & Corollary \ref{lemma:flattening_skeleton_formulas} \\
     iff & $(\exists \cnomo \rhd_\Box \cnomm)(\exists \nomh',\cnoml \rhd_\rightarrow \cnomo) (\nomh' \leq \Diamond \nomk \metaand \cnomn \leq \cnoml)$. & Corollary \ref{lemma:flattening_skeleton_formulas}
\end{tabular}
\end{center}
After all of these steps, the obtained formula is
\[
\begin{array}{l}
\forall\nomj\forall\cnomm\forall \nomi \forall \nomh \forall \nomk \forall \cnomn\big[
\nomj \leq \Diamond\nomi \metaand \nomj \leq \nomh \wedge \metaand (\exists \cnomo \rhd_\Box \cnomm)(\exists \nomh',\cnoml \rhd_\rightarrow \cnomo) (\nomh' \leq \Diamond \nomk \metaand \cnomn \leq \cnoml) \metaand \\
\phantom{\forall\nomj\forall\cnomm\forall \nomi \forall \nomh \forall \nomk \forall \cnomn\big[}
\nomj \nleq \cnomm \hfill \Rightarrow (\exists \nomi' \rhd_\Diamond \lambda(\cnomn))( \nomi' \leq \Diamondblack\nomh \metaand \nomi' \leq \nomi \metaand \nomi' \leq \nomk \metaand \nomi' \leq \Diamondblack\nomh) \ \ \big],
\end{array}
\]
which by rewriting the \emph{2} quantifiers in the prefix becomes
\[
\begin{array}{l}
\forall\nomj\forall\cnomm\forall \nomi \forall \nomh \forall \nomk \forall \cnomn(\forall \cnomo \rhd_\Box \cnomm)(\forall \nomh',\cnoml \rhd_\rightarrow \cnomo)\big[
\nomj \leq \Diamond\nomi \metaand \nomj \leq \nomh \wedge \metaand \nomh' \leq \Diamond \nomk \metaand \cnomn \leq \cnoml \metaand \\
\phantom{\forall\nomj\forall\cnomm\forall \nomi \forall \nomh \forall \nomk \forall \cnomn\big[}
\nomj \nleq \cnomm \hfill \Rightarrow (\exists \nomi' \rhd_\Diamond \lambda(\cnomn))( \nomi' \leq \Diamondblack\nomh \metaand \nomi' \leq \nomi \metaand \nomi' \leq \nomk \metaand \nomi' \leq \Diamondblack\nomh) \ \ \big].
\end{array}
\]
By using Notation \ref{notation:restrictedquant}, also the \emph{type 1} quantifiers take shape, indeed the previous formula is equivalent to
\[
\begin{array}{l}
\forall\nomj\forall\cnomm\forall \nomh \forall \cnomn(\forall \cnomo \rhd_\Box \cnomm)(\forall \nomh',\cnoml \rhd_\rightarrow \cnomo)(\forall \nomi \rhd_\Diamond \nomj)(\forall \nomk \rhd_\Diamond \nomh') \big[ \nomj \leq \nomh \wedge \metaand \cnomn \leq \cnoml \metaand \nomj \nleq \cnomm \\
\phantom{\forall\nomj\forall\cnomm\forall \nomi \forall \nomh \forall \nomk \forall \cnomn\big[}
\hfill \Rightarrow (\exists \nomi' \rhd_\Diamond \lambda(\cnomn))( \nomi' \leq \Diamondblack\nomh \metaand \nomi' \leq \nomi \metaand \nomi' \leq \nomk \metaand \nomi' \leq \Diamondblack\nomh) \ \ \big],
\end{array}
\]
which is in Kracht-DLE shape.
\end{example}

%% file: 04z04krachtoinductive.tex
In this section we show an algorithm that given a Kracht-DLE formula for some DLE language $\langbase$, outputs a very simple Sahlqvist formula in $\langresidual$ to which it corresponds.

\subsubsection{Compaction of the consequent} By  exhaustively  applying  Ackermann eliminations and inverse splitting, a Kracht disjunct $\theta(\purew)$ is shown to be equivalent to some inequality that has $\purew$ on display. This algorithm inverts the `operators stripping' algorithm discussed in Lemma \ref{lemma:inductivetokracht_succedent}. In the algorithm it is implicit that one main pure variable is fixed in Kracht-DLE disjunct in input. Every Kracht-DLE disjunct has a unique main pure variable, except those of shape $\nomi\leq\cnomn$, for which every choice leads to the same outcome.

\begin{algorithm}
\label{algo:elimination_of_non-inherently_universals}
\caption{Compaction of a Kracht-DLE disjunct $\theta(\purew)$.}
\begin{algorithmic}[1]
\Procedure{DisjunctCompaction}{$\theta$}
	\If{$\theta$ is a flat inequality}
		\Return $\theta$
	\Else
		\State Let $\{\theta_1, \ldots, \theta_n \}$ be the set of all the direct sub-Kracht-DLE disjuncts of $\theta$
		\State Let $I=[I_1,\ldots,I_n]$ a list of inequalities
		\ForAll{$\theta_i$ in $\{\theta_1,\ldots,\theta_n\}$}
			\State $I_i \gets$ DisjunctCompaction$(\theta_i)$
		\EndFor
		\If{$\theta$ is a (dis/con)junction of $\theta_i(\nomj)$}				\State Let $s_1,\ldots,s_n$ be formulas such that $I_i$ is $\nomj \leq s_i$ for $i=1,\ldots,n$
				\State\Return \ \ $\nomj \leq s_1 \ \wedge \ \cdots \ \wedge \ s_n$ if conjunction, $\nomj \leq s_1 \ \vee \ \cdots \ \vee \ s_n$ otherwise
		\ElsIf{$\theta$ is a (dis/con)junction of $\theta_i(\cnomm)$}
				\State Let $s_1,\ldots,s_n$ be formulas such that $I_i$ is $s_i\leq \cnomm$ for $i=1,\ldots,n$
				\State \Return \  $s_1 \vee  \cdots  \vee  s_n \leq \cnomm$ if conjunction, $s_1 \wedge  \cdots  \wedge  s_n \leq \cnomm$ otherwise
		\ElsIf{$\theta$ has form $(Q \ \overline\pureu \ \rhd_h \ \purev)\zeta$ with $Q \in \{\forall,\exists\}$ and $h \in \mathcal{F}^* \cup \mathcal{G}^*$}
			\State \Return Eliminate $\overline\pureu$ via Corollary \ref{lemma:flattening_skeleton_formulas}
		\EndIf
	\EndIf
	\EndProcedure
\end{algorithmic}
\end{algorithm}

\begin{lemma}
When applied a Kracht-DLE disjunct $\theta$, Algorithm \ref{algo:elimination_of_non-inherently_universals} outputs an inequality of shape $\nomk \leq s$ (resp.\ $s \leq \cnoml$), where $\nomk$ (resp.\ $\cnoml$) is the main pure variable of the branch, and it does not occur in $s$.
\end{lemma}
\begin{proof}
We proceed by induction on the structure of $\theta(\purew)$. If $\theta(\purew)$ is a flat inequality, then the statement holds by definition of Kracht disjunct. If $\theta(\purew):=\theta_1(\purew) \metaand \cdots \metaand \theta_n(\purew)$ (resp.\ $\theta(\purew):=\theta_1(\purew) \metaor \cdots \metaor \theta_n(\purew)$), then the algorithm applies inverse splitting in line 11 if $\purew$ is a nominal, line 14 if it is a conominal. As, by inductive hypothesis on each $\theta_i$, $\purew$ does not occur in $s_i$ ($\purew$ is main in each $s_i$), it does not occur in $\bigwedge_i s_i$ (resp.\ $\bigvee_i s_i$). 
Assume that $\theta(\purew) \coloneqq (\exists \overline\pureu \ \rhd_h \ \purew)(\bigmetaand_{i=1}^n \theta_i(\pureu_i))$ (resp.\ $\theta(\purew) \coloneqq (\forall \overline\pureu \ \rhd_h \ \dneg\purew)(\bigmetaor_{i=1}^n \theta_i(\dneg\pureu_i))$), with $h \in \mathcal{F}^* \cup \mathcal{G}^*$. The induction hypothesis on all the $\theta_i$ (for $1\leq i\leq n$) ensures that $\theta(\purew)$ is in Ackermann shape w.r.t.\ $\pureu$ (in the sense that i satisfies the conditions of Corollary \ref{lemma:flattening_skeleton_formulas}), which can then be eliminated by applying Corollary \ref{lemma:flattening_skeleton_formulas} in line 16 (cf. Table \ref{table:ackermann_restr_quant}), thus proving the statement.
\end{proof}

\begin{lemma}
\label{lemma:polarities_in_gammadelta}
When applied to some Kracht-DLE disjunct $\theta(\purew)$, Algorithm \ref{algo:elimination_of_non-inherently_universals} produces an inequality where the nominals (resp.\ conominals) different from $\purew$ occur in negative (resp.\ positive) polarity.
\end{lemma}
\begin{proof}
By definition of Kracht-DLE disjunct, the polarity of its non-main variables depends on the number of universal quantifiers under which they are nested. Indeed, polarities are preserved by applications of Corollary \ref{lemma:flattening_skeleton_formulas} on existential quantifiers and inverse splitting rules, and are reversed by applications of Corollary \ref{lemma:flattening_skeleton_formulas} on universal quantifiers (see Table \ref{table:ackermann_restr_quant}). Thus, in the end nominals (resp.\ conominals) must occur in negative (resp.\ positive) polarity.
\end{proof}

\input{examples04z04_compactionsuccedent}

\subsubsection{Compaction of the antecedent}
Given Kracht formula
\[
\forall\nomj\forall\cnomm\forall\overline\nomh\forall\overline\cnomo\forall^R\overline\nomi,\overline\cnomn(\eta(\nomj,\cnomm) \Rightarrow \theta_1(\purew_1) \metaor \cdots \metaor \theta_n(\purew_n) ),
\]
after applying Algorithm \ref{algo:elimination_of_non-inherently_universals} to $\theta_1,\ldots,\theta_n$, it becomes
\begin{equation}
\label{eqn:after_elimination_for_gammadeltamv}
\forall\nomj\forall\cnomm\forall\overline\nomh\forall\overline\cnomo\forall^R\overline\nomi,\overline\cnomn\left(\eta(\nomj,\cnomm) \Rightarrow {\bigmetaor_i {\lambda(\cnoml_i) \leq \delta_i} \metaor \bigmetaor_j {\gamma_j \leq \kappa(\nomk_j)} } \right),
\end{equation}
where each $\nomk$ (resp.\ $\cnoml$) is either an alias variable, or is bound by some \emph{type 1} quantifier\footnote{We know that variables in $\overline\nomk$ and $\overline\cnoml$ occur as described in \eqref{eqn:after_elimination_for_gammadeltamv} because the main variable of each disjunct in a Kracht-DLE formula is of the form $\dneg\pureu$ with $\pureu \in \mathrm{NOM}\cup\mathrm{CONOM}$.}. Furthermore, as the input formula is refined, each of alias and each \emph{type 1} variable occurs at least once in the consequent by Lemma \ref{lemma:kracht_aliases_in_succedent}. Let us abbreviate 
\[
\mathrm{SUCC} \coloneqq {\bigmetaor_i {\lambda(\cnoml_i) \leq \delta_i} \metaor \bigmetaor_j {\gamma_j \leq \kappa(\nomk_j)}}.
\]

Each restricted quantifier binds some nominals and conominals $\overline\pureu$, using some $f$ or $g$ operator; in each case, the quantifier comes equipped with a restricting inequality in the antecedent, which has shape $\nomk \leq f(\overline\pureu)$ in the first case, $g(\overline\pureu) \leq \cnoml $ in the second case, for some nominal $\nomk$ and conominal $\cnoml$. By currying, for any formula $\sigma$,
\[
\begin{array}{rl}
& (\forall \nomk \rhd_f \overline\pureu)(\sigma \Rightarrow \mathrm{SUCC})\\
\mbox{i.e.} & \forall \nomk( \nomk \leq f(\overline\pureu) \Rightarrow (\sigma \Rightarrow \mathrm{SUCC})) \\
\mbox{iff} & \forall \nomk((\nomk \leq f(\overline\pureu) \metaand \sigma) \Rightarrow \mathrm{SUCC}),
\end{array}
\]
and similarly for quantifiers using $g$ operators. Let us apply this procedure exhaustively, so to rewrite the antecedent of \eqref{eqn:after_elimination_for_gammadeltamv} by conjoining it with all the restricting inequalities of \emph{type 2} quantifiers and of \emph{type 1} quantifiers restricted by variables bound by \emph{type 2} quantifiers. The next lemma shows that all \emph{type 2} quantifiers can be eliminated by proceeding from the rightmost to the leftmost via Ackermann rule. 

\begin{lemma}
After exhaustively currying, the antecedent of \eqref{eqn:after_elimination_for_gammadeltamv} is in Ackermann shape for the elimination of the rightmost restricted quantifier, and, after the elimination, it becomes in Ackermann shape for the elimination of the successive quantifiers.
\end{lemma}
\begin{proof}
After currying, the variables nominals $\nomi$ and conominals $\cnomn$ bound by \emph{type 2} quantifiers can either occur in the antecedent in inequalities $\nomi \leq \nomh$ (resp.\ $\cnomo \leq \cnomn$) for some alias variable $\nomh$ (resp.\ $\cnomo$), or in restricting inequalities. Notice that each $\nomi$ (resp.\ $\cnomn$) occurs negatively (resp.\ positively) only in the restricting inequality of the quantifier that binds it (let us call it $I_1$), and occurs in the opposite polarity in any other restricting inequality where it is a restrictor. Therefore, we can merge via inverse splitting all the inequalities involving aliases and the restricting inequalities where $\nomi$ (resp.\ $\cnomn$) occur as restrictor, thus obtaining an inequality $I_2$.
When eliminating the rightmost quantifier introducing variables $\overline\nomk$ and $\overline\cnoml$, we only have to consider the restricting inequality of the quantifier $I_1$, and the inequalities (one for each variable) where nominals in $\overline\nomk$ (resp.\ conominals in $\overline\cnoml$) occur in positively (resp.\ negatively), thus, as they cannot occur in the consequent, $\overline\nomk$ and $\overline\cnoml$ have the correct polarity to apply an Ackermann elimination on the quantifier by Corollary \ref{lemma:flattening_skeleton_formulas} (cf. Table \ref{table:ackermann_restr_quant}). The variable $\pureu$ on display in the resulting inequality is the restrictor of $I_1$, and, if it is a nominal (resp.\ conominal) it occurs on the left (resp.\ right) hand side of the inequality; furthermore it occurs only once in the inequality. The variable $\pureu$ can either be a pivotal variable or a variable bound by another \emph{type 2} quantifier. In the latter case, $\pureu$ will be eliminated in a later stage by repeating the same procedure. At that stage, this inequality will be merged via inverse splitting with the ones where $\pureu$ is on display on the left (resp.\ right) hand side if it is a nominal (resp.\ conominal).
\end{proof}

After eliminating all the variables bound by \emph{type 2} quantifiers, the shape of the antecedent reduces to the inequality $\nomj \nleq \cnomm$ in conjunction with inequalities of the form $\nomj \leq \theta$ or $\eta \leq \cnomm$, and, moreover, the remaining \emph{type 1} quantifiers can only be restricted by $\nomj$ and $\cnomm$. Hence, by expanding these remaining quantifiers, exhaustively currying, and applying inverse splitting, the antecedent equivalently reduces to the following conjunction of inequalities·:
\begin{equation}
\nomj \leq \theta_1 \wedge \cdots \wedge \theta_n \ \metaand \ \eta_1 \vee \cdots \vee \eta_m \leq \cnomm \metaand  \ \nomj \nleq \cnomm.
\end{equation}

\begin{lemma}
\label{lemma:properties_of_the_skeleton}
After the elimination of \emph{type 2} restricted quantifiers and the expansion of \emph{type 1} quantifiers, the antecedent has form
\begin{equation}
\nomj \leq \bigwedge_{i=1}^n\theta_{i} \ \metaand \ \bigvee_{i=1}^m\eta_i \leq \cnomm \metaand  \ \nomj \nleq \cnomm,
\end{equation}
where $+\bigwedge_{i=1}^n\theta_{i}$ and $-\bigwedge_{i=1}^m\eta_i $ are pure scattered Skeleton formulas where $\nomj$ and $\cnomm$ do not occur, and any nominal (resp.\ conominal) occurs in positive (resp.\ negative) polarity.
\end{lemma}
\begin{proof}
It is sufficient to show that every $+\theta_i$ and $-\eta_i$ is made of Skeleton nodes and that each variable occurs only once, since $\nomj$ and $\cnomm$ cannot clearly occur there as they are not in $\overline\nomh$ or $\overline\cnomo$ and they are not even restricted variables. The conjuncts that come from the \emph{type 1} restricted quantifiers clearly satisfy the statement, hence it remains to show that the algorithm for the elimination of \emph{type 2} quantifiers produces conjuncts with the same property. We proceed by induction on the number of iterations. Let us consider the case in which we eliminate a quantifier of the kind $(\forall \overline\nomi,\overline\cnomn \rhd_f \nomk)$ for some $f \in \mathcal{F}$, as the case where the restricting operator is in $\mathcal{G}$ is treated similarly. Before the corresponding Ackermann elimination is performed, in the antecedent there is an inequality $\nomk \leq f(\overline\nomi,\overline\cnomn)$ and inequalities $\overline{\nomi \leq \varphi}$ and $\overline{\psi \leq \cnomn}$, where each formula in $\overline\varphi$ (resp.\ $\overline\psi$) is a pure scattered positive (resp.\ negative) skeleton formula by inductive hypothesis, or by the definition of Kracht-DLE formula in the base case, as there would be just restricting inequalities in the antecedent and each \emph{alias} variable occurs only once in the antecedent. After the application of the Ackermann lemma, we obtain the restricting inequality
\[
\nomh \leq f(\overline\varphi,\overline\psi).
\]
Clearly in the signed generation $+f(\overline\varphi,\overline\psi)$ (we propagate $+$ as the inequality is in the antecedent), every nominal (resp.\ conominal) occurs in positive (resp.\ negative) position.
\end{proof}

\begin{remark}
\label{remark:variables_in_the_skeleton}
The variables occurring $+\bigwedge_{i=1}^n\theta_{i}$ and $-\bigwedge_{i=1}^m\eta_i $ are exactly all the ones in $\overline\nomh$, $\overline\cnomo$ and the ones bound by \emph{type 1} quantifiers. The former variables are indeed captured because each one of them occurs at least (exactly) once in the antecedent, whilst the latter variables are clearly captured by writing the expansion of the quantifier.
\end{remark}

\input{examples04z04_compactionantecedent}

\subsubsection{Elimination of pivotal variables}

After the elimination of \emph{type 2} quantifiers, the contrapositive of the obtained formula is
\begin{equation}
\label{eqn:after_type_2_quant_elimination}
\forall\nomj\forall\cnomm\forall\overline\nomh\forall\overline\cnomo\forall
\overline\nomi' \forall \overline\cnomn'\left(
\overline{\nomk \leq \gamma} \metaand 
\overline{\delta \leq \cnoml} \Rightarrow
\left(\nomj \leq \bigwedge_{i=1}^n\theta_{i} \ \metaand \ \bigvee_{i=1}^m\eta_i \leq \cnomm \Rightarrow  \ \nomj \leq \cnomm \right) \right),
\end{equation}
where $\overline\nomi'$ and $\overline\cnomn'$ are the variables originally introduced by \emph{type 1} restricted quantifiers, and the variables in $\overline\nomk$ (resp.\ $\overline\cnoml$) are in $\mathrm{NOM}$ (resp.\ $\mathrm{CONOM}$)\footnote{It is important to note that they are not in $\nomset$ and $\cnomset$, so no $\kappa$ nor $\lambda$ is hidden in the notation there.} among the ones in $\overline\nomi'$ and $\overline\nomh$ (resp.\ $\overline\cnomn'$ and $\overline\cnomo$). By applying universal Ackermann elimination on $\nomj$ and $\cnomm$, and by putting $\varphi \coloneqq \bigwedge_{i=1}^n\theta_{i} $ and $\psi \coloneqq \bigwedge_{i=1}^m\eta_i$, the formula above is equivalent to 
\begin{equation}
\label{eqn:target_alba_regained}
\forall\overline\nomh\forall\overline\cnomo\forall
\overline\nomi' \forall\overline\cnomn'\left(
\overline{\nomk \leq \gamma} \metaand 
\overline{\delta \leq \cnoml} \Rightarrow
 \varphi \leq \psi
\right).
\end{equation}

We can assume that each variable in $\overline{\nomk}$ (resp.\ $\overline{\cnoml}$) is different, since if it occurs in more than one inequality, these two inequalities can be merged via inverse splitting.

\subsubsection{To very simple Sahlqvist}

To simplify notation, let $\overline\nomi$ (resp.\ $\overline\cnomn$) denote all the nominals (resp.\ conominals) occurring in $\overline\gamma$ and $\overline\delta$. 
The formula \eqref{eqn:target_alba_regained} is thus equivalent up to renaming to:
\begin{equation}
\label{eqn:almost_final_after_inverse_correspondence}
\forall\overline\nomk \forall\overline\cnoml \forall\overline\nomi \forall\overline\cnomn \left(
 \overline{\nomk \leq \gamma} \metaand 
\overline{\delta \leq \cnoml} \Rightarrow
 \varphi \leq \psi
\right).
\end{equation}
By Lemma \ref{lemma:properties_of_the_skeleton}, we know that the inequality $\varphi \leq \psi$ is a scattered Skeleton inequality containing every variable quantified in the prefix. Furthermore, each nominal in $\overline\nomi$ and $\overline\nomk$ occurs in positive polarity in it, and each conominal in $\overline\cnomn$ and $\overline\cnoml$ occurs in negative polarity; hence we are in Ackermann shape w.r.t.\ the elimination of each $\nomk$ and $\cnoml$. Therefore \eqref{eqn:almost_final_after_inverse_correspondence} is equivalent to:
\begin{equation}
\label{eqn:almost_final_after_inverse_correspondence_2}
\forall\overline\nomi \forall\overline\cnomn \left(
 \varphi [ \overline\gamma/\overline\nomk, \overline\delta/\overline\cnoml] \leq \psi[ \overline\gamma/\overline\nomk, \overline\delta/\overline\cnoml]
\right).
\end{equation}
For each (co)nominal in $\nomi$ (resp.\ in $\cnomn$) we introduce a new variable $p_\nomi$ (resp.\ $q_\cnomn$). Let 
\[
\begin{array}{rcr}
\varphi' & \coloneqq & \left( \varphi
 \left[
 \overline\gamma/\overline\nomk, \overline\delta/\overline\cnoml \right ] \right)
 \left[ \overline{p_\nomi} / \overline\nomi, \overline{q_\cnomn}/\overline\cnomn
\right]
\phantom{,} \\[1mm]
\psi' & \coloneqq & \left( \psi
 \left[
 \overline\gamma/\overline\nomk, \overline\delta/\overline\cnoml \right ] \right)
 \left[ \overline{p_\nomi} / \overline\nomi, \overline{q_\cnomn}/\overline\cnomn
\right]. \\
\end{array}
\] 
By Lemma \ref{lemma:polarities_in_gammadelta},  nominals (resp.\ conominals) in each $+\gamma$ in $\overline\gamma$ and $-\delta$ in $\overline\delta$ occur in negative (resp.\ positive) polarity; hence every $+\gamma$ and $-\delta$ is an $\varepsilon^\partial$-uniform subtree in $+\varphi'$ and $-\psi'$, where $\varepsilon$ is the order type on $\overline p_\nomi$ and $\overline q_\cnomn$ such that $\varepsilon(p_\nomi) = 1$ and $\varepsilon(q_\cnomn) = \partial$. 

Hence, $\varphi'\leq \psi'$ is a scattered very simple $\varepsilon$-Sahlqvist inequality in $\langresidual$, and, moreover, ALBA reduces it to \eqref{eqn:almost_final_after_inverse_correspondence}, as shown below 
\[
\hspace{-1.8mm}
\begin{array}{rl}
& \forall \overline p_\nomi \forall \overline q_\cnomn \left(\varphi' \leq \psi' \right) \\
\mbox{iff} & \forall \overline p_\nomi \forall q_\cnomn \forall\overline\nomj \forall\overline\cnomm \forall\overline\nomj' \forall\overline\cnomm' \Big( \overline\nomj \leq \overline\gamma[\overline p_\nomi /\overline\nomi, \overline q_\cnomn / \overline\cnomn] \metaand \overline\delta[\overline p_\nomi /\overline\nomi, \overline q_\cnomn / \overline\cnomn] \leq \overline\cnomm \metaand \overline\nomj' \leq \overline p_\nomi \metaand \overline p_\cnomn \leq \cnomm' \ \ \\ 
& \hfill \Rightarrow \varphi[\overline\nomj / \overline\nomk, \overline\cnomm/\overline\cnoml] \leq \psi[\overline\nomj / \overline\nomk, \overline\cnomm/\overline\cnoml]  \Big)  \\
\mbox{iff} & \forall\overline\nomj \forall\overline\cnomm \forall\overline\nomj' \forall\overline\cnomm' \Big( \overline\nomj \leq \overline\gamma[\overline \nomj' /\overline\nomi, \overline \cnomm' / \overline\cnomn] \metaand \overline\delta[\overline \nomj' /\overline\nomi, \overline \cnomm' / \overline\cnomn] \leq \overline\cnomm \Rightarrow \varphi[\overline\nomj / \overline\nomk, \overline\cnomm/\overline\cnoml] \leq \psi[\overline\nomj / \overline\nomk, \overline\cnomm/\overline\cnoml]  \Big)  \\
\end{array}
\]
From the above discussion, the main result follows.
\begin{theorem}
\label{thm:krachttoverysimplesahlqvist}
Every (refined) Kracht $\langmeta$ formula $\zeta$ can be effectively associated with a scattered very simple Sahlqvist inequality in $\langresidual$ whose first order correspondent is $\zeta$.
\end{theorem}

\input{examples04z04_final}

\subsubsection{Kracht to inductive shape}
Theorem \ref{thm:krachttoverysimplesahlqvist} shows how each Kracht-DLE formula is the correspondent of some very simple Sahlqvist formula in $\langresidual$. Thanks to Proposition \ref{prop:cryptotoinductive} and Lemma \ref{lemma:inductivetocrypto}, in order to target inductive formulas in $\langbase$ it is necessary to further restrict the class of Kracht-DLE formulas in order to target crypto inductive $\langbase$-inequalities (cf.\ Definition \ref{def:Crypto:inductive}). To do so, it is sufficient to note that the only part to check is that the Kracht-DLE disjuncts that start with some with some restricted quantifier using an operator in $\langresidual$ produce (through Algorithm \ref{algo:elimination_of_non-inherently_universals}) inequalities whose formulas (not in main position) have an $(\varepsilon,\Omega)$-unpackable (see Definition \ref{def:epsomega:unpackable}) signed generation tree, for some order type $\varepsilon$ and order $\Omega$ on the variables. To do so, it is sufficient to note that the operators in such formulas are exactly the ones of the restricted quantifiers in the disjuncts in their same position before Algorithm \ref{algo:elimination_of_non-inherently_universals} is executed; hence it is sufficient to use Definition \ref{def:epsomega:unpackable} considering the operators of the restricted quantifiers.

%% file: examples04z04_compactionsuccedent.tex
\begin{example}
\label{eg:goranko2}
Let us apply Algorithm \ref{algo:elimination_of_non-inherently_universals} to the consequent of the formulas in Example \ref{eg:goranko}. The first one is
\[
 (\exists   \nomi_{1}   \rhd_\Diamond  \lambda(\cnomo_1)  )  (\forall   \cnomn_{1}   \rhd_\Box  \kappa(\nomi_{1})  )  (\forall   \cnomn_{2}   \rhd_\Box  \cnomn_{1} )  (\exists   \nomi_{2}   \rhd_\Diamondblack \lambda(\cnomn_{2})  ) ( \nomi_{2}  \leq \Diamond  \nomh_{1}  \ \&\  \nomi_{2}  \leq \Diamondblack  \nomh_{2})
\]
The innermost Kracht disjunct is a conjunction sharing the same main variable, thus we can compact it into one inequality and then proceed applying Ackermann eliminations:
\begin{center}
\begin{tabular}{rl}
&  $(\exists   \nomi_{1}   \rhd_\Diamond  \lambda(\cnomo_1)  )  (\forall   \cnomn_{1}   \rhd_\Box  \kappa(\nomi_{1})  )  (\forall   \cnomn_{2}   \rhd_\Box  \cnomn_{1} )  (\exists   \nomi_{2}   \rhd_\Diamondblack \lambda(\cnomn_{2})  ) ( \nomi_{2}  \leq \Diamond  \nomh_{1}  \wedge \Diamondblack  \nomh_{2})$ \\
iff & $(\exists   \nomi_{1}   \rhd_\Diamond  \lambda(\cnomo_1)  )  (\forall   \cnomn_{1}   \rhd_\Box  \kappa(\nomi_{1})  )  (\forall   \cnomn_{2}   \rhd_\Box  \cnomn_{1} )  ( \lambda(\cnomn_2)  \leq \Diamondblack(\Diamond  \nomh_{1}  \wedge \Diamondblack  \nomh_{2}))$ \\
iff & $(\exists   \nomi_{1}   \rhd_\Diamond  \lambda(\cnomo_1)  )  (\forall   \cnomn_{1}   \rhd_\Box  \kappa(\nomi_{1})  )  ( \lambda(\cnomn_1)  \leq \Box\Diamondblack(\Diamond  \nomh_{1}  \wedge \Diamondblack  \nomh_{2}))$ \\
iff & $(\exists   \nomi_{1}   \rhd_\Diamond  \lambda(\cnomo_1)  ) ( \nomi_1  \leq \Box\Box\Diamondblack(\Diamond  \nomh_{1}  \wedge \Diamondblack  \nomh_{2}))$\\
iff & $ \lambda(\cnomo_1)  \leq \Diamond\Box\Box\Diamondblack(\Diamond  \nomh_{1}  \wedge \Diamondblack  \nomh_{2})$.
\end{tabular}
\end{center}

\end{example}

\begin{example}
\label{eg:secondgoranko2}
The consequent of the formula in Example \ref{eg:secondgoranko} is
\[
\lambda(\cnomo_1) \leq \nomj_2 \metaor (\exists \nomi_2 \rhd \lambda(\cnomn_1))(\forall \cnomn_2 \rhd \kappa(\nomi_2))(\lambda(\cnomn_2) \leq \Diamond\nomj_1).
\]
The first Kracht-DLE disjunct is already flat, in the second Kracht-DLE disjunct the algorithm yields
\[
\begin{array}{rl}
& (\exists \nomi_2 \rhd_\Diamond \lambda(\cnomn_1))(\forall \cnomn_2 \rhd_\Box \kappa( \nomi_2))(\lambda(\cnomn_2) \leq \Diamond\nomj_1) \\
\mbox{iff} & (\exists \nomi_2 \rhd_\Diamond \lambda(\cnomn_1))(\nomi_2 \leq \Box\Diamond\nomj_1)\\
\mbox{iff} & \lambda(\cnomn_1) \leq \Diamond\Box\Diamond\nomj_1\, .
\end{array}
\]
\end{example}

\begin{example}
\label{eg:full_lambek2}
Let us apply the compaction of the consequent to the formula in Example \ref{eg:full_lambek}, namely
\[
\begin{array}{l}
\forall \nomj \forall \cnomm
\forall \nomh_1 \forall \nomh_2
(\forall \cnomn_1,\nomi_1 \rhd_\slash \cnomm)
(\forall \nomi_2,\cnomn_2 \rhd_\backslash \cnomn_1)[
\nomi_1 \leq \nomh_1 \metaand
\nomj \leq \nomh_2 \metaand
\nomj \nleq \cnomm
\Rightarrow \hspace{1.7cm} \\
\hfill
(\exists \nomi_3,\nomi_4 \rhd_\circ \lambda(\cnomn_2))
(
\nomi_3 \leq \nomi_2
\metaand
(\forall \cnomn_5,\nomi_5 \rhd_\slash \kappa(\nomi_4))
(
\lambda(\cnomn_5) \leq \nomi_2 \circ \nomh_2
\metaand
\nomi_5 \leq \nomh_2
)
)
].
\end{array}
\]
There is only one Kracht disjunct in the consequent, the algorithm yields
\[
\begin{array}{l}
(\exists \nomi_3,\nomi_4 \rhd_\circ \lambda(\cnomn_2))
(
\nomi_3 \leq \nomi_2
\metaand
(\forall \cnomn_5,\nomi_5 \rhd_\slash \kappa(\nomi_4))
(
\lambda(\cnomn_5) \leq \nomi_2 \circ \nomh_2
\metaand
\nomi_5 \leq \nomh_2
)
)\\
(\exists \nomi_3,\nomi_4 \rhd_\circ \lambda(\cnomn_2))
(
\nomi_3 \leq \nomi_2
\metaand
\nomi_4 \leq (\nomi_2 \circ \nomh_2) \slash \nomh_2
) \\
\lambda(\cnomn_2)
\leq
\nomi_2 \circ ((\nomi_2 \circ \nomh_2) \slash \nomh_2)
).
\end{array}
\]
\end{example}

%% file: examples04z04_compactionantecedent.tex
\begin{example}
\label{eg:goranko3}
When treating the antecedent of the first formula in Example \ref{eg:goranko2}, we do not have inherently universal restricted quantifiers; hence this step consists in a straightforward application of the inverse splitting rule
\[
\begin{array}{rl}
&\forall  \nomj \forall  \cnomm \forall  \nomh_{1}  \forall  \nomh_{2}  \forall \cnomo_1  [ \nomj  \leq  \nomh_{1}  \ \&\  \nomj  \leq  \nomh_{2} \metaand \cnomo_1 \leq \cnomm  \ \&\  \nomj \nleq \cnomm \Rightarrow \lambda(\cnomo_1)  \leq \Box\Box\Diamondblack(\Diamond  \nomh_{1}  \wedge \Diamondblack  \nomh_{2})] \\
&\forall  \nomj \forall  \cnomm \forall  \nomh_{1}  \forall  \nomh_{2}  \forall \cnomo_1  [ \nomj  \leq  \nomh_{1} \wedge \nomh_{2} \metaand \cnomo_1 \leq \cnomm  \ \&\  \nomj \nleq \cnomm \Rightarrow \lambda(\cnomo_1)  \leq \Box\Box\Diamondblack(\Diamond  \nomh_{1}  \wedge \Diamondblack  \nomh_{2})] \\
\end{array}
\]
\end{example}

\begin{example}
\label{eg:secondgoranko3}
As for the formula in Example \ref{eg:secondgoranko2} , namely
\[
\begin{array}{rl}
& \forall \nomj \forall \cnomm \forall \nomh_1\forall \nomh_2  \forall \cnomo_1(\forall \nomi_1 \rhd_\Diamond \nomj)(\forall \cnomn_1 \rhd_\Box \cnomm) [ \nomi_1 \leq \nomh_1 \metaand \nomi_1 \leq \nomh_2 \metaand \cnomo_1 \leq \cnomm \metaand \nomj \nleq \cnomm \Rightarrow \ \ \ \\ 
& \hfill \lambda(\cnomo_1) \leq \nomh_2 \metaor  \lambda(\cnomn)_1 \leq \Diamond\Box\Diamond\nomh_1], \\
\end{array}
\]
the quantifier $(\forall \cnomn_1 \rhd \cnomm)$ is  of  \emph{type 1}, while $(\forall \nomi_1 \rhd \nomj)$ is of \emph{type 2}. We start by eliminating the latter and then we merge the inequalities of the antecedent with the one of the restricted quantifier of \emph{type 1}.
\[
\begin{array}{rl}
& \forall \nomj \forall \cnomm \forall \nomh_1\forall \nomh_2  \forall \cnomo_1\forall \nomi_1 \forall \cnomn_1 [ \nomi_1 \leq \nomh_1 \metaand \nomi_1 \leq \nomh_2 \metaand \nomj \leq \Diamond \nomi_1 \metaand \cnomo_1 \leq \cnomm \metaand \Box\cnomn_1\leq\cnomm \metaand \nomj \nleq \cnomm \Rightarrow \ \ \ \\ 
& \hfill \lambda(\cnomo_1) \leq \nomh_2 \metaor  \lambda(\cnomn_1) \leq \Diamond\Box\Diamond\nomh_1] \\
\mbox{iff} & \forall \nomj \forall \cnomm \forall \nomh_1\forall \nomh_2  \forall \cnomo_1\forall \nomi_1 \forall \cnomn_1 [ \nomi_1 \leq \nomh_1 \wedge \nomh_2 \metaand \nomj \leq \Diamond \nomi_1 \metaand \cnomo_1 \leq \cnomm \metaand \Box\cnomn_1\leq\cnomm \metaand \nomj \nleq \cnomm \Rightarrow \\ 
& \hfill \lambda(\cnomo_1) \leq \nomh_2 \metaor  \lambda(\cnomn_1) \leq \Diamond\Box\Diamond\nomh_1] \\
\mbox{iff} & \forall \nomj \forall \cnomm \forall \nomh_1\forall \nomh_2  \forall \cnomo_1\forall \cnomn_1 [ \nomj \leq \Diamond(\nomh_1 \wedge \nomh_2) \metaand \cnomo_1 \leq \cnomm \metaand \Box\cnomn_1\leq\cnomm \metaand \nomj \nleq \cnomm \Rightarrow \\ 
& \hfill \lambda(\cnomo_1) \leq \nomh_2 \metaor  \lambda(\cnomn_1) \leq \Diamond\Box\Diamond\nomh_1] \\
\mbox{iff} & \forall \nomj \forall \cnomm \forall \nomh_1\forall \nomh_2  \forall \cnomo_1\forall \cnomn_1 [ \nomj \leq \Diamond(\nomh_1 \wedge \nomh_2) \metaand \cnomo_1 \vee \Box\cnomn_1\leq\cnomm \metaand \nomj \nleq \cnomm \Rightarrow \\ 
& \hfill \lambda(\cnomo_1) \leq \nomh_2 \metaor  \lambda(\cnomn_1) \leq \Diamond\Box\Diamond\nomh_1]
\end{array}
\]
\end{example}

\begin{example}
\label{eg:full_lambek3}
The formula obtained in Example \ref{eg:full_lambek2} after the compaction of the consequent is the following one
\[
\begin{array}{l}
\forall \nomj \forall \cnomm
\forall \nomh_1 \forall \nomh_2
(\forall \cnomn_1,\nomi_1 \rhd_\slash \cnomm)
(\forall \nomi_2,\cnomn_2 \rhd_\backslash \cnomn_1)[
\nomi_1 \leq \nomh_1 \metaand
\nomj \leq \nomh_2 \metaand
\nomj \nleq \cnomm
\Rightarrow \hspace{1.7cm} \\
\hfill
\lambda(\cnomn_2)
\leq
\nomi_2 \circ ((\nomi_2 \circ \nomh_2) \slash \nomh_2)
)
].
\end{array}
\]
The quantifier $(\forall \cnomn_1,\nomi_1 \rhd_\slash \cnomm)$ is of \emph{type 2}, while $(\forall \nomi_2,\cnomn_2 \rhd_\backslash \cnomn_1)$ is of \emph{type 1}. Since the latter depends on the former, we start by expanding the restricted quantifier notation for both obtaining
\[
\begin{array}{l}
\forall \nomj \forall \cnomm
\forall \nomh_1 \forall \nomh_2
\forall \cnomn_1 \forall \nomi_1
\forall \nomi_2 \forall \cnomn_2[
\nomi_1 \leq \nomh_1 \metaand
\nomj \leq \nomh_2 \metaand
\cnomn_1 \slash \nomi_1 \leq \cnomm \metaand
\nomi_2 \backslash \cnomn_2 \leq \cnomn_1 \metaand
\nomj \nleq \cnomm
\Rightarrow \\
\hfill
\lambda(\cnomn_2)
\leq
\nomi_2 \circ ((\nomi_2 \circ \nomh_2) \slash \nomh_2)
)
].
\end{array}
\]
After the elimination of $\cnomn_1$ and $\nomi_1$, the formula becomes
\[
\forall \nomj \forall \cnomm
\forall \nomh_1 \forall \nomh_2
\forall \nomi_2 \forall \cnomn_2[
\nomj \leq \nomh_2 \metaand
(\nomi_2 \backslash \cnomn_2) \slash \nomh_1 \leq \cnomm \metaand
\nomj \nleq \cnomm
\Rightarrow 
\lambda(\cnomn_2)
\leq
\nomi_2 \circ ((\nomi_2 \circ \nomh_2) \slash \nomh_2)
)
].
\]
\end{example}

%% file: examples04z04_final.tex
\begin{example}
\label{eg:goranko4}
In Example \ref{eg:goranko3} we had
\[
\forall  \nomj \forall  \cnomm \forall  \nomh_{1}  \forall  \nomh_{2} \forall \cnomo_1   [ \nomj  \leq  \nomh_{1} \wedge \nomh_{2} \metaand \cnomo_1 \leq \cnomm \ \&\  \nomj \nleq \cnomm_1 \Rightarrow \lambda(\cnomo_1)  \leq \Box\Box\Diamondblack(\Diamond  \nomh_{1}  \wedge \Diamondblack  \nomh_{2})]
\]
After the contrapositive step it becomes
\[
\forall  \nomh_{1}  \forall  \nomh_{2}  \forall \cnomo_1  [ \Box\Box\Diamondblack(\Diamond  \nomh_{1}  \wedge \Diamondblack  \nomh_{2}) \leq \cnomo_1 \Rightarrow  \nomh_{1} \wedge \nomh_{2} \leq \cnomo_1];
\]
which, by the previous discussion, is equivalent to the very simple Sahlqvist formula
\[
\forall p_{h_1} \forall p_{h_2}  [   p_{h_1} \wedge p_{h_2} \leq \Box\Box\Diamondblack(\Diamond  p_{h_1}  \wedge \Diamondblack  p_{h_2})].
\]
\end{example}

\begin{example}
\label{eg:secondgoranko4}
Taking the formula in Example \ref{eg:secondgoranko3}, i.e. 
\[
\begin{array}{rl}
 & \forall \nomj \forall \cnomm \forall \nomh_1\forall \nomh_2  \forall \cnomo_1\forall \cnomn_1 [ \nomj \leq \Diamond(\nomh_1 \wedge \nomh_2) \metaand \cnomo_1 \vee \Box\cnomn_1\leq\cnomm \metaand \nomj \nleq \cnomm \Rightarrow \ \ \ \  \\ 
& \hfill \lambda(\cnomo_1) \leq \nomh_2 \metaor  \lambda(\cnomn_1) \leq \Diamond\Box\Diamond\nomh_1],
\end{array}
\]
after the contrapositive step we obtain
\[
 \forall \nomh_1\forall \nomh_2  \forall \cnomo_1\forall \cnomn_1[ \nomh_2 \leq \cnomo_1 \metaand \Diamond\Box\Diamond\nomh_1 \leq \cnomn_1 \Rightarrow \Diamond(\nomh_1 \wedge \nomh_2) \leq \cnomo_1 \vee \Box\cnomn_1],
\]
which in turn is equivalent to the very simple Sahlqvist
\[
\forall p_{h_1} \forall q_{o_1}[\Diamond(p_{h_1} \wedge q_{o_1}) \leq q_{o_1} \vee \Box\Diamond\Box\Diamond p_{h_1}]
\]
\end{example}

\begin{example}
\label{eg:full_lambek4}
Continuing from Example \ref{eg:full_lambek3}, after the contrapositive step, the formula
\[
\forall \nomj \forall \cnomm
\forall \nomh_1 \forall \nomh_2
\forall \nomi_2 \forall \cnomn_2[
\nomj \leq \nomh_2 \metaand
(\nomi_2 \backslash \cnomn_2) \slash \nomh_1 \leq \cnomm \metaand
\nomj \nleq \cnomm
\Rightarrow 
\lambda(\cnomn_2)
\leq
\nomi_2 \circ ((\nomi_2 \circ \nomh_2) \slash \nomh_2)
)
],
\]
becomes
\[
\forall \nomh_1 \forall \nomh_2
\forall \nomi_2 \forall \cnomn_2[
\nomi_2 \circ ((\nomi_2 \circ \nomh_2) \slash \nomh_2)
\leq
\cnomn_2
\Rightarrow 
\nomh_2
\leq
(\nomi_2 \backslash \cnomn_2) \slash \nomh_1
)
],
\]
which is equivalent to the very simple Sahlqvist
\[
\forall p_{i_2} \forall p_{h_1} \forall p_{h_2}
[
p_{h_2}
\leq
(p_{i_2} \backslash (p_{i_2} \circ ((p_{i_2} \circ p_{h_2}) \slash p_{h_2}))) \slash p_{h_1}
].
\]
\end{example}

%% file: 05conclusions.tex
\paragraph{Our contribution.} This paper presents two main contributions that naturally extend the results in \cite{inverse_correspondence_tense} to the DLE setting. Firstly, it shows a characterization of the fragment of very simple Sahlqvist formulae in the language with all the residuals that are equivalent to some inductive formula in the language without residuals. Secondly, it shows an algorithm for inverse correspondence for the logics whose algebraic semantics are given by normal distributive lattice expansions. This allows to effectively compute axiomatizations that enforce some first order condition on their models, provided that the condition can be written in the shape of Definition \ref{def:dlekracht}.

\paragraph{Future work.} One natural development of this work consists in expanding and modifying the algorithm in order to make it work also with non-distributive lattice logics. 

%% file: appendix01albacorrespondence.tex
The present subsection reports on the rules and execution of the algorithm ALBA in the setting of $\mathcal{L}_\mathrm{DLE}$.
		
		The version of ALBA relative to $\mathcal{L}_\mathrm{DLE}$ runs as detailed in \cite{CoPa2012distr}. In a nutshell, $\mathcal{L}_\mathrm{DLE}$-inequalities are equivalently transformed into the conjunction of one or more $\mathcal{L}_\mathrm{DLE}^{*+}$ quasi-inequalities, with the aim of eliminating propositional variable occurrences via the application of Ackermann rules. We refer the reader to \cite{CoPa2012distr} for details. In what follows, we illustrate how ALBA works, while at the same time we introduce its rules. The proof of the soundness and invertibility of the general rules for the DLE-setting is similar to the one provided in \cite{CoPa2012distr,CoGhPa14}. ALBA manipulates input inequalities $\phi\leq\psi$ and proceeds in three stages:
		
		\textbf{First stage: preprocessing and first approximation.} 
		ALBA preprocesses the input inequality $\phi\leq \psi$ by performing the following steps
		exhaustively in the signed generation trees $+\phi$ and $-\psi$:
		\begin{enumerate}
			\item
			\begin{enumerate}
				\item Push down, towards variables, occurrences of $+\land$, by distributing each of them over their children nodes labelled with $+\lor$ which are not in the scope of PIA nodes;
				\item Push down, towards variables, occurrences of $-\lor$, by distributing each of them over their children nodes labelled with $-\land$ which are not in the scope of PIA nodes;
				\item Push down, towards variables, occurrences of $+f$ for any $f\in \mathcal{F}$, by distributing each such occurrence over its $i$th child node whenever the child node is labelled with $+\lor$ (resp.\ $-\land$) and is not in the scope of PIA nodes, and whenever $\epsilon_f(i)=1$ (resp.\ $\epsilon_f(i)=\partial$);				
				\item Push down, towards variables, occurrences of $-g$ for any $g\in \mathcal{G}$, by distributing each such occurrence over its $i$th child node whenever the child node is labelled with $-\land$ (resp.\ $+\lor$) and is not in the scope of PIA nodes, and whenever $\epsilon_g(i)=1$ (resp.\ $\epsilon_g(i)=\partial$).
				
			\end{enumerate}
			\item Apply the splitting rules:
			\[\infer{\alpha\leq\beta\ \ \ \alpha\leq\gamma}{\alpha\leq\beta\land\gamma}
			\qquad
			\infer{\alpha\leq\gamma\ \ \ \beta\leq\gamma}{\alpha\lor\beta\leq\gamma}
			\]
			\item Apply the monotone and antitone variable-elimination rules:
			\[
			\infer{\alpha(\perp)\leq\beta(\perp)}{\alpha(p)\leq\beta(p)}
			\qquad
			\infer{\beta(\top)\leq\alpha(\top)}{\beta(p)\leq\alpha(p)}
			\]
			for $\beta(p)$ positive in $p$ and $\alpha(p)$ negative in $p$.
		\end{enumerate}

		Let $\mathsf{Preprocess}(\phi\leq\psi)$ be the finite set $\{\phi_i\leq\psi_i\mid 1\leq i\leq n\}$ of inequalities obtained after the exhaustive application of the previous rules. We proceed separately on each of them, and hence, in what follows, we focus only on one element $\phi_i\leq\psi_i$ in $\mathsf{Preprocess}(\phi\leq\psi)$, and we drop the subscript. Next, the following {\em first approximation rule} is applied {\em only once} to every inequality in $\mathsf{Preprocess}(\phi\leq\psi)$:
		$$\infer{\nomi_0\leq\phi\ \ \ \psi\leq \cnomm_0}{\phi\leq\psi}
		$$
		Here, $\nomi_0$ and $\cnomm_0$ are a nominal and a conominal respectively. The first-approximation
		step gives rise to systems of inequalities $\{\nomi_0\leq\phi_i, \psi_i\leq \cnomm_0\}$ for each inequality in $\mathsf{Preprocess}(\phi\leq\psi)$. Each such system is called an {\em initial
			system}, and is now passed on to the reduction-elimination cycle.
		
		\textbf{Second stage: reduction-elimination cycle.} The goal of the reduction-elimination cycle is to eliminate all propositional variables from the systems
		received from the preprocessing phase. The elimination of each variable is effected by an
		application of one of the Ackermann rules given below. In order to apply an Ackermann rule, the
		system must have a specific shape. The adjunction, residuation, approximation, and splitting rules are used to transform systems into this shape. The rules of the reduction-elimination cycle, viz.\ the adjunction, residuation, approximation, splitting, and Ackermann rules, will be collectively called the {\em reduction} rules.
		
		\textbf{Residuation rules.} Here below we provide the residuation rules relative to each $f\in \mathcal{F}$ and $g\in \mathcal{G}$ of arity at least $1$: for each $1\leq h\leq n_f$ and each $1\leq k\leq n_g$:
		\begin{center}
			\begin{tabular}{cc}
				\AxiomC{$f(\psi_1,\ldots,\psi_h, \ldots, \psi_{n_f})\leq \chi$}
				\LeftLabel{($\varepsilon_f(h) = 1$)}
				\UnaryInfC{$\psi_h\leq f_h^{\sharp}(\psi_1,\ldots,\chi, \ldots, \psi_{n_f})$}
				\DisplayProof
				&
				\AxiomC{$f(\psi_1,\ldots,\psi_h, \ldots, \psi_{n_f})\leq \chi$}
				\RightLabel{($\varepsilon_f(h) = \partial$)}
				\UnaryInfC{$f_h^{\sharp}(\psi_1,\ldots,\chi, \ldots, \psi_{n_f})\leq \psi_h$}
				\DisplayProof
				\\
			\end{tabular}
		\end{center}
		
		\begin{center}
			\begin{tabular}{cc}
				\AxiomC{$\chi \leq g(\psi_1,\ldots,\psi_k, \ldots, \psi_{n_g})$}
				\LeftLabel{($\varepsilon_g(k) = \partial$)}
				\UnaryInfC{$\psi_k \leq g_k^{\flat}(\psi_1,\ldots,\chi, \ldots, \psi_{n_g})$}
				\DisplayProof
				&
				\AxiomC{$\chi \leq g(\psi_1,\ldots,\psi_k, \ldots, \psi_{n_g})$}
				\RightLabel{($\varepsilon_g(k) = 1$)}
				\UnaryInfC{$g_k^{\flat}(\psi_1,\ldots,\chi, \ldots, \psi_{n_g})\leq \psi_k$}
				\DisplayProof
				\\
			\end{tabular}
		\end{center}
		
		
		\textbf{Approximation rules.} Here below we provide the approximation rules\footnote{The version of the approximation rules given in \cite{CoPa2012distr,PaSoZh15r,CGPSZ14} is slightly different from but equivalent to that of the approximation rules reported on here. That formulation is motivated by the need of enforcing the invariance of certain topological properties for the purpose of proving the canonicity of the inequalities on which ALBA succeeds. In this context, we do not need to take these constraints into account, and hence we can take this more flexible version of the approximation rules as primitive, bearing in mind that when proving canonicity one has to take a formulation analogous to that in in \cite{CoPa2012distr,PaSoZh15r,CGPSZ14} as primitive.} relative to each $f\in \mathcal{F}$ and $g\in \mathcal{G}$ of arity at least $1$: for each $1\leq h\leq n_f$ and each $1\leq k\leq n_g$,
		
		\begin{center}
			\begin{tabular}{cc}
				\AxiomC{$\nomi\leq f(\psi_1,\ldots,\psi_h, \ldots, \psi_{n_f})$}
				\LeftLabel{$(\varepsilon_f(h) = 1)$}
				\UnaryInfC{$\nomi\leq f(\psi_1,\ldots,\nomj, \ldots, \psi_{n_f})\quad\nomj\leq \psi_h$}
				\DisplayProof
				&
				\AxiomC{$g(\psi_1,\ldots,\psi_k, \ldots, \psi_{n_g})\leq \cnomm$}
				\RightLabel{$(\varepsilon_g(k) = 1)$}
				\UnaryInfC{$g(\psi_1,\ldots,\cnomn, \ldots, \psi_{n_g})\leq \cnomm\quad\psi_k\leq\cnomn$}
				\DisplayProof
				\\
				& \\
				\AxiomC{$\nomi\leq f(\psi_1,\ldots,\psi_h, \ldots, \psi_{n_f})$}
				\LeftLabel{$(\varepsilon_f(h) = \partial)$}
				\UnaryInfC{$\nomi\leq f(\psi_1,\ldots,\cnomn, \ldots, \psi_{n_f})\quad \psi_k\leq\cnomn$}
				\DisplayProof
				&
				\AxiomC{$g(\psi_1,\ldots,\psi_k, \ldots, \psi_{n_g})\leq \cnomm$}
				\RightLabel{$(\varepsilon_g(k) = \partial)$}
				\UnaryInfC{$g(\psi_1,\ldots,\nomj, \ldots, \psi_{n_g})\leq \cnomm\quad \nomj\leq \psi_h$}
				\DisplayProof
			\end{tabular}
		\end{center}
		where the variable $\nomj$ (resp.\ $\cnomn$) is a nominal (resp.\ a conominal). 
		The nominals and conominals introduced by the approximation rules must be {\em fresh}, i.e.\ must not already occur in the system before applying the rule.
		
		
		\textbf{Ackermann rules.} These rules are the core of ALBA, since their application eliminates proposition variables. As mentioned earlier, all the preceding steps are aimed at equivalently rewriting the input system into one of a shape in which the Ackermann rules can be applied. An important feature of Ackermann rules is that they are executed on the whole set of inequalities in which a given variable occurs, and not on a single inequality.
		\begin{center}
			\AxiomC{$\bigmetaand \{ \alpha_i \leq p \mid 1 \leq i \leq n \} \metaand \bigmetaand \{ \beta_j(p)\leq \gamma_j(p) \mid 1 \leq j \leq m \} \; \Rightarrow \; \nomi \leq \cnomm$}
			\RightLabel{$(RAR)$}
			\UnaryInfC{$\bigmetaand \{ \beta_j(\bigvee_{i=1}^n \alpha_i)\leq \gamma_j(\bigvee_{i=1}^n \alpha_i) \mid 1 \leq j \leq m \} \; \Rightarrow \; \nomi \leq \cnomm$}
			\DisplayProof
		\end{center}
		where $p$ does not occur in $\alpha_1, \ldots, \alpha_n$, $\beta_{1}(p), \ldots, \beta_{m}(p)$ are positive in $p$, and $\gamma_{1}(p), \ldots, \gamma_{m}(p)$ are negative in $p$.
		
		\begin{center}
			\AxiomC{$\bigmetaand \{ p \leq \alpha_i \mid 1 \leq i \leq n \} \metaand \bigmetaand \{ \beta_j(p)\leq \gamma_j(p) \mid 1 \leq j \leq m \} \; \Rightarrow \; \nomi \leq \cnomm$}
			\RightLabel{$(LAR)$}
			\UnaryInfC{$\bigmetaand \{ \beta_j(\bigwedge_{i=1}^n \alpha_i)\leq \gamma_j(\bigwedge_{i=1}^n \alpha_i) \mid 1 \leq j \leq m \} \; \Rightarrow \; \nomi \leq \cnomm$}
			\DisplayProof
		\end{center}
		where $p$ does not occur in $\alpha_1, \ldots, \alpha_n$, $\beta_{1}(p), \ldots, \beta_{m}(p)$ are negative in $p$, and $\gamma_{1}(p), \ldots, \gamma_{m}(p)$ are positive in $p$.
		
		\textbf{Third stage: output.}
		If there was some system in the second stage from which not all occurring propositional variables could be eliminated through the application of the reduction rules, then ALBA reports failure and terminates. Else, each system $\{\nomi_0\leq\phi_i, \psi_i\leq \cnomm_0\}$ obtained from $\mathsf{Preprocess}(\varphi\leq \psi)$ has been reduced to a system, denoted $\mathsf{Reduce}(\varphi_i\leq \psi_i)$, containing no propositional variables. Let ALBA$(\varphi\leq \psi)$ be the set of quasi-inequalities \begin{center}{\Large{\&}}$[\mathsf{Reduce}(\varphi_i\leq \psi_i) ]\Rightarrow \nomi_0 \leq \cnomm_0$\end{center} for each $\varphi_i \leq \psi_i \in \mathsf{Preprocess}(\varphi\leq \psi)$.
		
		Notice that all members of ALBA$(\varphi\leq \psi)$ are free of propositional variables. ALBA returns
		\[
		\mathrm{ALBA}(\varphi\leq \psi)
		\]
		and terminates. An inequality $\varphi\leq \psi$ on which ALBA succeeds will be called an ALBA-{\em inequality}.

		The proof of the following theorem is a straightforward generalization of \cite[Theorem 8.1]{CoPa2012distr}, and hence its proof is omitted.
		\begin{thm}[Correctness]\label{albacorrect}
			If ALBA succeeds on a $\mathcal{L}_{\mathrm{DLE}}$-inequality $\varphi\leq\psi$, then for every perfect    $\mathcal{L}_{\mathrm{DLE}}$-algebra $\bbA$, $$\bbA\models\varphi\leq\psi\quad\mbox{iff}\quad\bbA\models\mathrm{ALBA}(\varphi\leq\psi).$$
		\end{thm}